\documentclass[12pt]{article}
\usepackage[utf8]{inputenc}
\usepackage{algorithm}
\usepackage{algorithmicx}
\usepackage{algpseudocode}

\usepackage[utf8]{inputenc} 
\usepackage[T1]{fontenc}    
\usepackage[english]{babel}

\usepackage{amsfonts}
\usepackage{nicefrac}       
\usepackage{microtype}      
\usepackage{lmodern}
\usepackage{amssymb,amsmath,amsthm}
\usepackage{bbm}
\usepackage{mathrsfs}
\usepackage{stmaryrd}
\usepackage{bbm,bm}
\usepackage{latexsym}
\usepackage{xcolor}
\usepackage{graphicx}
\usepackage{subfigure}
\usepackage{booktabs}
\usepackage{url}
\usepackage{hyperref}
\hypersetup{
	colorlinks=true,
	linkcolor=blue,
	filecolor=blue,
	anchorcolor=blue,
	urlcolor=cyan,
	citecolor=purple
}
\usepackage{enumerate}
\usepackage[shortlabels]{enumitem}
\usepackage{verbatim}
\usepackage{booktabs}       
\usepackage{multirow}

\usepackage{url}            
\usepackage[numbers]{natbib}

\usepackage[in]{fullpage}

\usepackage[short]{optidef}
\usepackage{algorithm}
\usepackage{thmtools}

\usepackage{tikz}
\usepackage{tcolorbox}
\usepackage[framemethod=tikz]{mdframed}

\declaretheorem{theorem}
\declaretheorem{corollary}
\declaretheorem{lemma}
\declaretheorem{proposition}

\declaretheorem{fact}

\declaretheoremstyle[qed=$\square$]{definitionwithend}
\declaretheorem{definition}
\declaretheorem{assumption}
\declaretheorem[style=definitionwithend]{example}
\declaretheorem{remark}

\definecolor{gold}{rgb}{0.85,0.65,0}

\numberwithin{subsection}{section}

\let\emptyset\varnothing

\usepackage{mathrsfs}
\def\A{{\bm{A}}}
\def\B{{\mathbb{B}}}

\def\N{{\mathbb{N}}}

\def\Q{{\mathbb{Q}}}
\def\R{{\mathbb{R}}}

\def\bI{{\mathbf{I}}}

\def\cA{{\cal A}}
\def\cB{{\cal B}}
\def\cC{{\cal C}}
\def\cD{{\cal D}}
\def\cE{{\cal E}}
\def\cF{{\cal F}}
\def\cG{{\cal G}}
\def\cH{{\cal H}}
\def\cI{{\cal I}}
\def\cJ{{\cal J}}

\def\cL{{\cal L}}

\def\cN{{\cal N}}
\def\cO{{\cal O}}

\def\cS{{\cal S}}

\def\cU{{\cal U}}
\def\cV{{\cal V}}

\def\sP{\mathscr{P}}

\def\a{{\boldsymbol{a}}}
\def\b{{\bm b}}

\def\d{{\bm d}}
\def\e{{\bm e}}

\def\p{{\bm p}}

\def\v{{\bm v}}

\def\x{{\bm x}}
\def\y{{\bm y}}
\def\z{{\bm z}}
\def\bz{{\bm 0}}

\def\1{{\bm 1}}

\newcommand{\blam}{{\boldsymbol\lambda}}

\newcommand{\dist}{\operatorname{dist}}

\DeclareMathOperator*{\argmin}{arg\,min}
\DeclareMathOperator*{\argmax}{arg\,max}

\DeclareMathOperator{\Diag}{Diag}

\DeclareMathOperator{\range}{range}
\DeclareMathOperator{\graph}{graph}

\DeclareMathOperator{\spann}{span}
\DeclareMathOperator{\inter}{int}

\DeclareMathOperator{\cl}{cl}
\DeclareMathOperator{\bd}{bd}

\DeclareMathOperator{\dom}{dom}

\usepackage{soul}

\title{ {\Large{\bf  A Unified Framework for Iterate Convergence of Bregman Proximal Methods} }}
\author{He Chen\thanks{Department of Systems Engineering and Engineering Management, The Chinese University of Hong Kong, Shatin, NT, Hong Kong. \texttt{hchen@se.cuhk.edu.hk}},\quad Jiaming Fan\thanks{Department of Systems Engineering and Engineering Management, The Chinese University of Hong Kong, Shatin, NT, Hong Kong. \texttt{jmfan@se.cuhk.edu.hk}},\quad Anthony Man-Cho So\thanks{Department of Systems Engineering and Engineering Management, The Chinese University of Hong Kong, Shatin, NT, Hong Kong. \texttt{manchoso@se.cuhk.edu.hk}}
}
\date{}
\begin{document}
	\maketitle
	\begin{abstract}
		Iterate convergence of Bregman proximal methods (BPMs) has long remained open, especially for nonconvex objectives. Recently, \citet{chen2026skl} made progress by establishing iterate convergence for a BPM via the so-called scaled Kurdyka-\L{}ojasiewicz (SK\L{}) property, but only for the Shannon entropy kernel and linearly constrained problems.  In this paper, we develop a unified iterate convergence framework that applies to a broad group of kernels and composite objective functions.  Our approach extends the analytical tools in \cite{chen2026skl}, in particular the SK\L{} property, which plays a central role in ensuring convergence of the generated sequences. By introducing kernel-dependent parameterization functions, we show that the extended SK\L{} property holds for all continuous subanalytic functions, particularly when the kernel has a closed domain. We then verify that the assumptions of the framework are satisfied by standard BPMs under mild regularity conditions, thereby establishing their iterate convergence for a wide range of objective functions. Furthermore, based on the parameterization functions, we show that the continuous-time BPM (mirror flow) converges to a stationary point for o-minimal definable objective functions, yielding the first trajectory convergence result for mirror flow without imposing convexity assumptions on the objective function or isolation assumptions on stationary points. Taken together, these discrete- and continuous-time convergence results provide a unified trajectory convergence theory for BPMs.
	\end{abstract}
	\section{Introduction}
	The composite optimization problem is a fundamental framework that arises in a wide range of applications, including signal and image processing \cite{Jain2017nonconvex,daubechies2004iterative}, statistical learning \cite{ng2004feature,wang2023linear,bao2022fast}, sparse recovery \cite{donoho2006compressed,candes2006robust}, and portfolio optimization \cite{brodie2009sparse,demiguel2009generalized}, among others. A typical composite optimization problem takes the form
	\begin{equation}\label{eq:model}
		\min_{\x\in\R^n}~F(\x):=f(\x)+g(\x),\tag{$\sP$}
	\end{equation}
	where $\dom(g)=\cD$ is a nonempty closed convex set, $f:\R^n \rightarrow  \R$ is a continuously differentiable function, and $ g:\R^n \rightarrow \overline \R$ is a proper convex function that may be nonsmooth. In many of these applications, the smooth term $f$ models data fidelity or empirical risk, and the nonsmooth component $g$ encodes structural properties such as sparsity, low-rankness, regularization, or explicit constraints through indicator functions.
	
	The composite structure of \eqref{eq:model} has motivated the development of a broad class of proximal methods that exploit the distinct properties of the two objective components. Given a reference point $\hat{\x}\in\R^n$, these methods typically replace the smooth part $f$ by a surrogate function $\x\mapsto\gamma(\x;\hat{\x})$ to improve tractability, and incorporate a quadratic regularizer $\x\mapsto \|\x-\hat{\x}\|_2^2$ to ensure well-posedness and uniqueness of the subproblem solution. This leads to the iterative scheme
	\begin{equation*}
		\x^{k+1}=\argmin_{\x\in\R^n}\left\{ \gamma(\x;\x^k)+g(\x)+ \frac{1}{\alpha_k}\|\x-\x^k\|_2^2\right\},\qquad\forall~k\geq0,
	\end{equation*}
	where $\alpha_k\in\R_{++}$ is the step size at the $k$-th iteration.
	A common choice of surrogate is the first-order approximation 
	\[ \gamma(\x;\hat{\x})=f(\hat{\x})+\left<\nabla f(\hat{\x}),\x-\hat{\x}\right>,\] which gives the proximal gradient method. A more conservative alternative is to take $\gamma(\x;\hat{\x})=f(\x)$, leading to the proximal point method.
	These schemes have been extensively studied \cite{beck2017first,cai2022developments} and their convergence theory is is now well established, even for nonconvex objectives. In a  seminal work, \citet{attouch2013convergence} developed a unified convergence framework rooted in the Euclidean geometry induced by the quadratic regularizer, which characterizes the behavior of Euclidean proximal methods through the so-called sufficient decrease and relative error conditions. As a consequence, the iterates generated by these methods are guaranteed to converge to a stationary point of $F$ provided that $F$ satisfies the Kurdyka-\L{}ojasiewicz (K\L{}) property, $\nabla f$ is $L$-Lipschitz continuous, and step sizes satisfy $0<\underline{\alpha}<\alpha_k<1/L$. 
	
	Despite their success, Euclidean proximal methods may fail to fully exploit the geometry induced by $g$ or the non-Euclidean smoothness structure of $f$ arising in many settings. For example, objective functions arising in certain applications do not satisfy the classical $L$-smoothness assumption, making the proximal gradient method less effective; see \cite[Section 5]{bauschke2017descent} and also \cite{lu2018relatively}. 
	A common remedy is to replace the quadratic regularizer with a generalized distance whose growth adapts to the underlying problem geometry. The Bregman distance $D_h(\x,\y)$ naturally serves as an alternative due to flexible choice of the kernel $h$ \cite{bregman1967relaxation}. Here, $h:\R^n\to\overline\R$ is a strictly convex function differentiable on $\inter(\dom(h))$, and the Bregman distance with kernel $h$ is defined by \[D_h(\x,\y)=h(\x)-h(\y)-\left<\nabla h(\y),\x-\y\right>,\qquad \forall~ \x\in\dom(h),~\y\in\inter(\dom(h)).\]
	This leads to geometry-adaptive variants of Euclidean proximal algorithms, collectively known as Bregman proximal methods (BPMs). Two prominent examples are the Bregman proximal point method (BPPM) and the Bregman proximal gradient method (BPGM). 
	Let the step sizes be denoted by $\{\alpha_k\}_{k\geq0}\subseteq\R_{++}$.
	The BPPM, introduced by \citet{censor1992proximal}, iteratively minimizes the objective function regularized by the Bregman distance, whose iterates satisfy
	\begin{equation}\label{eq:bppm}
		\begin{aligned}
			&\x^{k+1}=\argmin_{\x\in\R^n}\left\{ F(\x)+\frac{1}{\alpha_k}D_h(\x,\x^k) \right\};\\
			&\x^k\in\inter(\dom(h))\cap \dom(F),\qquad\quad\forall~k\geq0.
		\end{aligned}
		\tag{$\dagger$}
	\end{equation}
	The BPGM, also known as mirror descent \cite{nemirovskij1983problem}, follows the update rule
	\begin{equation}\label{eq:bpgm}
		\begin{aligned}
			&\x^{k+1}=\argmin_{\x\in\R^n}\left\{\left<\nabla f(\x^k),\x-\x^k\right>+g(\x)+\frac{1}{\alpha_k}D_h(\x,\x^k) \right\};\\
			&\x^k\in\inter(\dom(h))\cap \dom(F),\qquad\quad\forall~k\geq0.
		\end{aligned} \tag{$\ddagger$}
	\end{equation}
	Here, the condition $\x^k\in\inter(\dom(h))\cap \dom(F)$ in \eqref{eq:bppm} and\eqref{eq:bpgm} ensures the well-definedness of the Bregman distance and is satisfied under mild assumptions; see \cite{chen1993convergence,bolte2018first,chen2025spurious} for detailed discussions.
	
	The above BPMs have been applied in various applications \cite{lu2018relatively,ding2025stochastic,yang2025inexact}, which could admit simpler updates \cite{beck2003mirror}, achieve improved convergence rates \cite[Section 4]{bubeck2014theory}, and accommodate a wider range of tractable objective functions \cite{bauschke2017descent,lu2018relatively} compared with the Euclidean proximal methods. These advantages attract much interest to their convergence behavior. \citet{iusem1999central} posed the question on iterate convergence of BPMs for general kernels and general objective functions as early as the 1990s. For convex objective functions, substantial progress has been made: \citet{chen1993convergence} proved the iterate convergence of BPPM to an optimal solution of \eqref{eq:model}; \citet{bauschke2017descent} showed that when $Lh-f$ is convex,  the BPGM sequence $\{\x^k\}_{k\geq0}$ converges to a minimizer of $F$ with the function value gaps $F(\x^k)-\min_{\x\in\R^n}F(\x)$ decreasing to $0$ at a rate of $\cO(1/k)$. See also \cite{lu2018relatively,beck2003mirror,yang2022bregman} for other convergence analysis based on the function-value sequence. These results parallel the classical convergence theory of Euclidean proximal methods in convex optimization \cite{beck2017first} and by now are fairly complete. In contrast, for nonconvex objective functions, convergence guarantees for the sequence $\{\x^k\}_{k\geq0}$ generated by BPMs remain relatively limited. Existing convergence results can be broadly categorized according to their assumptions as follows:
	\begin{enumerate}
		\item[{\rm (S1)}] Suppose that $\nabla h$ is (locally) Lipschitz continuous with $\dom(h)=\R^n$, $F$ satisfies the K\L{} property, and $f$ is $L$-smooth. Then, the sequence $\{\x^k\}_{k\geq0}$ converges to a critical point of $F$ \cite{teboulle2018simplified,zhu2021level,wu2021inertial,takahashi2022new,takahashi2025approximate}.
		\item[{\rm (S2)}] Suppose that $F$ satisfies certain structural properties ensuring that its critical points are isolated. Then, under kernels with non-Lipschitz continuous gradients, $\{\x^k\}_{k\geq0}$ still converges to a critical point of $F$ \cite{ding2025exploration,azizian2022rate}.
		\item[{\rm (S3)}] Suppose that $F$ satisfies some Bregman growth conditions under a general kernel $h$. Then, the sequence $\{\x^k\}_{k\geq0}$ converges to an optimal solution of \eqref{eq:model} \cite{bauschke2019linear,zhang2021proximal,chirinosrodríguez2026linear}.
	\end{enumerate}
	The setting in (S1) precludes many popular kernels such as the Shannon entropy kernel  $h_S:\x\mapsto\sum_{i=1}^nx_i\log(x_i)$.
	Although (S2) and (S3) accommodate general kernels, their assumptions on the objective functions are often restrictive and hard to verify in practice. A central obstacle to removing these assumptions and establishing broad convergence guarantees is the sharp growth property of the Bregman distance, particularly near the boundary of kernel's domain. Such a geometry is highly distinguished from the Euclidean distance and leads to spurious stationary points that can trap BPMs \cite{chen2025spurious,ding2026nonkkt}. To overcome these issues and yield iterate convergence of BPGM to stationary points, \citet{chen2026skl} focuses on the Shannon entropy kernel $h_S$ and introduced the scaled K\L{} (SK\L{}) property that captures the growth behavior of functions in the Bregman geometry. Then, in the linearly constrained setting, \citet{chen2026skl} showed that (i) the sequence of BPGM with kernel $h_S$ converges to a stationary point if the objective function satisfies the SK\L{} property; (ii) all continuous subanalytic functions satisfy the SK\L{} property.  
	Though these developments provide a practical convergence guarantee, they crucially rely on the specific structure of the Shannon entropy kernel and the explicit update form of the BPGM in linearly constrained settings. This naturally calls for an extension to general kernels and problem settings. In this paper, we extend the analysis tools in \cite{chen2026skl} to a broad group of kernels and establish a unified iterate convergence framework for the BPMs. In particular, if a bounded sequence satisfies certain variants of relative error and sufficient decrease, and the objective function satisfies the (extended) SK\L{} property, then it converges. Such a framework generalizes the classical one for Euclidean proximal methods \cite{attouch2013convergence}) and requires a complementary condition that the convergence of concerned sequence implies the stationarity of limiting point. We then verify the required properties for sequences generated by the BPGM and BPPM for a broad host of kernels under mild regularity conditions. Moreover, by defining a parameterization function $\phi$ associated with each kernel $h$, we show that the extended SK\L{} property also holds for continuous subanalytic functions, particularly when the kernel has a closed domain. We thereby obtain the iterate convergence of BPGM and BPPM for a general host of kernels and objective functions through the framework. This significantly expands the applicable settings of the results in \cite{chen2026skl} and demonstrates the power of our Bregman convergence framework.
	
	
	As a byproduct of our analysis, particularly the introduction of parameterization functions, we additionally develop trajectory convergence of continuous-time BPM, i.e., the mirror flow, for \eqref{eq:model}. Let $\phi$ be the parameterization function associated with the kernel $h$.
	Our approach reduces the convergence of mirror flow for the objective $F$ to that of a Euclidean gradient flow for the composition $F\circ\phi$ via the transformation $\x=\phi(\y)$. We subsequently show that $\phi$ is o-minimal definable under mild conditions, and further derive trajectory convergence of the resulted gradient flow provided that it is bounded and the new objective $F$ is o-minimal definable. Then, to show the convergence of original mirror flow, which is generally assumed bounded, the key issue is whether our transformation to Euclidean gradient flow preserves the boundedness. As it turns out, the preservation holds when $\dom(h)$ is closed. We therefore establish trajectory convergence of bounded mirror flows to stationary points of \eqref{eq:model} for o-minimal definable objectives $F$ and kernels with closed domains, which improves existing results that require either convexity-type conditions on objective functions \cite{Attouch2004RegularizedLD,Alvarez2018HessianRG,Bolte2003BarrierOA} or an isolation assumption on stationary points \cite{ding2025exploration}. 
	Taken together, our discrete- and continuous-time convergence results contribute to a general iterate and trajectory convergence theory for BPMs, marking a significant step toward resolving this long-standing open problem.

	\paragraph{Organization.} The rest of the paper is organized as follows. Sec. \ref{sec:pre} records properties of separable kernel functions and reviews some basic notions from o-minimal structure. Sec. \ref{sec:general} defines the SK\L{} property for separable kernels and presents a general framework for establishing the iterate convergence of discrete-time BPMs.  Sec. \ref{sec:application} verifies the framework's conditions for BPPM and BPGM, completing their convergence analysis.  In Sec. \ref{sec:flow}, we establish the trajectory convergence of mirror flow. Final remarks are given in Sec. \ref{sec:conclusion}.
	
	\paragraph{Notation.} 
	The notation adopted in this paper is mostly standard. We denote the set $\{1,2,\ldots,n\}$ by $[n]$ for each positive integer $n$, and the extended real number set $\R\cup\{+\infty\}$ by $\overline{\R}$.  For a vector $\x\in\R^n$, we denote its $i$-th element (resp. subvecter indexed by $\cI\subseteq[n]$) by $x_i$ (resp. $\x_{\cI}$).
	We use $\|\x\|_2$ to denote the Euclidean norm of $\x$ and $\B(\x,r)$ to denote the ball $\{\y\in\R^n:\|\y-\x\|_2\leq r\}$.
	We use
	$\dist(\x,\cU)\coloneqq \inf_{\y\in\cU}\|\x-\y\|_2$ 
	to denote the distance of $\x$ to a set $\cU\subseteq\R^n$.  For a closed convex set $\cD\subseteq\R^n$, we define its normal cone at $\x\in\cD$ by
	\[\cN_{\cD}(\x)\coloneqq \left\{\v\in\R^n:\left<\v,\y-\x\right>\leq0,~\forall~\y\in\cD\right\}.\]
	For a proper lower-semicontinuous (lsc) function $H:\R^n\to\overline{\R}$, we define its domain by $\dom(H)\coloneqq \{\x\in\R^n:H(\x)<+\infty\}$ and its graph by $\graph(H)\coloneqq\{(\x,H(\x))\in\R^{n+1}:\x\in\dom(H)\}$. We write $H\in\cC^k(\cU)$ if $\cU\subseteq\dom(H)$ and $H$ is $k$-th order continuously differentiable on $\cU$. 
	We denote the Fr${\rm \acute{e}}$chet subdifferential of $H$ at $\x\in {\rm dom}(H)$ by $\widehat{\partial} H(\x)$, which is the set of vectors $\v\in\mathbb{R}^{n}$ such that 
	\[ \lim\limits_{\y\neq \x}\inf\limits_{\y\to\x}\frac{H(\y)-H(\x)-\left< \bm{v}, \bm{y}-\bm{x}\right>}{\|\y-\x\|_2}\geq0.\]
	We denote the limiting subdifferential of $H$ at $\x\in {\rm dom} (H)$ by $\partial H(\x)$, which is defined by
	\[\partial H(\x)\coloneqq\left\{\v\in\R^n: \exists\ \x^k\to \x,~H(\x^k)\to H(\x),~ \v^k\in\widehat{\partial} H(\x^k)\rightarrow \v \right\}.\]
	For any univariate function $\zeta:\R\rightarrow\overline{\R}$ and any vector $\x\in\R^n$, we define $\zeta(\x)\coloneqq(\zeta(x_1),\ldots,\zeta(x_n))$, and $\Diag(\x)$ by the diagonal matrix whose $i$-th diagonal element is $x_i$. 
	
	\section{Preliminaries}\label{sec:pre}
	In this section, we record essential preliminaries for our developments.
	Throughout the paper, we study kernels with a separable structure, defined as follows.
	\begin{definition}[Separable Kernels]\label{def:kernel}
		A function $h:\R^n \rightarrow \overline\R$ is called a {separable kernel function} if the following hold: 
		\begin{enumerate}[label={{\rm (\roman*)}}]
			\item There exists a lsc function  $\varphi:\R\rightarrow  \overline\R$ such that 
			$h(\x)=\sum_{i=1}^n\varphi(x_i)$.
			\item The function $\varphi$ is strongly convex and $\varphi\in\cC^3({\rm{int}}(\dom(\varphi)))$.
			\item For every sequence  $\{x^k\}_{k \ge 0} \subseteq \mathrm{int}(\dom(\varphi)) $  converging to a point $x \in \mathrm{bd}(\dom(\varphi))$, we have $|\varphi'(x^k)| \to +\infty$. 
		\end{enumerate}
	\end{definition}
	The separable structure of kernel functions has been widely adopted in the literature on BPMs; see e.g., \cite{liconvergent,bauschke2019linear,chen2025spurious}. The strong convexity and high-order smoothness of $\varphi$ are commonly assumed for studying the continuous-time BPM \cite{Alvarez2018HessianRG,bomze2019hessian,ding2025exploration}. Property (iii), known as essential smoothness, was throughly studied in \cite[Sec. 26]{rockafellar2015convex} and has been recognized as a defining characteristic of kernel functions \cite{bauschke2017descent,chen1993convergence}. As the following examples show, these properties hold for many commonly used kernels.
	\begin{example}(See \citep[Example 1]{bauschke2017descent}).
		\begin{enumerate}[label={{\rm (\roman*)}}, itemsep=0.5pt]
			\label{example:h}
			\item  Shannon entropy kernel 
			$h(\x)=\sum_{i=1}^nx_i\log(x_i)$;
			\item Fractional power kernel $h(\x)=\sum_{i=1}^npx_i-\frac{x_i^p}{1-p}$ ($0<p<1$);
			\item Fermi–Dirac entropy kernel
			$h(\x)=\sum_{i=1}^nx_i\log(x_i)+(1-x_i)\log(1-x_i)$;
			\item Hellinger entropy kernel $h(\x)=\sum_{i=1}^n-\sqrt{1-x_i^2}$;
			\item Burg entropy kernel
			$h(\x)=\sum_{i=1}^n-\log(x_i)$.
		\end{enumerate}
	\end{example}
	As a direct corollary of the properties in Definition \ref{def:kernel}, we have the following lemma concerning the directional limit of the kernel gradients, which will prove useful in analyzing the behavior of BPMs.
	\begin{lemma}\label{le:divide}
		Consider a separable kernel $h:\x\mapsto\sum_{i=1}^n\varphi(x_i)$. Let $\{\z^k\}_{k\geq0}\subseteq \inter(\dom(h))$ be a sequence converging to $\bar{\z}\in\dom(h)$. Let $\{T_k\}_{k\geq0}$ be a positive sequence such that $T_k\to+\infty$ and $\nabla h(\z^k)/T_k\to\bar\blam\in\R^n$. Then, we have
		\[\bar\blam\in\cN_{\dom(h)}(\bar{\z}),\quad{\rm  i.e. },\quad \left<\bar{\blam},\z-\bar{\z} \right>\leq0,\quad\forall~\z\in\dom(h). \]
	\end{lemma}
	\begin{proof}
		We define the interior and boundary index sets 
		\[\cI=\{i:\bar{z}_i\in\inter(\dom(\varphi))\},\quad \cJ=\{j:\bar{z}_j\in\inter(\dom(\varphi))\}.\]
		For $\bar{\blam}_{\cI}$, the convergence $\z^k\to\bar{\z}$ implies $\varphi^{\prime}(z^k_i)\to\varphi^{\prime}(\bar{z}_i)$, $i\in\cI$, and hence \[\bar{\lambda}_{i}=\lim_{k\to\infty}\frac{\varphi^{\prime}(z^k_i)}{T_k}=0,\qquad\forall~i\in\cI.\] 
		For $\bar{\blam}_{\cJ}$,
		when $\bar{z}_j$ is the left-end (resp. right-end) point of $\dom(\varphi)$, one has
		$z_j-\bar{z}_j\geq0$ (resp. $z_j-\bar{z}_j\leq0$) for $z_j\in\dom(\varphi)$; and the essential smoothness and convexity of $\varphi$, the convergence $z^k_j\to\bar{z}_j$ ensure $\varphi^{\prime}(z^k_j)\to-\infty$ (resp. $\varphi^{\prime}(z^k_j)\to+\infty$), yielding $\bar{\lambda}_j\leq0$ (resp.  $\bar{\lambda}_j\geq0$). Combined with $\bar\blam_{\cI}=\bz$, these imply
		\begin{equation*}
			\left<\bar{\blam},\z-\bar{\z} \right>=\sum_{j\in\cJ}\bar{\lambda}_j(z_j-\bar{z}_j)\leq0,\qquad\forall~\z\in\dom(h)=[\dom(\varphi)]^n. 
		\end{equation*}
	\end{proof}
	We then introduce the following settings on Problem \eqref{eq:model} and kernel functions, which are significantly more general than those in \cite{chen2026skl}.
	\begin{assumption}\label{assum:general} We assume:
		\begin{enumerate}[label={{\rm (\roman*)}}]
			\item $h:\x\mapsto\sum_{i=1}^n\varphi(x_i)$ is a separable kernel with $\cl(\dom(h))\supseteq\dom(F)=\dom(g)$. Moreover, the function $\hat{\varphi}:\cl(\dom(\varphi))\to\R_+$ defined by $\hat{\varphi}(x)=1/\varphi^{\prime\prime}(x)$ on ${\rm int}(\dom(\varphi))$ and $\hat{\varphi}(x)=0$ on ${\rm bd}(\dom(\varphi))$ is continuous, globally subanalytic.
			\item $F=f+g$, where $\dom(g)$ is a closed set,  $f$ is continuously differentiable on $\dom(g)$, and $g$ is convex, locally Lipschitz continuous on $\dom(g)$. Furthermore, $\inter(\dom(h))\cap\dom(g)$ is nonempty.
		\end{enumerate}
	\end{assumption}
	As (global) subanalyticity covers quite broad functions, the condition (i) above is very general and covers all kernels in Example \ref{example:h}. Hence, Assumption \ref{assum:general} substantially generalizes the setting $g=\delta_{\cD_P}$ and $h=h_S$ considered in \cite{chen2026skl}. More importantly, it enables an estimation of $\varphi^{\prime\prime}(x)$ near the boundary and ensures that $\varphi$ is generalized self-concordant (see \cite[Definition 1]{Sun2017GeneralizedSF}), which further leads to a characterization of the closedness of $\dom(h)$.
	\begin{proposition}\label{pro:estimate}
		Suppose that Assumption \ref{assum:general} (i) holds. Then, for any $\bar{x}\in\bd(\dom(\varphi))$, there exists a unique $\theta_{\bar{x}}\in(1,2]$ such that for $x\in\inter(\dom(\varphi))$ near $\bar{x}$,
		\begin{equation}\label{eq:estimate}
			|\varphi^{\prime\prime\prime}(x)|=\Theta\left(\varphi^{\prime\prime}(x)^{\theta_{\bar{x}}} \right),\qquad 	\varphi^{\prime\prime}(x)=\Theta\left(|x-\bar{x}|^{-\frac1{\theta_{\bar{x}}-1} }\right).
		\end{equation}
	\end{proposition}
	\begin{corollary}\label{co:closed}
		Under the setting of Proposition \ref{pro:estimate}, the following equivalence holds:
		\[\begin{aligned}
			& \dom(\varphi)~{\rm is~ closed} \\
			\Longleftrightarrow~&\theta_{\bar{x}}\in(\frac32,2],\quad\forall~\bar{x}\in\bd(\dom(\varphi));\\
			\Longleftrightarrow~& \left|\int^{x}_{y}\sqrt{\varphi^{\prime\prime}(s)}~{\rm d}s\right| <+\infty,~\forall~x,y\in\cl(\dom(\varphi));.
		\end{aligned}
		\] 
	\end{corollary}
	The above preliminaries characterize the growth behavior of the kernel functions. We next turn to regularity conditions for the objective functions and review several standard analytic notions
	\begin{definition}[O-minimal Structure \cite{coste2000introduction}]\label{def:o}
		Let $\cS=(\cS_n)_{n\in\N}$ be a collection such that $\cS_n$ is a family of subsets of $\R^n$ for each $n\in\N$. We call $\cS$ an o-minimal structure if it satisfies:
		\begin{itemize}
			\item For each $n\in\N$, $\cS_n$ contains all algebraic subsets of $\R^n$ and is a Boolean subalgebra.
			\item If $\cU\in\cS_m$ and $\cV\in\cS_n$, then $\cU\times \cV\in\cS_{m+n}$.
			\item If $\Pi_n:\R^{n+1}\to\R^n$ is the projection on the first $n$ coordinates and $\cU\in\cS_{n+1}$, then $\Pi_n(\cU)\in\cS_n$.
			\item The elements of $\cS_1$ are precisely the finite unions of points and intervals.
		\end{itemize}
	\end{definition}
	We say that	a set $\cU\subseteq\R^m$ is definable in $\cS$ if $\cU\in\cS_m$; and a function $H:\cU\to\R^n$ is definable in $\cS$ if its graph is definable in $\cS$.
	Based on Definition \ref{def:o}, it is easy to show the following properties for definable sets and functions. 
	\begin{proposition}\label{pro:o1}
		Let $\cS$ be an o-minimal structure and $\cU\subseteq\R^m$ be a definable set in $\cS$. Let $H:\cU\to\R^n$ be a definable function in $\cS$. The following hold:
		\begin{enumerate}[label={{\rm (\roman*)}}]
			\item The sets $\inter(\cU)$, $\cl(\cU)$ are also definable in $\cS$.
			\item If $H$ is injective, then its inverse $H^{-1}$ is definable in $\cS$.
			\item If the subset $\cV\subseteq\cU$ is definable in $\cS$, then the restricted function $H|_{\cV}$ is definable in $\cS$.
			\item If $n=1$ and the function $\tilde{H}:\R^m\to\R^{q}$ is defined by $\tilde{H}(\x)=(H(\x),\ldots,H(\x))$, then $\tilde{H}$ is definable in $\cS$. 
			\item If $H(\cU)\subseteq \cV\subseteq \R^n$ and $G:\cV\to\R^q$ is definable in $\cS$, then the composition $G\circ H:\cU\mapsto\R^q$ is also definable in $\cS$.
		\end{enumerate}
	\end{proposition}
	\begin{proof}
		See \cite[Proposition 1.12]{coste2000introduction} for (i); \cite[Lemma 2.3 (iv), p.14]{van1998tame} for (ii); \cite[Lemma 2.3 (ii), p.14]{van1998tame} for (iii); \cite[Exercise 1.10]{coste2000introduction} for (iv);  \cite[Lemma 2.3 (v), p.14]{van1998tame} for (v).
	\end{proof}
	We next introduce two important o-minimal structures constructed from globally subanalytic sets. Roughly speaking, a set $\cU\subseteq\R^n$ is subanalytic if it is locally the projection of a bounded real-analytic set; and it is globally subanalytic if it further remains stable under a certain semialgebraic map. Globally subanalytic sets strictly extend the class of semialgebraic sets but exclude, for instance, the graph of the exponential function.
	For precise definitions and further details on global subanalyticity, see \cite[p. 190]{van1986generalization}, \cite[p. 506]{van1996geometric}, and \cite[Sec. 2.1]{bolte2007lojasiewicz}.
	\begin{proposition}\label{pro:o2} The following hold:
		\begin{enumerate}[{\rm (i)}]
			\item There exists an o-minimal structure, denoted by $\cS(\R_{{\rm an}})=(\cS_n(\R_{{\rm an}}))_{n\in\N}$, such that $\cS_n(\R_{{\rm an}})$ is the collection of globally subanalytic sets in $\R^n$.
			\item There exists a smallest o-minimal structure that contains the graph of the exponential function and all globally subanalytic sets, denoted by $\cS(\R_{{\rm an, exp}})$. 
		\end{enumerate}
	\end{proposition}
	\begin{proof}
		See \cite{van1986generalization}, \cite[Sec. 4.1.2]{portales2025sequential} for (i); and \cite{opris2023preparation}, \cite[Sec. 4.1.3]{portales2025sequential} for (ii).
	\end{proof}
	We call a function (resp. globally) subanalytic if its graph is (resp. globally) subanalytic. It is well-known that globally subanalytic functions are not stable under integration. That is, the parameterized integration of a globally subanalytic function need not to be globally subanalytic; see, e.g., \cite[Sec. 4.2.4]{portales2025sequential}. Nevertheless, the integration might be definable in a larger structure like $\cS(\R_{{\rm an, exp}})$. 
	\begin{lemma}{\rm(\cite[Lemma 4.8]{portales2025sequential}).}\label{le:integ}
		Let $H:\R^m\times \cU\to\R$ be a globally subanalytic function with $\cU\subseteq\R^n$ open. Define $\hat{H}:\y\mapsto\int_{\R^m}H(\x,\y){\rm d}\x$ and suppose that $\hat{H}$ is well-defined on $\cU$. Then, $\hat{H}$ is definable in $\cS(\R_{{\rm an, exp}})$. 
	\end{lemma}
	
	\section{Unified Convergence Framework}\label{sec:general}
	To generalize the convergence results in \cite{chen2026skl}, we refine the analytical tools developed therein and extend them to more general settings. In particular, recall that the scaled sufficient decrease condition associated with the Shannon entropy kernel $h_S$ (see \cite[Proposition 7(ii)]{chen2026skl}) takes the form
	\begin{equation}\label{eq:scaled_suff_example}
		F(\x^{k+1})-F(\x^k)\leq -\frac{\kappa_1}2\left\|\Diag\left(\sqrt{\x^k}\right)^{-1} \left({\x^{k+1}-\x^k}\right) \right\|^2_2.
	\end{equation}
	It is helpful to realize that the scaling matrix $\Diag(\sqrt{\x^k})^{-1}$ aligns with $\nabla^2h_S(\x^k)^{\frac12}$ and the decrease quantity in the right-hand side appears as a lower bound on the Bregman sufficient decrease quantity $D_{h_S}(\x^{k+1},\x^k)$ for the BPGM. Note that the Bregman sufficient decrease 
	\begin{equation}\label{eq:Bsd}
		F(\x^{k+1})-F(\x^k)\leq -L D_h\left(\x^{k+1},\x^k\right), 
	\end{equation}
	where $L>0$ is a known constant, has become standard for BPMs under mild conditions. A natural idea to extend \eqref{eq:scaled_suff_example} to BPMs with a general kernel $h$ is to find a lower bound on $D_h(\x^{k+1},\x^k)$ in the form of scaled distance. This motivates the following scaled sufficient decrease property for the sequence $\{F(\x^k)\}_{k\geq0}$ under the kernel $h$.
	\begin{enumerate}
		\item[{\rm (H1)}] \textbf{(Scaled Sufficient Decrease).} There exists $a>0$ such that for all $k\geq0$,
		\[F(\x^{k+1})-F(\x^k)\leq -a\left\| \nabla^2h(\x^{k+1})^{\frac12} \left(\x^{k+1}-\x^k\right)\right\|_2^2.\]
	\end{enumerate}
	Accordingly, we have the generalized version of scaled relative error condition compatible with (H1). 
	\begin{enumerate}
		\item[{\rm (H2)}] \textbf{(Scaled Relative Error).} There exists $b>0$ such that for all $k\geq0$,
		\[\dist\left(\bz, \nabla^2h(\x^{k+1})^{-\frac12}\partial F\left(\x^{k+1}\right)\right)\leq b\left\|\nabla^2h(\x^{k+1})^{\frac12} \left(\x^{k+1}-\x^k\right)\right\|_2.\]
	\end{enumerate}
	It is worthy noting that the scaling matrix in (H1), (H2) is the kernel's Hessian at $\x^{k+1}$ instead of $\x^k$, different from that in \cite{chen2026skl}. This is to facilitate the development of scaled relative error for BPMs for general $g$ in the objective of \eqref{eq:model} (see Section \ref{sec:application} for details). 
	With (H1) and (H2) in place, we next extend the SK\L{} property to a general kernel $h$.
	\begin{definition}[Extended SK\L{} Property]\label{def:general_skl}
		Fix a strongly convex kernel function $h\in\cC^2(\inter(\dom(h)))$. 
		We say that a function $H:\cl(\dom(h)) \rightarrow \overline{\R}$ has the {\rm  SK\L} property at $\x^* \in \operatorname{dom} (\partial H)$ under the kernel $h$ if there exists a scalar $\eta \in(0,+\infty]$, a neighborhood $\cU$ of $\x^*$, and a continuous concave function $\zeta:[0, \eta) \rightarrow \mathbb{R}_{+}$ such that
		\begin{enumerate}[label={{\rm (\roman*)}}]
			\item $\zeta$ is continuously differentiable on $(0, \eta)$ with $\zeta(0)=0$ and $\zeta^{\prime}>0$ over $(0, \eta)$,
			\item for all $\x\in\cU\cap\inter(\dom(h))$ with $H(\x^*)<H(\x)<H(\x^*)+\eta$, it holds that
			$$
			\zeta^{\prime}\left(H(\x)-H(\x^*)\right) \cdot\operatorname{dist}\left(\bz,  \nabla^2h(\x)^{-\frac12}\partial H(\x)\right) \geq 1.
			$$
		\end{enumerate}		
		If $H$ possesses the above SK\L\ property\footnote{A key distinction between the  SK\L{} property and the K\L{} property on Hessian-Riemannian manifold (see, e.g., \cite{bolte2003barrier,alvarez2004hessian}) lies in the treatment of boundary points. The SK\L{} property is formulated for points $\x^*\in\dom(\partial F)$, including those on $\bd(\dom(h))$, whereas the Hessian-Riemannian K\L{} property is defined only for $\x^*\in\inter(\dom(h))$, where the underlying manifold structure exists. As a result, the Hessian-Riemannian K\L{} property does not generally imply the SK\L{} property; see Example \ref{example:skl_fails}.} at every $\x^*\in\dom(\partial H)$, then we call $H$ an SK\L\ function under the kernel $h$.
	\end{definition}
	The (extended) scaled sufficient decrease (H1) and scaled relative error (H2), along with the (extended) SK\L{} property, ensure the convergence of the sequence $\{\x^k\}_{k\geq0}$ through an analysis analogous to that of \citet{attouch2013convergence}. However, due to the presence of spurious stationary points introduced in \cite{chen2025spurious}, the convergence of $\{\x^k\}_{k\geq0}$ with $\dist(\bz, \nabla^2h(\x^{k+1})^{-\frac12}\partial F(\x^{k+1}))\to0$ does not, by itself, guarantee that the limit point is stationary. This gap necessitates the introduction of the complementary condition (H3), which ensures the avoidance of spurious stationary points.
	\begin{enumerate}
		\item[{\rm (H3)}] \textbf{(Limiting Stationarity).} The convergence of $\{\x^k\}_{k\geq0}$ implies the stationarity of its limiting point, i.e., \[\x^k\to\bar{\x}~\Longrightarrow~\bz\in\partial F(\bar{\x}).\]
	\end{enumerate}
	Now, we are ready to present the main theorem of this section, which ensures the convergence of sequences satisfying the conditions (H1)---(H3).  We defer its proof to Section \ref{sec:proof_general}.
	\begin{theorem}[Convergence Framework]\label{th:general}
		Suppose that Assumption \ref{assum:general} (i) holds, and $F:\R^n\to\overline{\R}$ is an SK\L{} function under the kernel $h$ and continuous in its closed domain. Let $\{\x^k\}_{k\geq0}\subseteq \dom(\partial F)\cap\inter(\dom(h))$ be a bounded sequence satisfying {\rm (H1)---(H3)}. Then, $\{\x^k\}_{k\geq0}$ converges to a stationary point of $F$ with a finite length.
	\end{theorem}
	Theorem \ref{th:general} provides a powerful approach for establishing iterate convergence in Bregman geometry: For BPMs and an SK\L{} objective function, if we can verify all three conditions (H1)---(H3) for the generated sequences, then they converge to stationary points. Particularly, it substantially extends the classical convergence result of \citet{attouch2013convergence} in Euclidean geometry and the limited result of \cite{chen2026skl} for Shannon entropy kernel. Nevertheless, before applying the framework to the usual BPMs such as the BPPM and BPGM, a primary question is how general the extended  SK\L{} property is under kernels satisfying Assumption \ref{assum:general} (i). One may expect that the extended  SK\L{} property still holds for all continuous subanalytic functions. Curiously, the result is negative: Even when $F$ is a simple semi-algebraic function, the extended SK\L{} property of $F$ may fail, as illustrated by the following example. 
	\begin{example}[Failure of Extended SK\L{} Property]\label{example:skl_fails}
		Consider the Burg entropy kernel $h(x)=-\log(x)$, the function $F:x\mapsto x+\delta_{\R_+}(x)$, and the boundary point $x^*=0$. We have
		\[  F(x)-F(x^*)=x=\left|h^{\prime\prime}(x)^{-\frac12} F^{\prime}(x)\right|,   \qquad\forall~x>0. \]
		We note that $F$ does not satisfy the SK\L{} property at $\x^*$, i.e., there does not exist a function $\zeta:[0,\eta)\to\R_+$ continuously differentiable on $(0,\eta)$ with $\zeta(0)=0$, $\eta>0$  such that 
		\[\zeta^{\prime}\left(F(x)-F(x^*)\right)\cdot \left|h^{\prime\prime}(x)^{-\frac12} F^{\prime}(x)\right|=\zeta^{\prime}(x)\cdot x\geq1,\qquad \forall~x\in(0,\eta). \]
		Because this inequality would imply
		\[\zeta(x)=\zeta(0)+\int_0^{x}\zeta^{\prime}(s){\rm d}s\geq \int_0^{x}\frac1s{\rm d}s=+\infty,\qquad \forall~x\in(0,\eta),\]
		which leads to a contradiction.
	\end{example}
	The failure of extended SK\L{} property in the above example can be traced to the fact that $|h^{\prime\prime}(x)^{-\frac12} F^{\prime}(x)|$ becomes too small; or equivalently, $h^{\prime\prime}(x)= 1/x^2$ grows excessively near the boundary point $x^*=0$. In comparison,  for the Shannon entropy kernel considered in \cite{chen2026skl}, where $h_S(\x)=\sum_{i=1}^n\varphi_S(x_i)$ with $\varphi_S(x_i)=x_i\log(x_i)$, the second-order derivative $\varphi_S^{\prime\prime}(x_i)=1/x_i$ exhibits moderate growth near the boundary $x^*_i=0$, under which the SK\L{} property broadly holds. This motivates the search for structural conditions on kernels that ensure the validity of the extend SK\L{} property. As it turns out, the closedness of $\dom(h)$ is sufficient.
	\begin{proposition}\label{pro:skl_general}
		Suppose that Assumption \ref{assum:general} (i) holds with $\dom(\varphi)$ closed. Then, every proper, subanalytic function $H:\cl(\dom(h))\to\overline{\R}$ that is continuous on $\dom(H)$ is an SK\L{} function under the kernel $h$.
	\end{proposition} 
	Proposition \ref{pro:skl_general} provides a solid basis for our convergence framework, particularly when the kernel has a closed domain. Its proof relies on the  parameterization functions associated with the kernels, as detailed in Sections \ref{sec:parameterization} and \ref{sec:proof_skl_general}. Readers who are more interested in iterate convergence of BPPM and BPGM may directly go to the next section.
	\subsection{Toolbox: Parameterization Functions} \label{sec:parameterization}
	We introduce the following parameterization functions for analyzing the extended SK\L{} property.
	\begin{definition}\label{def:psi}
		Given a separable kernel  $h:\x\mapsto\sum_{i=1}^n\varphi(x_i)$, 
		we define its inverse parameterization function $\psi:\cl(\dom(\varphi))\to\overline{\R}$ and  parameterization function $\phi:\range(\psi)\to\cl(\dom(\varphi))$ by
		\begin{equation}\label{eq:psi}
			\begin{aligned}
				\psi(x)&=\int_{x_0}^x\sqrt{\varphi^{\prime\prime}(s)}~{\rm d}s+c,\quad &x&\in\cl(\dom(\varphi));\\
				\phi(y)&=\psi^{-1}(y),\quad &y&\in\range(\psi),
			\end{aligned}\tag{$\Psi$}
		\end{equation}
		where  $c\in\R$ and $x_0\in\{y\in\dom(\varphi):|\int_{y}^x\sqrt{\varphi^{\prime\prime}(s)}~{\rm d}s|<+\infty,~\forall~x\in\inter(\dom(\varphi))\}$. 
	\end{definition}
	Clearly, the definition of the parameterization function $\phi$ is closely related to the SK\L{} property under the kernel $h$. Through a subdifferential chain rule, one may interpret the SK\L{} property of $H$ under $h$ as the K\L{} property of the composite function $H\circ\phi$, and hence utilize the K\L{} theory to prove Proposition \ref{pro:skl_general}.  Before proceeding with this analysis, we calculate several examples of parameterization functions to illustrate the accessibility of Definition \ref{def:psi}. 
	\begin{example}\label{example:psi} Let $x_0=0$, $c=0$ in \eqref{eq:psi} for Example \ref{example:h} (i)---(iii)\footnote{We omit the parameterization functions for the Hellinger entropy kernel since they involve elliptic integrals and do not admit simple forms.}; and $x_0=1$, $c=0$ for Example \ref{example:h} (v). The resulting expressions of $\psi$, $\phi$ are as follows:
		\begin{enumerate}[label={{\rm (\roman*)}}]
			\item For the Shannon entropy kernel  $h(\x)=\sum_{i=1}^nx_i\log(x_i)$,
			\[\begin{aligned}
				\psi(x)&=\int_{0}^x \frac{{\rm d}s}{\sqrt{s}}=\sqrt{x},\qquad\qquad &x\in\R_+;\\
				\phi(y)&=\psi^{-1}(y)=y^2,\qquad\qquad &y\in\R_+. 
			\end{aligned}
			\]
			\item For the fractional power kernel $h(\x)=\sum_{i=1}^npx_i-\frac{x_i^p}{1-p}$ ($0<p<1$),
			\[\begin{aligned}
				\psi(x)&=\int_{0}^x \frac{\sqrt{p}~{\rm d}s}{s^{1-\frac12p}}=\frac{2}{\sqrt{p}}x^{\frac{p}2},\qquad &x\in\R_+;\\
				\phi(y)&=\psi^{-1}(y)=\left(\frac{\sqrt{p}}{2}y\right)^{\frac2p},\qquad &y\in\R_+.
			\end{aligned}
			\]
			\item  For the Fermi–Dirac entropy kernel  	$h(\x)=\sum_{i=1}^nx_i\log(x_i)+(1-x_i)\log(1-x_i)$, 
			\[\begin{aligned}
				\psi(x)=&\int_{0}^x \frac{{\rm d}s}{\sqrt{s(1-s)}}=2\arcsin\left(\sqrt{x}\right),\quad &x\in[0,1];\\
				\phi(y)=&\psi^{-1}(y)=\sin^2\left(\frac{y}2\right),\quad &y\in[0,\pi].
			\end{aligned}
			\] 
			\item For the Burg entropy kernel
			$h(\x)=\sum_{i=1}^n-\log(x_i)$, 
			\[\begin{aligned}
				\psi(x)&=\int_{1}^x \frac{{\rm d}s}{s}=\log(x),\qquad\qquad &x&\in\R_{++};\\
				\phi(y)&=\psi^{-1}(y)=\exp(y),\qquad\qquad &y&\in\R. 
			\end{aligned}
			\]
		\end{enumerate}
	\end{example}
	The above examples suggest that $\psi,\phi$ are (or can be approximated by) power functions near the boundary when $\dom(h)$ is closed. Indeed, under the closedness of $\dom(h)$, we have the following properties for $\psi$ and $\phi$, which will find useful for the proof of Proposition \ref{pro:skl_general}. 
	\begin{proposition}\label{pro:parameterization}
		Suppose that Assumption \ref{assum:general} (i) holds with $\dom(h)$ closed. Let $\psi,\phi$ be parameterization functions for the kernel $h:\x\mapsto\sum_{i=1}^n\varphi(x_i)$ as in Definition \ref{def:psi}. Then, we have:
		\begin{enumerate}[label={{\rm (\roman*)}}]
			\item $\psi$ and $\phi$ are well-defined, continuous, and strictly increasing on $\dom(\varphi)$ and $\range(\psi)$, respectively.
			\item For any $\bar{y}\in\bd(\dom(\phi))$, there is a unique $\vartheta_{\bar{y}}\in[2,+\infty)$ such that for $y\in\dom(\phi)$ near $\bar{y}$,\[|\phi(y)-\phi(\bar{y})|=\Theta\left(|y-\bar{y}|^{\vartheta_{\bar{y}}}\right).\]
			In particular, $\vartheta_{\bar{y}}=\frac{2\theta_{\bar{x}}-2}{2\theta_{\bar{x}}-3}$ with $\bar{x}=\phi(\bar{y})$ and $\theta_{\bar{x}}\in (\frac32,2]$ as defined in Proposition \ref{pro:estimate}.
		\end{enumerate}
	\end{proposition}
	\begin{proof}
		(i) By the closedness of $\dom(h)$, we know that $\dom(\varphi)$ is closed. Then, Corollary \ref{co:closed} yields $|\int^{x}_{y}\sqrt{\varphi^{\prime\prime}(s)}~{\rm d}s| <+\infty$ for all $x,y\in\dom(\varphi)$, which ensures the well-definedness and finiteness of $\psi(x)$ for $x\in\dom(\varphi)$. Moreover, since $\varphi^{\prime\prime}$ is positive and continuous on $\inter(\dom(\varphi))$ by Definition \ref{def:kernel} (ii),  we see that $\psi$ is strictly increasing and continuous from the definition \eqref{eq:psi}. It follows that its inverse function $\phi=\psi^{-1}$ is also well-defined, continuous, and strictly increasing. 
		
		(ii) Let $\bar{x}=\phi(\bar{y})$. By $\bar{y}\in\bd(\dom(\phi))$ and strict increasing of $\phi$, we also have $\bar{x}\in \bd(\range(\phi))=\bd(\dom(\varphi))$. Then, Proposition \ref{pro:estimate} and Corollary \ref{co:closed} ensure that there exists $\theta_{\bar{x}}\in(\frac32,2]$ and $\Delta>0$ such that 
		\begin{equation*}
			\varphi^{\prime\prime}(x)=\Theta\left(|x-\bar{x}|^{-\frac1{\theta_{\bar{x}}-1}} \right)\quad {\rm on }\quad[\bar{x}_j-\Delta,\bar{x}_j+\Delta]\cap \inter(\dom(\varphi)).
		\end{equation*}
		This, together with $|\psi(x)-\psi(\bar{x})|=|\int_{\bar{x}}^x\sqrt{\varphi^{\prime\prime}(s)}{\rm d}s|$ implies that on $[\bar{x}-\Delta,\bar{x}+\Delta]\cap \inter(\dom(\varphi))$,
		\[ |\psi(x)-\psi(\bar{x})|=\Theta\left(\left|\int_{\bar{x}}^x\left| s-\bar{x}  \right|^{-\frac1{2(\theta_{\bar{x}}-1)}}{\rm d}s\right| \right)=\Theta\left(\left|x-\bar{x}  \right|^{\frac{2\theta_{\bar{x}}-3}{2\theta_{\bar{x}}-2}} \right) \]
		Let $x=\phi(y)$ in the above equation and recall $\bar{x}=\phi(\bar{x})$. Using $\phi=\psi^{-1}$, we obtain $|y-\bar{y}|=\Theta(|\phi(y)-\phi(\bar{y})|^{\frac{2\theta_{\bar{x}}-3}{2\theta_{\bar{x}}-2}})$, or equivalently
		\[ \left|\phi(y)-\phi(\bar{y})\right|=\Theta\left(|y-\bar{y}|^{\frac{2\theta_{\bar{x}}-2}{2\theta_{\bar{x}}-3}} \right)\text{ on }\quad [\psi(\bar{y}-\Delta),\psi(\bar{y}+\Delta)]\cap \inter(\dom(\phi)). \]
		Let $\vartheta_{\bar{y}}=\frac{2\theta_{\bar{x}}-2}{2\theta_{\bar{x}}-3}$. We have $\vartheta_{\bar{y}}\in[2,+\infty)$ by $\theta_{\bar{x}}\in(\frac32,2]$. Notice that (i) $[\psi(\bar{y}-\Delta),\psi(\bar{y}+\Delta)]$ is a neighborhood of $\bar{x}$ by the strict increasing of $\psi$; and  (ii) both sides equal zero when $y=\bar{y}$. The above equation completes the proof.

	\end{proof}
	
	\subsection{Proof of Proposition \ref{pro:skl_general}}\label{sec:proof_skl_general}
	Fix the kernel $h:\x\mapsto\sum_{i=1}^n\varphi(x_i)$ and let $\psi,\phi$ be the parameterization functions for $\varphi$ as defined in Definition \ref{def:psi}.	We prove the SK\L{} property of $H$ under $h$ by considering the K\L{} property of an appropriately constructed composite function. Since there is no guarantee that $\phi$ is subanalytic, the naive construction $H\circ\phi$ may neither be subanalytic nor satisfy the K\L{} property. Instead, our strategy is to compose $H$ with a power function based on the estimation in Proposition \ref{pro:parameterization} (ii).
	
	To begin, for an arbitrary $\x^*\in\dom(\partial F)$, we define the interior, left-end, and right-end index sets as follows:
	\[\begin{aligned}
		\cI&:= \left\{i:x^*_i\in\inter(\dom(\varphi))\right\};\\
		\cJ_L&:= \left\{j:x^*_j\in\argmin_{x\in\dom(\varphi)}x \right\}; \\
		\cJ_R&:= \left\{j:x^*_j\in\argmax_{x\in\dom(\varphi)}x \right\}. 
	\end{aligned}
	\]
	Then, Proposition \ref{pro:estimate} and Corollary \ref{co:closed} ensure that for $j\in\cJ_L(\x^*)\cup\cJ_R(\x^*)$, there exists $\theta_j\in(\frac32,2]$ and $\Delta_j>0$ such that 
	\begin{equation}\label{eq:theta}
		\varphi^{\prime\prime}(x)=\Theta\left(|x-x^*_j|^{-\frac1{\theta_j-1}} \right)\quad {\rm on }\quad[x^*_j-\Delta_j,x^*_j+\Delta_j]\cap \inter(\dom(\varphi)).
	\end{equation}
	Let $\y^*=\phi^{-1}(\x^*)=\psi(\x^*)$. By the strict increasing property of $\psi$, $\cI$, $\cJ_L$, and $\cJ_R$ are also the interior, left-end, and right-end index sets of $\y^*$, respectively.
	By Proposition \ref{pro:parameterization} (ii), for $j\in\cJ_L(\x^*)\cup\cJ_R(\x^*)$, there exists $\vartheta_j=\frac{2\theta_j-2}{2\theta_j-3}\in[2,+\infty)$ such that for $y_j\in\dom(\phi)$ near $y^*_j$,
	\[|\phi(y_j)-\phi(y^*_j)|=\Theta\left(|y_j-y^*_j|^{\vartheta_j}\right).\]
	Consider $\vartheta^*_j=\frac{2\theta^*_j-2}{2\theta^*_j-3}\in[2,+\infty)\cap\Q$ with $\vartheta^*_j\geq\vartheta_j$ and $\frac32<\theta^*_j\leq\theta_j$ for $j\in\cJ_L(\x^*)\cup\cJ_R(\x^*)$. We define the composite function $\Phi:\R^n\to\overline{\R}$ by $\Phi(\y)=H(G(\y))$ with $G$ given by
	\[\begin{aligned}
		G_i(y)&=x^*_i+y,\qquad&\forall~y\in\R,~&i\in\cI(\x^*);\\
		G_j(y)&= x^*_j+y^{\vartheta^*_j },\qquad&\forall~y\geq0,~&j\in\cJ_L(\x^*);\\
		G_j(y)&= x^*_j-y^{\vartheta^*_j},\qquad&\forall~y\geq0,~&j\in\cJ_R(\x^*).\\
	\end{aligned}
	\]
	Since $\vartheta^*_j\in\Q$, the graph of $G$ can be characterized by several polynomial equalities. It follows that $G$ is a semi-algebraic function, and thus subanalytic. Since $F$ is also subanalytic and $G$ is continuous, by \cite[(p5), p. 597]{facchinei2003finite}, the composite function $\Phi$ is subanalytic. Moreover, $\Phi$ is continuous on $\dom(\Phi)$ since $F$ and $G$ are both continuous. 
	
	Now, apply Theorem \cite[Theorem 3.1 and Remark 3.2]{bolte2007lojasiewicz}. We see that there exists an exponent $\eta\in[0,1)$, a scalar $c>0$, and a neighborhood $\cV$ of $\y^*=\bz$ such that 
	\begin{equation}\label{eq:H_kl}
		\left|\Phi(\y)-\Phi(\bz)\right|^{\eta}\leq c\cdot \dist\left(\bz,\partial \Phi(\y) \right)\qquad\forall~\y\in\cV,
	\end{equation}
	where we use the convention $\dist(\bz,\emptyset)=+\infty$ and $\infty\leq c^{\prime}\cdot\infty$ for all $c^{\prime}>0$.
	
	We then deduce the extended SK\L{} inequality of $H$ from \eqref{eq:H_kl}.
	Note that $G_i$ is a continuous surjective onto $\dom(h)$ and strictly increasing/decreasing. We know that $G^{-1}$ is well-defined and continuous on $\dom(h)$. Moreover, since $G(\bz)=\x^*$, there is a neighborhood $\cU\subseteq\dom(h) $ of $\x^*$ such that $G^{-1}(\cU)\subseteq\cV$.  
	
	Then, for every $\x\in\cU\cap\inter(\dom(h))$, let $\y=G^{-1}(\x)$ in \eqref{eq:H_kl}. Notice that $\y\in\cV$ and $\Phi(\bz)=H(G(\bz))=H(\x^*)$. We obtain
	\begin{equation}\label{eq:FH}
		\left|H(\x)-H(\x^*)\right|^{\eta}\leq c\cdot \dist\left(\bz,\partial \Phi(\y) \right),\qquad\forall~\x\in\cU\cap\inter(\dom(h)),~\y=G^{-1}(\x).
	\end{equation}
	Observe that when $x_j\in\inter(\dom(\varphi))$ and $y_j=G^{-1}_j(x_j)$, for $j\in\cJ_L(\x^*)\cup\cJ_R(\x^*)$, we have $y_j>0$ and $|G_j^{\prime}(y_j)|>0$; for  $i\in\cI(\x^*)$, we have $G_i^{\prime}(y_i)=1$. It follows that $\nabla G(\y)=\Diag(G_1^{\prime}(y_1),\ldots,G_n^{\prime}(y_n))$ is of full rank for $\y=G^{-1}(\x)$, $\x\in\inter(\dom(h))$. Then, by the subdifferential chain rule \cite[Exercise 10.7]{rockafellar2009variational}, we have
	\[\partial \Phi(\y) =\Diag\left(G_1^{\prime}(y_1),\ldots,G_n^{\prime}(y_n)\right)\partial H(\x),\quad\forall~\x\in\inter(\dom(h)),~\y=G^{-1}(\x). \]
	This, together with \eqref{eq:FH}, implies that for $\x\in\cU\cap\inter(\dom(h)),~\y=G^{-1}(\x)$,
	\begin{equation}\label{eq:FG}
		\left|H(\x)-H(\x^*)\right|^{\theta^*}\leq c\cdot \dist\left(\bz,\Diag(G_1^{\prime}(y_1),\ldots,G_n^{\prime}(y_n))\partial H(\x)\right).
	\end{equation}
	It is left to find an upper bound on the right-hand side. WLOG, we may assume that $\cU$ is closed and $\cU\subseteq\B(\x^*,\Delta^*)$ for a sufficient small $\Delta^*>0$ so that for $\x\in\cU\cap\inter(\dom(h))$, 
	\[\begin{aligned}
		&x_i\in\inter(\dom(\varphi)),\qquad&\forall~ i\in\cI;\\
		&x_j\in[x^*_j-\Delta_j,x^*_j+\Delta_j]\cap \inter(\dom(\varphi)),\qquad&\forall~j\in\cJ_L\cup\cJ_R.
	\end{aligned}\]
	Then, the estimate \eqref{eq:theta} and $\theta_j\geq \theta^*_j>\frac32$, together with the continuity of $\varphi^{\prime\prime}$ on $\inter(\dom(\varphi))$,  implies that there is a lower bound $a>0$ such that for all $\x\in\cU\cap\inter(\dom(h))$ and $\y=G^{-1}(\x)$,
	\begin{equation*} 
		\begin{aligned}
			&\varphi^{\prime\prime}(x_i)^{-\frac12}\geq a,\quad  &\forall~i\in\cI;\qquad
			&\varphi^{\prime\prime}(x_j)^{-\frac12}\geq a\cdot|x-x^*_j|^{\frac{1}{2(\theta^*_j-1)}},\quad  &\forall~j\in\cJ_L\cup\cJ_R,\\
		\end{aligned}
	\end{equation*}
	Recall that (i) $|G^{\prime}_i(y_i)|=1$ for $i\in\cI$; (ii) $|x_j-x^*_j|=|G_j(y_j)-x^*_j|=y_j^{\vartheta^*_j}=y_j^{\frac{2\theta^*_j-2}{2\theta^*_j-3}}$ for $j\in\cJ_L\cup\cJ_R$. We further have
	\begin{equation}\label{eq:boundvarphi}
		\begin{aligned}
			&\varphi^{\prime\prime}(x_i)^{-\frac12}\geq a\cdot \left|G^{\prime}_i(y_i)\right|,\quad  &\forall~i\in\cI;\\
			&\varphi^{\prime\prime}(x_j)^{-\frac12}\geq a\cdot y_j^{\frac{1}{2\theta^*_j-3}}=\frac{a}{\vartheta^*_j}\cdot \left|G^{\prime}_j(y_j)\right|,\quad  &\forall~j\in\cJ_L\cup\cJ_R,\\
		\end{aligned}
	\end{equation}
	
	Let $b=\min\{a,a/\vartheta_j^*:j\in\cJ_L\cup\cJ_R\}$, which is positive since $a>0$ and $\vartheta^*_j\geq2$. We combine \eqref{eq:FG} and \eqref{eq:boundvarphi} to yield
	\[\begin{aligned}
		\left|H(\x)-H(\x^*)\right|^{\eta}&\leq \frac{c}{b}\cdot \dist\left(\bz,\Diag\left(\varphi^{\prime\prime}(x_1)^{-\frac12},\ldots,\varphi^{\prime\prime}(x_n)^{-\frac12}\right)\partial H(\x)\right)\\
		&= \frac{c}{b}\cdot\dist\left(\bz,\nabla^2h(\x)^{-\frac12}\partial H(\x)\right),
		\qquad\forall~\x\in\cU\cap\inter(\dom(h)).
	\end{aligned}\]
	This says that $H$ satisfies the SK\L{} property at $\x^*$ with the desingularizing function $\zeta:x\mapsto \frac{c}{b(1-\eta)}x^{1-\eta}$, where $\eta\in[0,1)$. Since $\x^*\in\dom(\partial H)$
	is arbitrary, we conclude that $F$ is an SK\L{} function under the kernel $h$.
	\subsection{Proof of Theorem \ref{th:general}}\label{sec:proof_general}
	Let $\x^*$ be an accumulation point of the bounded sequence $\{\x^k\}_{k\geq0}$, i.e.,
	there exists a subsequence $ \{\x^{k_t}\}_{t\geq0} $ such that $ \x^{k_t} \to \x^* $. 
	Since the function $F$ is continuous, the subsequence convergence directly yields $ F(\x^{k_t}) \to F(\x^*) $. 
	We observe that $F(\x^k)$ admits a monotone decreasing property from the scaled sufficient decrease condition (H1). Combined with the subsequence convergence, this yields 
	\begin{equation}\label{eq:value_conv}
		F(\x^k) \;\to\; F(\x^*);\qquad F(\x^0)\geq F(\x^k)\geq F(\x^*),\quad\forall~k\geq0.
	\end{equation}
	Since $F$ satisfies the SK\L{} property at $ \x^*$ under the kernel $h$, there exists a scalar $\eta \in(0,+\infty]$, a neighborhood $\cU$ of $\x^*$, and a continuous concave function $\zeta:[0, \eta) \rightarrow \mathbb{R}_{+}$ such that
	\begin{enumerate}[label={{\rm (\roman*)}}]
		\item $\zeta$ is continuously differentiable on $(0, \eta)$ with $\zeta(0)=0$ and $\zeta^{\prime}>0$ over $(0, \eta)$,
		\item for all $\x\in\cU\cap\inter(\dom(h))$ with $F(\x^*)<F(\x)<F(\x^*)+\eta$, 
		$$
		\zeta^{\prime}\left(F(\x)-F(\x^*)\right) \cdot\operatorname{dist}\left(\bz,  \nabla^2h(\x)^{-\frac12}\partial F(\x)\right) \geq 1.
		$$
	\end{enumerate}		
	Consider $d>0$ sufficiently small so that $\B(\x^*, d) \subseteq \cU $. Since the  kernel function $h$ is strongly convex, there exists a constant $\sigma > 0 $ satisfying
	\begin{equation}\label{stongconvexkernel}
		\nabla^2h(\x) \geq\sigma\bI, 
		\quad \forall\, \x \in \mathbb{B}(\x^*, d).
	\end{equation}
	Recall $F(\x^k)\to F(\x^*)$. By discarding finite iterates, without loss of generality, we may assume that
	\begin{equation}\label{initialbound}
		\|\x^0 - \x^*\|_2 
		+ 2\sqrt{\frac{ F(\x^0)-F(\x^*) }{a\sigma}}
		+ \frac{b}{a\sqrt{\sigma}}\,
		\zeta\left(F(\x^0)-F(\x^*)\right)
		< d;
	\end{equation}
	\begin{equation}\label{fxeta}
		F(\x^{0}) \;<\; F(\x^*) + \eta.
	\end{equation}
	Our strategy is to establish the following inequality by induction:
	\begin{equation}\label{16}
		\|\x^{0} - \x^*\|_2 + \sum_{t=0}^K \|\x^{t+1} - \x^t\|_2 < d.
	\end{equation}
	We first prove \eqref{16} for $K=0$. 
	Note that the scaled sufficient decrease (H1) for $k=0$ yields
	\begin{align*}
		F(\x^{1}) - F(\x^{0}) &\leq -a
		\left\|\nabla^2h(\x^{1})^{\frac{1}{2}}(\x^{1}-\x^{0})\right\|^2_2.
	\end{align*}
	This, together with \eqref{stongconvexkernel}, implies
	$F(\x^{1}) - F(\x^{0}) \leq -a\sigma\|\x^{1} - \x^{0}\|_2^2$, and further  
	\begin{equation*}\label{upper-bound}
		\|\x^{0} - \x^1\|_2  
		\;\leq\; \sqrt{\frac{ F(\x^{0}) - F(\x^{1}) }{a\sigma}}  
		\;\leq\; \sqrt{\frac{ F(\x^{0}) - F(\x^*) }{a\sigma}},
	\end{equation*}
	where the second inequality uses $F(\x^k)\geq F(\x^*)$ for all $k\geq0$.
	
	Combined with \eqref{initialbound}, the above bound yields 
	$\|\x^{0}-\x^*\|_2 + \|\x^{0}-\x^{1}\|_2 < d$.  
	We conclude that \eqref{16} holds for $K=0$.
	Now, suppose that \eqref{16} holds for $K=k$ for some $k\geq 0$, which, by the triangle inequality, yields
	\[
	\x^t \in \mathbb{B}(\x^*,d), \quad t=0,1,\ldots, k+1.
	\]
	To complete the induction, we prove that \eqref{16} holds for $K=k+1$.
	Let us first consider the trivial case $F(\x^{t_0}) = F(\x^*)$ for some $t_0\in\N$.  
	By the descrease condition (H1) and \eqref{eq:value_conv}, we have
	\[
	F(\x^*) \leq F(\x^{t_0+1}) \leq F(\x^{t_0})=F(\x^*),
	\]
	and further $F(\x^{t_0+1})=F(\x^*)=F(\x^{t_0})$. By (H1), we have $\x^{t_0+1}=\x^{t_0}$, and inductively $\x^k \equiv \x^{t_0}$ for all $k \geq t_0$.
	It follows that $\{\x^k\}_{k\geq0}$ converges to $\x^{t_0}$. Then, by the limiting stationarity condition (H3), we see that $\x^{t_0}$ is a stationary point of $F$. Moreover, we have
	\[
	\|\x^{0}-\x^*\|_2 + \sum_{t=0}^\infty \|\x^t-\x^{t+1}\|_2
	= \|\x^{0}-\x^*\|_2 + \sum_{t=0}^{t_0-1} \|\x^t-\x^{t+1}\|_2 < d,
	\]
	where the last inequality follows from~\eqref{16} with $K=k$. Hence, in this case, the induction \eqref{16} and the desired convergence result hold.
	
	Therefore, we only need to focus on the nontrivial case, where $F(\x^k) > F(\x^*)$ for all $k\geq0$. Recall \eqref{fxeta}. We know that in this case
	\begin{equation}\label{17}
		F(\x^*) < F(\x^k) < F(\x^*)+\eta, \qquad \forall~k\geq0.
	\end{equation}
	The above bound, together with the inclusion 
	$\x^t \subseteq \mathbb{B}(\x^*, d) \subseteq \mathcal{U}$ for $t=0,1,\ldots,k+1$, 
	ensures that the extended SK\L{} inequality of $F$ at $\x^*$ applies to each $\x^t$, $t=0,1,\ldots,k+1$.  
	Using the concavity of $\zeta$ and the SK\L{} inequality
	\[
	\zeta\big(F(\x)-F(\x^*)\big) \cdot 
	\dist\left(\bz, \nabla^2 h(\x)^{-\frac{1}{2}} \partial F(\x)\right)\geq 1,
	\]
	we have, for $t=0,1,\ldots,k+1$,
	\begin{equation}\label{18}
		\begin{aligned}
			&\zeta\big(F(\x^t)-F(\x^*)\big) - \zeta\big(F(\x^{t+1})-F(\x^*)\big) \\
			\geq& \zeta'\big(F(\x^t)-F(\x^*)\big)\big(F(\x^t)-F(\x^{t+1})\big) \\
			\geq &\frac{F(\x^t)-F(\x^{t+1})}{\dist\left(\bz, \nabla^2 h(\x^t)^{-\frac{1}{2}}  \partial F(\x^t)\right)}.
		\end{aligned}
	\end{equation}
	This, together with (H1) and (H2), yields that for $t=1,2\ldots,k+1$,
	\[
	\zeta\big(F(\x^t)-F(\x^*)\big) - \zeta\big(F(\x^{t+1})-F(\x^*)\big) \geq \frac{a}{b}\cdot \frac{ \left\| \nabla^2h(\x^{t+1})^{\tfrac{1}{2}} (\x^{t+1}-\x^t) \right\|_2^2}{\left\| \nabla^2h(\x^{t})^{\tfrac{1}{2}} (\x^{t}-\x^{t-1}) \right\|_2}.
	\]
	Multiply $\frac{b}{a}$ and $\| \nabla^2h(\x^{t})^{\tfrac{1}{2}} (\x^{t}-\x^{t-1}) \|_2$ on both sides and take a square root. We have
	\[\begin{aligned}
		&\left\| \nabla^2h(\x^{t+1})^{\tfrac{1}{2}} (\x^{t+1}-\x^t) \right\|_2\\
		\leq&  \sqrt{\frac{b}{a}\left(\zeta\big(F(\x^t)-F(\x^*)\big) - \zeta\big(F(\x^{t+1})-F(\x^*)\big)\right)} \cdot\sqrt{\left\| \nabla^2h(\x^{t})^{\tfrac{1}{2}} (\x^{t}-\x^{t-1}) \right\|_2 } \\
		\leq& \frac{b}{2a}\left(\zeta\big(F(\x^t)-F(\x^*)\big) - \zeta\big(F(\x^{t+1})-F(\x^*)\big)\right)+\frac12\left\| \nabla^2h(\x^{t})^{\tfrac{1}{2}} (\x^{t}-\x^{t-1}) \right\|_2,
	\end{aligned} \]
	where the second inequality uses the Cauchy inequality.
	
	Sum the above inequality over $t=1,\ldots,k+1$. We obtain
	\[
	\begin{aligned}
		&\left\| \nabla^2h(\x^{k+2})^{\tfrac{1}{2}} (\x^{k+2}-\x^{k+1}) \right\|_2+\frac12\sum_{t=1}^{k}\left\| \nabla^2h(\x^{t+1})^{\tfrac{1}{2}} (\x^{t+1}-\x^t) \right\|_2\\
		\leq &\frac{b}{2a}\left(\zeta\big(F(\x^1)-F(\x^*)\big) - \zeta\big(F(\x^{k+2})-F(\x^*)\big)\right)+\frac12\left\| \nabla^2h(\x^{1})^{\tfrac{1}{2}} (\x^{1}-\x^{0}) \right\|_2\\
		\leq& \frac{b}{2a}\zeta\big(F(\x^0)-F(\x^*)\big)+\frac12\left\| \nabla^2h(\x^{1})^{\tfrac{1}{2}} (\x^{1}-\x^{0}) \right\|_2,
	\end{aligned}
	\]
	where the second inequality is due to (i) $\zeta(x)\geq0$ on $[0,\eta)$; (ii) $F(\x^1)\leq F(\x^0)$; and (iii) the monotonic increasing of $\zeta$. It follows that
	\[\begin{aligned}
		\sum_{t=0}^{k+1}\left\| \nabla^2h(\x^{t+1})^{\tfrac{1}{2}} (\x^{t+1}-\x^t) \right\|_2&\leq \frac{b}{a}\cdot\zeta\big(F(\x^0)-F(\x^*)\big)+2\left\| \nabla^2h(\x^{1})^{\tfrac{1}{2}} (\x^{1}-\x^{0}) \right\|_2.
	\end{aligned} \]
	Note that (H1) and $F(\x^1)\geq F(\x^*)$ yield
	\[\left\| \nabla^2h(\x^{1})^{\tfrac{1}{2}} (\x^{1}-\x^{0}) \right\|_2 \leq\sqrt{\frac{F(\x^0)-F(\x^1)}{a}}\leq \sqrt{\frac{F(\x^0)-F(\x^*)}{a}}.  \]
	We see that
	\[	\sum_{t=0}^{k+1}\left\| \nabla^2h(\x^{t+1})^{\tfrac{1}{2}} (\x^{t+1}-\x^t) \right\|_2\leq \frac{b}{a}\cdot\zeta\big(F(\x^0)-F(\x^*)\big)+ 2\sqrt{\frac{F(\x^0)-F(\x^*)}{a}} .\]
	Combine this with \eqref{stongconvexkernel}. We obtain
	\[ \sum_{t=0}^{k+1}\|\x^{t+1}-\x^t\|_2\leq \frac{b}{a\sqrt{\sigma}}\cdot\zeta\big(F(\x^0)-F(\x^*)\big)+ 2\sqrt{\frac{F(\x^0)-F(\x^*)}{a\sigma}}. \]
	The above inequality, together with \eqref{initialbound}, implies
	\[
	\|\x^{0} - \x^*\|_2 + \sum_{t=0}^{k+1}\|\x^{t+1}-\x^t\|_2 
	\le \|\x^{0}-\x^*\|_2 + 2\sqrt{\frac{F(\x^0)-F(\x^*)}{a\sigma}}+\frac{b}{a\sqrt{\sigma}}\cdot\zeta\big(F(\x^0)-F(\x^*)\big)< d,
	\]
	which proves \eqref{16} for $K=k+1$, thereby completing the induction. We conclude that  \eqref{16} holds for all $K \ge 0$. It follows that
	\[
	\|\x^{0}-\x^*\|_2 + \sum_{k=0}^{\infty} \|\x^{k+1}-\x^k\|_2 \;\le d < +\infty.
	\]
	Using the Cauchy convergence criterion and the condition that $\x^*$ is an accumulation point of  $\{\x^k\}_{k\geq0}$, we conclude that $\x^k $ converges to $\x^* $ with a finite length. Moreover, by (H3), $\x^*$ is a stationary point.
	We complete the proof.

	\section{Application to BPPM and BPGM}\label{sec:application}
	To apply Theorem \ref{th:general} to the BPMs, we verify (H1)---(H3) for sequences generated by them. As a preliminary step, we establish the relationship between different distances with the scaled norm $\| \nabla^2h(\x)^{\frac12}(\x-\y)\|_2$.
	\begin{proposition}\label{pro:BPM}
		Suppose that Assumption \ref{assum:general} (i) holds. Write $E_{r}(\x,\y):=\exp(r\|\nabla h(\x)-\nabla h(\y)\|_2)$. Then, for any bounded convex open subset $\cU\subseteq\inter(\dom(h))$, there exists $r>0$ such that for $\x,\y\in\cU$,
		\begin{enumerate}[label={{\rm (\roman*)}}]
			\item  $E_{-r}(\x,\y)\cdot \nabla^2h(\x) \preceq \nabla^2h(\y)\preceq E_{r}(\x,\y)\cdot \nabla^2h(\x)$;
			\item  $D_h(\x,\y)\geq \frac12E_{-r}(\x,\y)\cdot\| \nabla^2h(\x)^{\frac12} (\x-\y)\|^2_2$;
			\item  $ \|\nabla^2h(\x)^{-\frac12} (\nabla h(\x)-\nabla h(\y))\|_2\leq E_{r}(\x,\y)\cdot \| \nabla^2h(\x)^{\frac12}(\x-\y)\|_2$. 
		\end{enumerate}
	\end{proposition}
	\begin{proposition}\label{pro:BPMf}
		Consider the setting of Proposition \ref{pro:BPM}. Suppose that $H\in\cC^2(\cU)$, and $Lh+H$, $Lh-H$ are convex for some $L>0$. Then, for $\x,\y\in\cU$,
		\begin{equation}\label{eq:hessianF}
			\left\|\nabla^2h(\x)^{-\frac12} (\nabla H(\x)-\nabla H(\y)) \right\|_2\leq5L E_r(\x,\y)\left\| \nabla^2h(\x)^{\frac12}(\x-\y)\right\|_2. 
		\end{equation}
	\end{proposition}
	\begin{remark}\label{re:f1}
		The condition $H\in\cC^2(\cU)$ can be relaxed to $H\in\cC^1(\cU)$; see Sec. \ref{sec:remarkf1} for a proof.
	\end{remark}
	Proposition \ref{pro:BPM} (ii), along with the Bregman sufficient decrease for BPMs \eqref{eq:Bsd}, yields
	\[  F(\x^{k+1})-F(\x^k)\leq-\frac{L}2E_{-r}(\x^{k+1},\x^k)  \left\| \nabla^2h(\x)^{\frac12} (\x^{k+1}-\x^k)\right\|^2_2,\]
	which aligns with the scaled sufficient decrease (H1) provided that $E_{-r}(\x^{k+1},\x^k)$ is uniformly bounded for $k\geq0$. 
	Proposition \ref{pro:BPM} (iii) and Proposition \ref{pro:BPMf} link the scaled relative error (H2) to the optimality conditions of \eqref{eq:bppm} and \eqref{eq:bpgm}. Specifically, consider the optimality condition optimality conditions of the subproblems \eqref{eq:bppm}, \eqref{eq:bpgm}:
	\begin{equation}\label{eq:op}
		\nabla f\left(\x^{k+q}\right)+\p^{k+1}+\frac1{\alpha_k}\left(\nabla h\left(\x^{k+1}\right)-\nabla h\left(\x^k\right)\right) =\bz;\quad\p^{k+1}\in\partial g\left(\x^{k+1}\right),\tag{Op}
	\end{equation}
	where $q=1$ for the BPPM, and $q=0$ for the BPGM. 
	Multiply \eqref{eq:op} by the scaling matrix $\nabla^2h(\x^{k+1})^{-\frac12}$. Then, given that Proposition \ref{pro:BPM} (iii) and Proposition \ref{pro:BPMf} hold for $h$ and $f$, the triangle inequality yields 
	\[\begin{aligned}
		&\left\|\nabla^2h(\x^{k+1})^{-\frac12}\left(\nabla f(\x^{k+1})+\p^{k+1}\right)\right\|_2\\
		\leq &\left(\frac{1}{\alpha_k}+5L(1-q) \right)E_r(\x^{k+1},\x^k) \left\| \nabla^2h(\x^{k+1})^{\frac12} (\x^{k+1}-\x^k)\right\|_2  ,
	\end{aligned} \]
	which gives the scaled relative error inequality (H2) for the $k$-th iteration.
	
	With the above relationships, to ensure the uniform bound in (H1), (H2), a sufficient condition is that the ratio $E_r(\x^{k+1},\x^k)$ remains uniformly bounded for $k\geq0$. It amounts to the uniform boundedness of the kernel gradient gap $\nabla h(\x^{k+1})-\nabla h(\x^k)$. \citet{chen2026skl} established this bound for iterates $\x^k$ near a feasible point under linear equality constraints. However, such a guarantee might fail when $\dom(F)$ is an ill-posed subset of $\dom(h)$, as suggested by the following example.
	\begin{example}\label{example:fail}
		Consider $\cD=\B((1,1),1)\subseteq\R^2_+$, $h_S(\x)=\sum_{i=1}^2x_i\log(x_i)$, and $F(\x)=-x_2+\delta_{\cD}(\x)$ with $\x\in\R^2$. Let $\{\z^k\}_{k\geq0}\in\cD\cap\R^2_{++}$ be an arbitrary sequence converging to $\z^*=(0,1)$, and $\tilde{\z}^k$ be the next BPGM iterate at $\z^k$ with step size $1$, i.e.,
		\[ \tilde{\z}^k\coloneqq\argmin_{\z\in\cD}\left\{-z_2+D_{h_S}(\z,\z^k)\right\}. \]
		As we shall prove, one has $\|\nabla h_S(\tilde{\z}^k)-\nabla h_S(\z^k)\|_2\to+\infty$ in this setting, which implies that a uniform bound guarantee similar to \cite[Lemma 3]{chen2026skl} is impossible.
		
		\textit{Proof details:} By the definition of $\tilde{\z}^k$, we have $-\tilde{z}^k_2+D_{h_S}(\tilde{\z}^k,\z^k)\leq-z^k_2$ and further $D_{h_S}(\tilde{\z}^k,\z^k)\leq 4$. This, together with $z^k_1\to z^*_1=0$, implies $\tilde{z}^k_1\to0$. Observe that in the set $\cD$, $\z^*$ is the unique point with first coordinate equal to zero. We obtain $\tilde{\z}^k\to\z^*=(0,1)$. Combining it with $\z^k\to\z^*=(0,1)$ and the optimality condition 
		\[ \left(0,-1\right)+\tilde{\p}^k+\nabla h_S(\tilde{\z}^k)-\nabla h_S(\z^k)=\bz\]
		\[\text{with }\quad  \tilde{\p}^k\in\partial \delta_{\cD}(\tilde{\z}^k)=\cN_{\cD}(\tilde{\z}^k)=\{c(\tilde{z}^k_1-1,\tilde{z}^k_2-1):c\geq0\}, \]
		we see that $\tilde{p}^k_2\to1$. Let $\tilde{\p}^k=c_k(\tilde{z}^k_1-1,\tilde{z}^k_2-1)$ with $c_k\geq0$. We have $c_k=\tilde{p}^k_2/(\tilde{z}^k_2-1)\to+\infty$. As $\|\tilde{\p}^k\|_2\geq c_k|\tilde{z}^k_1-1|$ with $\tilde{z}^k_1\to0$, we further have  $\|\tilde{\p}^k\|_2\to+\infty$. Then, the optimality condition yields 
		\[\lim_{k\to\infty}\|\nabla h_S(\tilde{\z}^k)-\nabla h_S(\z^k)\|_2=\lim_{k\to\infty}\left\|\log\left(\frac{\tilde{\z}^k}{\z^k}\right)\right\|_2=+\infty.\]
	\end{example}
	To avoid the phenomenon illustrated in Example \ref{example:fail} and to guarantee the boundedness of $\{\nabla h(\x^{k+1})-\nabla h(\x^k)\}_{k\geq0}$, we impose well-posedness conditions on the domain $\dom(F)$.
	\begin{assumption}\label{assum:general2}
		Either $\dom(F)=\{\y:\A\x\leq\b\}$ with $\A\in\R^{m\times n}$, $\b\in\R^m$;  or $\dom(F)=\cF\cap\cl(\dom(h))$ with $\cF$ closed convex and $\cN_{\cF}(\x)\cap\spann\{\e_j:x_j\in\bd(\dom(\varphi))\}=\{\bz\}$ for $\x\in\dom(F)$.
	\end{assumption} 
	The above setting extends the linear equality constraints considered in \cite{chen2026skl} to general polyhedral constraints and, moreover, includes sets that satisfy a constraint qualification (CQ). Notably, when $\dom(F)=\{\x\in\dom(h):G_i(\x)\leq0,i\in[m]\}$ for some smooth functions $G_i$, the CQ in Assumption \ref{assum:general2} is strictly weaker than the linear independence constraint qualification (LICQ) commonly adopted in the variational analysis literature.
	
	With condition Assumption \ref{assum:general2} in place, we are now prepared to establish the uniform boundedness of the kernel gradient gap sequence; see Section \ref{appen:bound} in Appendix for its proof.
	\begin{proposition}\label{pro:bound}
		Suppose that Assumptions  \ref{assum:general}, \ref{assum:general2} hold. Let $\{\x^k\}_{k\geq0}$ be a bounded sequence satisfying \eqref{eq:bppm} or \eqref{eq:bpgm} with $\|\x^{k+1}-\x^k\|_2\to0$. Then, the sequence $\{\nabla h(\x^{k+1})-\nabla h(\x^{k})\}_{k\geq0}$ is bounded.
	\end{proposition}
	Proposition \ref{pro:bound} not only enables the verification of (H1) and (H2), but also ensures the validity of (H3) for the BPPM and BPGM sequences. Consequently, the general convergence Theorem \ref{th:general} applies to these sequences, leading to the main results of this section.
	\begin{theorem}[BPPM]\label{th:BPPM}
		Suppose that Assumptions  \ref{assum:general}, \ref{assum:general2} hold. Let $\{\x^k\}_{k\geq0}$ be a bounded BPPM sequence satisfying \eqref{eq:bppm} with $0<\underline{\alpha}\leq \alpha_k\leq \bar{\alpha}$. Then, {\rm (H1)---(H3)} hold. Further, if $F$ is an SK\L{} function under the kernel $h$, then  $\{\x^k\}_{k\geq0}$ converges to a stationary point of $F$ with a finite length. 
	\end{theorem}
	
	\begin{theorem}[BPGM]\label{th:BPGM}
		Suppose that Assumptions  \ref{assum:general}, \ref{assum:general2} hold, and $Lh+ F$, $Lh- F$ are convex for some $L>0$. Let $\{\x^k\}_{k\geq0}$ be a bounded BPGM sequence satisfying \eqref{eq:bpgm} with $0<\underline{\alpha}\leq \alpha_k\leq \bar{\alpha}<1/L$. Then, {\rm (H1)---(H3)} hold. 
		Further, if $F$ is an SK\L{} function under the kernel $h$, then  $\{\x^k\}_{k\geq0}$ converges to a stationary point of $F$ with a finite length. 
	\end{theorem}
	Theorems \ref{th:BPPM} and \ref{th:BPGM} complete our discrete-time convergence analysis for BPMs. Combined with Proposition \ref{pro:skl_general}, they provide, to the best of our knowledge, the first iterate convergence guarantee for BPPM and BPGM that applies to general closed kernels and composite objective functions. Consequently, we significantly improve the result in \cite{chen2026skl} that relies on specific kernel and problem structure.

	\subsection{Proof of Proposition \ref{pro:BPM}}
	We begin with a simple fact that will be repeatedly used.
	\begin{fact}\label{fact:ys}
		Let $h$ be a separable kernel with $\x,\y\in\inter(\dom(h))$. Let $\y^s\coloneqq \y+s(\x-\y)$ for $s\in[0,1]$. Then, we have
		\[\|\nabla h(\x)-\nabla h(\y^s)\|_2\leq \|\nabla h(\x)-\nabla h(\y)\|_2;\quad\forall~s\in[0,1].\]
		Consequently, one has $E_r(\x,\y^s)\leq E_r(\x,\y) $ and $E_{-r}(\x,\y^s)\geq E_{-r}(\x,\y) $ for all $r>0$.
	\end{fact}
	\begin{proof}[Proof of Fact \ref{fact:ys}]
		Since $\y^{s}$ lies in the line between $\x$ and $\y$, we have $|\varphi^{\prime}(x_i)-\varphi^{\prime}(y^{s}_i)|\leq|\varphi^{\prime}(x_i)-\varphi^{\prime}(y_i)|$ for $i\in[n]$ by the strict increasing of $\varphi^{\prime}$. By the separable structure $\nabla h(\z)=(\varphi^{\prime}(z_1),\ldots,\varphi^{\prime}(z_n))$, we further have $\|\nabla h(\x)-\nabla h(\y^s)\|_2\leq \|\nabla h(\x)-\nabla h(\y)\|_2$ as desired. This, along with the definition of $E_r$, yields $E_r(\x,\y^s)\leq E_r(\x,\y) $ and $E_{-r}(\x,\y^s)\geq E_{-r}(\x,\y)$ for $r>0$. 
	\end{proof}
	We are then ready to prove Proposition \ref{pro:BPM}.
	
	(i) Our first step is to show that for any bounded interval $\cV\subseteq\inter(\dom(\varphi))$,
	\begin{equation}\label{eq:self}
		|\varphi^{\prime\prime\prime}(x)| =\cO\left(\varphi^{\prime\prime}(x)^{2}\right)\quad{\rm on }\quad\cV.
	\end{equation}
	By Proposition \ref{pro:estimate}, for any $\bar{x}\in\bd(\dom(\varphi))$, there exists $\Delta_{\bar{x}}>0$ and an exponent $\alpha_{\bar{x}}\in(1,2]$ such that $|\varphi^{\prime\prime\prime}(x)|=\Theta(\varphi^{\prime\prime}(x)^{\alpha_{\bar{x}}} )$ on $[\bar{x}-\Delta_{\bar{x}},\bar{x}+\Delta_{\bar{x}}]\cap \inter(\dom(\varphi))$. Note that $\varphi^{\prime\prime}(x)=\Omega(1)$ on any bounded subset of $\inter(\dom(\varphi))$ by Assumption \ref{assum:general} (i) and strong convexity of $\varphi$. We see that for any $\bar{x}\in\bd(\dom(\varphi))$,
	\begin{equation}\label{eq:boundary}
		|\varphi^{\prime\prime\prime}(x)|=\Theta\left(\varphi^{\prime\prime}(x)^{\alpha_{\bar{x}}} \right)=\cO\left(\varphi^{\prime\prime}(x)^{2}\right)\quad{\rm on }\quad[\bar{x}-\Delta_{\bar{x}},\bar{x}+\Delta_{\bar{x}}]\cap \inter(\dom(\varphi)).
	\end{equation}   
	On the other hand, for any bounded interval $\cV\subseteq \inter(\dom(\varphi))$, the continuity of $\varphi^{\prime\prime\prime}$ on $\inter(\dom(\varphi))$ and $\varphi^{\prime\prime}(x)=\Omega(1)$ imply
	\[	|\varphi^{\prime\prime\prime}(x)|=\cO\left(\varphi^{\prime\prime}(x)^{2}\right)\quad{\rm on }\quad\tilde{\cV}=\cV\setminus\bigcup_{\bar{x}\in\bd(\dom(\varphi))}[\bar{x}-\Delta_{\bar{x}},\bar{x}+\Delta_{\bar{x}}], \]
	where the set $\tilde{\cV}$ is clearly a bounded subset of $\inter(\dom(\varphi))$ with $\cl(\tilde{\cV})\subseteq\inter(\dom(\varphi))$.
	
	The above estimate, together with \eqref{eq:boundary}, yields \eqref{eq:self}. That is, there exists a scalar $r_{\cV}>0$ such that $|\varphi^{\prime\prime\prime}(x)| \leq r_{\cV}\cdot \varphi^{\prime\prime}(x)^{2}$  on $\cV$. It follows that for any bounded interval $\cV\subseteq\inter(\dom(\varphi))$ and $x,y\in\cV$, 
	\begin{equation}\label{eq:self2}
		\begin{aligned}
			\left|\log(\varphi^{\prime\prime}(x))-\log(\varphi^{\prime\prime}(y))\right|&=\left|\int_{y}^x\frac{\varphi^{\prime\prime\prime}(s)}{\varphi^{\prime\prime}(s)}{\rm d}s\right|\\
			&\leq \left|\int_{y}^x\frac{\left|\varphi^{\prime\prime\prime}(s)\right|}{\varphi^{\prime\prime}(s)}{\rm d}s\right|\\
			&\leq r_{\cV}\left|\int_{y}^x\varphi^{\prime\prime}(s){\rm d}s\right|\\
			&=r_{\cV}\left|\varphi^{\prime}(x)-\varphi^{\prime}(y) \right|.
		\end{aligned}
	\end{equation}
	Now, given a bounded convex subset $\cU\subseteq\inter(\dom(h))$, we let $\bar{\cV}\subseteq\inter(\dom(\varphi))$ be a bounded interval such that $x_i\in\bar{\cV}$ for all $\x\in\cU$, $i\in[n]$. Let $r:=r_{\bar{\cV}}>0$.
	Then, for $\x,\y\in\cU$, $i\in[n]$, the inequality \eqref{eq:self2} yields
	\[\begin{aligned}	
		&\frac{\varphi^{\prime\prime}(y_i)}{\varphi^{\prime\prime}(x_i)}\leq\exp\left(r\left|\varphi^{\prime}(x_i)-\varphi^{\prime}(y_i)\right|\right)\leq E_{r}(\x,\y);\\
		&\frac{\varphi^{\prime\prime}(y_i)}{\varphi^{\prime\prime}(x_i)}\geq\exp\left(-r\left|\varphi^{\prime}(x_i)-\varphi^{\prime}(y_i)\right|\right)\geq E_{-r}(\x,\y),
	\end{aligned}   \]
	which lead to the desired inequality on the Hessian $\nabla^2h(\y)=\Diag(\varphi^{\prime\prime}(y_1),\ldots,\varphi^{\prime\prime}(y_n))$.
	
	(ii) Write $\y^{s_1}=\y+s_1(\x-\y)$, $\y^{s_1,s_2}=\y+s_2(\y^{s_1}-\y)$ for $s_1,s_2\in[0,1]$. Note $h(\x)-h(\y)=\int_0^1\left<\nabla h(\y^{s_1}),\x-\y\right>{\rm d}s_1$. We have
	\begin{equation}\label{eq:Dhint}
		\begin{aligned}
			D_h(\x,\y)&= h(\x)-h(\y)-\left<\nabla h(\y),(\x-\y)\right> \\
			&=\int_0^1\left<\nabla h(\y^{s_1})-\nabla h(\y),\x-\y\right>{\rm d}s_1 \\
			&=\int_0^1s_1\left<\int_0^1\nabla h^2\left(\y^{s_1,s_2}\right){\rm d}s_2\cdot(\x-\y),\x-\y\right>{\rm d}s_1. \\
		\end{aligned} 
	\end{equation}
	We then estimate the Hessian inside.
	Fact \ref{fact:ys} yields $E_{-r}(\x,\y^{s_1,s_2})\geq E_{-r}(\x,\y^{s_1})\geq E_{-r}(\x,\y)>0$.
	Then, the result in (i) implies 
	\[ \nabla h^2\left(\y^{s_1,s_2}\right)\succeq E_{-r}(\x,\y^{s_1,s_2})\nabla^2h(\x)\succeq E_{-r}(\x,\y)\nabla^2h(\x). \] 
	It follows that $\int_0^1\nabla h^2(\y^{s_1,s_2}){\rm d}s_2\succeq E_{-r}(\x,\y)\nabla^2h(\x)$. Revisiting \eqref{eq:Dhint}, we obtain
	\[\begin{aligned}
		D_h(\x,\y)&\geq\int_0^1s_1{\rm d}s_1\cdot E_{-r}(\x,\y) \left<\nabla^2h(\x)(\x-\y),\x-\y\right>\\
		&=\frac12E_{-r}(\x,\y)\cdot\left\| \nabla^2h(\x)^{\frac12} (\x-\y)\right\|^2_2. 
	\end{aligned}\]
	(iii) Write $\y^s=\y+s(\x-\y)$ for $s\in[0,1]$. Fact \ref{fact:ys} gives $E_r(\x,\y^s)\leq E_r(\x,\y)$. By the result in (i), we have
	\begin{equation}\label{eq:Er}
		\nabla^2h(\y^s)\preceq E_r(\x,\y^s)\nabla^2h(\x)\preceq E_r(\x,\y)\nabla^2h(\x).
	\end{equation}
	We then estimate $\|\nabla^2h(\x)^{-\frac12} (\nabla h(\x)-\nabla h(\y))\|^2_2$:
	\begin{equation}\label{eq:h_h2}
		\begin{aligned}
			&\left\|\nabla^2h(\x)^{-\frac12} (\nabla h(\x)-\nabla h(\y))\right\|^2_2\\
			=&\left\| \nabla^2h(\x)^{-\frac12}\int_0^1\nabla^2h(\y^s)(\x-\y){\rm d}s   \right\|^2_2 \\
			\leq& \int_0^1\left\| \nabla^2h(\x)^{-\frac12}\nabla^2h(\y^s)(\x-\y)\right\|^2_2{\rm d}s\\
			=&\int_0^1\left<\nabla^2h(\x)^{-1}\nabla^2h(\y^s)(\x-\y),\nabla^2h(\y^s)(\x-\y)\right>{\rm d}s,
		\end{aligned}
	\end{equation}
	where the inequality is due to Jensen's inequality.
	
	The above estimate, along with \eqref{eq:Er} and its corollary $\nabla^2h(\x)^{-1}\preceq E_r(\x,\y)\nabla^2h(\y^s)^{-1}$, further implies
	\[ \begin{aligned}
		&\left\|\nabla^2h(\x)^{-\frac12} (\nabla h(\x)-\nabla h(\y))\right\|^2_2\\
		\leq& E_r(\x,\y)\int_0^1\left<\nabla^2h(\y^s)^{-1}\nabla^2h(\y^s)(\x-\y),\nabla^2h(\y^s)(\x-\y)\right>{\rm d}s\\
		=&E_r(\x,\y)\int_0^1\left<\x-\y,\nabla^2h(\y^s)(\x-\y)\right>{\rm d}s\\
		\leq& E_r(\x,\y)^2\int_0^1\left<\x-\y,\nabla^2h(\x)(\x-\y)\right>{\rm d}s\\
		=&  E_r(\x,\y)^2 \left\| \nabla^2h(\x)^{\frac12}(\x-\y)\right\|^2_2.
	\end{aligned}
	\]
	By taking a square root in both sides, we obtain the desired inequality.
	
	\subsection{Proof of Proposition \ref{pro:BPMf}}\label{sec:BPMf}
	Let $G=2Lh+H$. Since $Lh+H$ is convex, we have
	$\nabla^2 G=2L\nabla^2h+\nabla^2H\succeq L\nabla^2h\succ\bz.$ 
	By the triangle inequality and $ \nabla H=\nabla G-2L\nabla h$, it is clear that
	\begin{equation}\label{eq:h_F}
		\begin{aligned}
			&\left\|\nabla^2h(\x)^{-\frac12} (\nabla H(\x)-\nabla H(\y)) \right\|_2\\
			\leq& \left\|\nabla^2h(\x)^{-\frac12} (\nabla G(\x)-\nabla G(\y)) \right\|_2+ 2L\left\|\nabla^2h(\x)^{-\frac12} (\nabla h(\x)-\nabla h(\y)) \right\|_2\\
			\leq&\left\|\nabla^2h(\x)^{-\frac12} (\nabla G(\x)-\nabla G(\y)) \right\|_2+ 2LE_r(\x,\y)\left\|\nabla^2h(\x)^{\frac12} (\x-\y) \right\|_2,\\
		\end{aligned}
	\end{equation}
	where the second inequality is due to Proposition \ref{pro:BPM} (iii).
	
	We then estimate $\|\nabla^2h(\x)^{-\frac12} (\nabla G(\x)-\nabla G(\y))\|_2$. Write $\y^s=\y+s(\x-\y)$ for $s\in[0,1]$. Fact \ref{fact:ys} and Proposition \ref{pro:BPM} (i) give $\nabla^2h(\y^s)\preceq E_r(\x,\y)\nabla^2h(\x)$. Similar to \eqref{eq:h_h2}, we further have 
	\[\begin{aligned}
		&\left\|\nabla^2h(\x)^{-\frac12} (\nabla G(\x)-\nabla G(\y))\right\|^2_2\\
		=&\left\| \nabla^2h(\x)^{-\frac12}\int_0^1\nabla^2G(\y^s)(\x-\y){\rm d}s   \right\|^2_2 \\
		\leq& \int_0^1\left\| \nabla^2h(\x)^{-\frac12}\nabla^2G(\y^s)(\x-\y)\right\|^2_2{\rm d}s\\
		\leq&\int_0^1\left<\nabla^2h(\x)^{-1}\nabla^2G(\y^s)(\x-\y),\nabla^2G(\y^s)(\x-\y)\right>{\rm d}s.
	\end{aligned}
	\]
	Note that $G=2Lh+H$, the convexity of $Lh-H$, and $\nabla^2h(\y^s)\preceq  E_r(\x,\y)\nabla^2h(\x)$ imply
	\[\nabla^2G(\y^s)=2L\nabla^2h(\y^s)+\nabla^2H(\y^s)\preceq 3L\nabla^2h(\y^s)\preceq 3L E_r(\x,\y)\nabla^2h(\x). \]
	Since $\nabla^2G,\nabla^2h\succ\bz$, it follows that $\nabla^2h(\x)^{-1}\preceq 3L  E_r(\x,\y)\nabla^2G(\y^s)^{-1}$. We further have
	\[\begin{aligned}
		&\left\|\nabla^2h(\x)^{-\frac12} (\nabla G(\x)-\nabla G(\y))\right\|^2_2\\
		\leq& 3L  E_r(\x,\y)\int_0^1\left<\nabla^2G(\y^s)^{-1}\nabla^2G(\y^s)(\y-\x),\nabla^2G(\y^s)(\x-\y)\right>{\rm d}s\\
		=& 3L  E_r(\x,\y)\int_0^1\left<\y-\x,\nabla^2G(\y^s)(\x-\y)\right>{\rm d}s\\
		\leq&\left(3L E_r(\x,\y)\right)^2\int_0^1\left<\y-\x,\nabla^2h(\x)(\x-\y)\right>{\rm d}s\\
		=& \left(3L E_r(\x,\y)\right)^2\left\|\nabla^2h(\x)^{\frac12} (\x-\y) \right\|^2_2.
	\end{aligned}\]
	Taking a square root in both sides, we obtain
	\[\left\|\nabla^2h(\x)^{-\frac12} (\nabla G(\x)-\nabla G(\y))\right\|_2
	\leq 3L E_r(\x,\y) \left\|\nabla^2h(\x)^{\frac12} (\x-\y) \right\|_2.  \]
	This, together with \eqref{eq:h_F}, yields the desired inequality.
	
	\subsection{Proof of Remark \ref{re:f1}}\label{sec:remarkf1}
	Given that the condition $H\in\cC^2(\cU)$ in Proposition \ref{pro:BPMf} is relaxed to $H\in\cC^1(\cU)$, to apply the arguments in Sec. \ref{sec:BPMf} to proving \eqref{eq:hessianF}, we need the following claim:
	
	\textbf{Claim 1.}  The gradient $\nabla H$ is locally Lipschitz continuous on $\cU$. Furthermore, for almost every pair $(\x,\y)\in\cU\times\cU$, $\nabla^2H(\y^s)$ exists for almost every $s\in[0,1]$, where $\y^s=\y+s(\x-\y)$. 
	
	Observe that when Claim 1 holds, one has
	\[\nabla H(\x)-\nabla H(\y)=\int_0^1\nabla^2H(\y^s)(\x-\y){\rm d}s ~\text{ for almost every }~(\x,\y)\in\cU\times\cU.\] This enables the arguments in Sec. \ref{sec:BPMf}  to establish \eqref{eq:hessianF} for almost every $(\x,\y)\in\cU\times\cU$. Then, by its continuity w.r.t. $(\x,\y)$, \eqref{eq:hessianF}  holds for all $\x,\y\in\cU$. 
	Therefore, to establish Remark \ref{re:f1}, it suffices to prove Claim 1. 
	
	We first show that $\nabla H$ is locally Lipschitz continuous on $\cU$.
	Note that the convexity of $Lh+F$ and $Lh-F$ implies
	\begin{equation}\label{eq:convexity}
		-L D_h(\y,\z)
		\le H(\y)-H(\z)-\langle \nabla H(\z),\y-\z\rangle
		\le L D_h(\y,\z), \quad\forall~\y,\z\in\cU.
	\end{equation}
	Restrict $\y,\z$ in $\B(\x,r_{\x})$ for some $\x\in\cU$, $r_{\x}>0$ so that $\B(\x,r_{\x})\subseteq\inter(\dom(h))$. Then, the continuity of $\nabla^2h$ ensures $\nabla^2h(\y)\preceq a_{\x}\bI$ for all $\y\in\B(\x,r_{\x})$ for some $a_{\x}>0$. It follows that 
	\[ D_h(\y,\z)\leq \frac{a_{\x}}{2}\left\|\y-\z\right\|^2_2,\qquad\forall~\y,\z\in\B(\x,r_{\x}). \]
	This, together with \eqref{eq:convexity}, implies that for all $\y,\z\in\cU\cap\B(\x,r_{\x})$,
	\[ 	-\frac{La_{\x}}{2}\left\|\y-\z\right\|^2_2
	\le H(\y)-H(\z)-\langle \nabla H(\z),\y-\z\rangle
	\le \frac{La_{\x}}{2}\left\|\y-\z\right\|^2_2. \]
	These inequalities ensure that the functions $\y\mapsto {La_{\x}\|\y\|^2_2}/2+H(\y)$, $\y\mapsto {La_{\x}\|\y\|^2_2}/2-H(\y)$ are convex near $\x$ for every $\x\in\cU$. Hence, $H$ and $-H$ are lower-$C^2$ on $\cU$ by \cite[Theorem 10.33]{rockafellar2009variational}, i.e., $H$ is both lower- and upper-$C^2$ on $\cU$. Thanks to \cite[Theorem 13.34]{rockafellar2009variational}, we know that $\nabla H$ is strictly continuous (locally Lipschitz continuous) on the open set $\cU$.
	
	The remaining task is showing that for almost every $(\x,\y)\in\cU\times\cU$, the Hessian $\nabla^2H(\y^s)$ exists a.e. on $[0,1]$. 
	
	By Rademacher’s theorem (see, e.g., \cite[Theorem 3.2]{evans2025measure}),  $\nabla H$ is differentiable a.e. on $\cU$. That is, $\nabla^2 H$ exists a.e. on $\cU$. 
	Let $\cV$ denote the subset of $\cU$ where $\nabla^2 H$ does not exist, and $\cL^n$ denote the $n$-dimensional Lebesgue measure. Then, we have $\cL^n(\cV)=0$. 
	
	Define the function $G:\cU\times\cU\times(0,1)\to\R^n$  and the set $\cS\subseteq \cU\times\cU\times(0,1)$ by
	\[G(\x,\y,s)=\y+s(\x-\y);\qquad \cS\coloneqq G^{-1}(\cV).  \]
	It is clear that $G$ is a continuous measurable function and hence $\cS$ is a measurable set.
	Our strategy is to show $\cL^{2n+1}(\cS)=0$, which, together with Fubini's theorem and  $\cS\subseteq \cU\times\cU\times(0,1)$, would imply \[0=\cL^{2n+1}(\cS)=\int_{\cU\times\cU\times(0,1)}\hat{\delta}_{\cS}(\x,\y,s){\rm d}(\x,\y,s)=\int_{\cU\times\cU}\left(\int_0^1\hat{\delta}_{\cS}(\x,\y,s){\rm d}s\right){\rm d}(\x,\y),\]
	where $\hat{\delta}_{\cS}(\x,\y,s)=1$ if $(\x,\y,s)\in\cS$; and $\hat{\delta}_{\cS}(\x,\y,s)=0$ otherwise.
	
	Then, it would follow that for almost every $(\x,\y)\in\cU\times\cU$, $\hat{\delta}_{\cS}(\x,\y,s)=0$ for almost every $s\in(0,1)$. 
	Note that by definition we have
	\[\hat{\delta}_{\cS}(\x,\y,s)=0\Longleftrightarrow (\x,\y,s)\notin\cS\Longleftrightarrow \y^s\notin \cV\Longleftrightarrow \nabla^2 H(\y^s)\text{ exists}.\]
	We would complete the proof of Claim 1 if $\cL^{2n+1}(\cS)=0$ was established.
	
	To show $\cL^{2n+1}(\cS)=0$, we consider the coarea formula \cite[Theorem 3.10]{evans2025measure}. Observe that $G$ is a Lipschitz continuous function due to the boundedness of $\cU$ and $(0,1)$. Kirszbraun's theorem (see, e.g., \cite[Theorem 3.1]{evans2025measure}) ensures that there is a Lipschitz continuous extension of $G$ to $\R^{2n+1}$, denoted by $\tilde{G}$. Clearly, we have 
	\begin{equation}\label{eq:tildeG}
		\nabla \tilde{G}(\x,\y,s)=\nabla G(\x,\y,s)\quad\text{ on }\quad\cU\times\cU\times(0,1)\supseteq\cS.
	\end{equation}
	Let $J\tilde{G}(\x,\y,s)$ be the Jacobian of $\tilde{G}$ at $(\x,\y,s)\in\cS$ in the sense of \cite[Definition 3.4 and Theorem 3.6]{evans2025measure}, i.e.,
	\begin{equation}
		J\tilde{G}(\x,\y,s)=\sqrt{\det(\nabla \tilde{G}(\x,\y,s)\nabla \tilde{G}(\x,\y,s)^{\top})}.\label{eq:Jaco}
	\end{equation}
	The coarea formula \cite[Theorem 3.10]{evans2025measure} asserts
	\begin{equation}\label{eq:area}
		\int_{\cS} J \tilde{G}(\x,\y,s){\rm d}(\x,\y,s)
		= \int_{\R^n} \cH^{n+1}\left(\cS\cap \tilde{G}^{-1}(\z)\right){\rm d}\z,
	\end{equation}
	where $\cH^{n+1}$ denotes the $(n+1)$-dimensional Hausdorff measure, and $ \cH^{n+1}(\cS\cap \tilde{G}^{-1}(\z))$ is bounded for $\z\in\R^n$ due to the boundedness of $\cS$. 
	
	Observe that $J \tilde{G}(\x,\y,s)=J_n {G}(\x,\y,s)$ on $\cS$ due to \eqref{eq:tildeG} and the definition \eqref{eq:Jaco}. Moreover,  by the definitions of $\tilde{G}$ and $\cS$, we have 
	\[\cS\cap \tilde{G}^{-1}(\z)=G^{-1}\left(\cV\right)\cap \tilde{G}^{-1}(\z) =G^{-1}\left(\cV\right)\cap {G}^{-1}(\z),\]
	which is nonempty only if $\z\in\cV$. That is, $\cH^{n+1}(\cS\cap \tilde{G}^{-1}(\z))>0$ only if $\z\in\cV$. We see that \eqref{eq:area} amounts to
	\begin{equation}\label{eq:area2}
		\int_{\cS} J {G}(\x,\y,s){\rm d}(\x,\y,s)
		= \int_{\cV} \cH^{n+1}\left(\cS\cap \tilde{G}^{-1}(\z)\right){\rm d}\z=0, 
	\end{equation}
	where the second equality is due to $\cL^n(\cV)=0$ and the boundedness of $\cH^{n+1}(\cS\cap \tilde{G}^{-1}(\z))$.
	
	Finally, note that $\nabla G(\x,\y,s)=\begin{bmatrix}
		s\bI_n&(1-s)\bI_n&\x-\y
	\end{bmatrix}$ is a linear map with full rank on $\cU\times\cU\times(0,1)\supseteq\cS$. We see that $J {G}(\x,\y,s)=\det(\nabla G(\x,\y,s))=\sqrt{\det(\nabla G(\x,\y,s)\nabla G(\x,\y,s)^{\top})}>0$ on $\cS$ by \cite[Theorem 3.6]{evans2025measure}, and further $J {G}$ is continuous and thus measurable. These, combined with \eqref{eq:area2}, yield $\cL^{2n+1}(\cS)=0$. We complete the proof.
	
	\subsection{Proof of Proposition \ref{pro:bound}}\label{appen:bound}
	Suppose that there is a subsequence $\{k_l\}_{l\geq0}\subseteq\N$ such that $\{\nabla h(\x^{k_l+1})-\nabla h(\x^{k_l})\}_{l\geq0}$ is unbounded. Our strategy is to yield a contradiction through the optimality conditions \eqref{eq:op}.
	
	Observe that the boundedness of $\{\x^k\}_{k\geq0}$ and continuity of $\nabla f$ ensure the boundedness of $\{\nabla f(\x^{k+q})\}_{k\geq0}$. This, together with \eqref{eq:op} and $\|\nabla h(\x^{k_l+1})-\nabla h(\x^{k_l})\|_2\to\infty$, implies
	$\|\p^{k_l+1}\|_2\to\infty$.
	By passing to a subsequence if necessary, we may assume that 
	\begin{equation}\label{eq:px}
		\x^{k_l}\to\x^*\in\dom(F);\qquad \frac{\p^{k_l+1} }{\|\p^{k_l+1}\|_2}\to\p^*. 
	\end{equation}
	We then characterize $\p^k\in\partial g(\x^k)$. We first present the following lemma to characterize the subdifferential of a convex function that is Lipschitz continuous on its domain.
	\begin{lemma}\label{le:f2}
		Let $H:\R^n\to\overline{\R}$ be a convex function that is $L_H$-Lipschitz continuous on the nonempty domain $\dom(H)$ with $L_H>0$. Then, we have
		\[\partial H(\x)\subseteq \B(\bz,L_H)+\cN_{\dom(H)}(\x),\qquad\forall~\x\in\dom(H).\]
	\end{lemma}
	\begin{proof}[Proof of Lemma \ref{le:f2}]
		Thanks to \cite[Theorem 1]{cobzas1978norm}, we know that there is an extension function $\tilde{H}:\R^n\to\R$ that is proper convex and $L_H$-Lipschitz continuous on $\R^n$ such that $\tilde{H}(\x)=H(\x)$ for $\x\in\dom(H)$. In particular, we have
		\[ H=\tilde{H}+\delta_{\dom(H)}. \]
		Clearly, $\inter(\dom(\tilde{H}))\cap\dom(H)=\R^n\cap\dom(H)=\dom(H)\neq\emptyset$. By the subdifferential sum rule \cite[Corollary 10.9]{rockafellar2009variational}, we have
		\[\partial H(\x)=\partial \tilde{H}(\x)+\partial \delta_{\dom(H)}(\x)=\partial \tilde{H}(\x)+\cN_{\dom(H)}(\x),\qquad\forall~\x\in\dom(H),\]
		where the second equality is due to the convexity of $\dom(H)$ and \cite[Exercise 8.14]{rockafellar2009variational}.
		
		The above equality, together with $\partial \tilde{H}(\x)\subseteq\B(\bz,L_H)$, which is ensured by the $L_H$-Lipschitz continuity of $\tilde{H}$ on $\R^n$, yields the desired inclusion result.
	\end{proof}
	We then apply Lemma \ref{le:f2} to the function $g$ restricted in a bounded region.
	Since $\{\x^k\}_{k\geq0}$ is bounded and $g$ is locally Lipschitz on $\dom(g)$, there exist scalars $R,L_g>0$ such that $\x^k\in\inter(\B(\bz,R))$, $k\geq0$, and $g+\delta_{\B(\bz,R)}$ is $L_g$-Lipschitz continuous on $\dom(g)\cap\B(\bz,R)$. By Lemma \ref{le:f2}, we have the following inclusion:
	\[ \partial\left(g+\delta_{\B(\bz,R)}\right)(\x)\subseteq \B(\bz,L_g)+\cN_{\dom(g)\cap \B(\bz,R)}(\x),\qquad\forall~\x\in\dom(g)\cap\B(\bz,R). \]
	Note that $\x^k\in\inter(\B(\bz,R))\cap\dom(g)$. We have $\partial\delta_{\B(\bz,R)}(\x^k)=\cN_{\B(\bz,R)}(\x^k)=\{\bz\}$. Further, the sum subdifferential rule \cite[Corollary 10.9]{rockafellar2009variational} and intersection rule for normal cones \cite[Theorem 6.42]{rockafellar2009variational} imply that for $k\geq0$, 
	\[\begin{aligned}
		\partial\left(g+\delta_{\B(\bz,R)}\right)(\x^k)&=	\partial g(\x^k)+\{\bz\}=\partial g(\x^k);\\
		\cN_{\dom(g)\cap \B(\bz,R)}(\x^k)&=\cN_{\dom(g)}(\x^k)+\{\bz\}=\cN_{\dom(g)}(\x^k).
	\end{aligned}\]
	These, together with the previous inclusion, yield $\partial g(\x^k)\subseteq \B(\bz,L_g)+\cN_{\dom(g)}(\x^k)$, i.e.,
	\begin{equation}\label{eq:decompose}
		\p^k=\hat{\p}^k+\d^k\qquad\text{ with }\quad  \|\hat{\p}^k\|\leq L_g,~\d^k\in\cN_{\dom(g)}(\x^k),\quad\forall~k\geq0.
	\end{equation}
	Next, we separately analyze $\p^*$ under two cases of Assumption \ref{assum:general2} to yield a contradiction.
	
	\subsubsection{Case 1. The domain is a polyhedral.}
	In this case, $\dom(g)=\dom(F)=\{\y:\A\x\leq\b\}$ for some $\A=(\a_1,\ldots,\a_m)^{\top}\in\R^{m\times n}$, $\b\in\R^m$. Define the active index set $\cA(\x)=\{i:(\A\x-\b)_i=0\}$. By \cite[Theorem 6.46]{rockafellar2009variational}, we have 
	\[\cN_{\dom(g)}(\x)= \left\{\A^{\top}\blam:\blam\in\R^m_+,\lambda_{i}^k=0,i\notin\cA(\x)\right\}. \]
	This, together with \eqref{eq:decompose}, yields
	\begin{equation}\label{eq:lambda}
		\p^k=\hat{\p}^k+\A^{\top}\blam^k,\quad \|\hat{\p}^k\|_2\leq L_g,\quad \blam^k\in\R^m_+,\quad\lambda_{i}^k=0,\quad i\notin\cA(\x^k).  
	\end{equation}
	By passing to a subsequence if necessary, we may assume that $\cA(\x^{k_l+1})\equiv\cA_0$ for some $\cA_0\subseteq[m]$. Let $\cB_0=[n]\setminus\cA_0$.  We have  $\blam^{k_l+1}_{\cB_0}\equiv\bz$, and hence
	\[\A^{\top}\blam^{k_l+1}\in\cD_0:=\left\{\A^{\top}\blam:\blam\in\R^m_+,\blam_{\cB_0}=\bz\right\}.\] 
	Then, the convergence \eqref{eq:px} and the decomposition \eqref{eq:lambda} yield
	\[\p^*=\lim_{k\to\infty}\frac{\hat{\p}^{k_l+1}+\A^{\top}\blam^{k_l+1}}{\|\p^{k_l+1}\|_2}=\lim_{k\to\infty}\frac{\A^{\top}\blam^{k_l+1}}{\|\p^{k_l+1}\|_2}\in\cD_0,\]
	where the second equality is due to $\|\hat{\p}^{k_l+1}\|_2\leq L_g$; the inclusion uses $\A^{\top}\blam^{k_l+1}/\|\p^{k_l+1}\|_2\in\cD_0$ and  the closedness of $\cD_0$ .
	
	Hence, there is a vector $\blam^*\in\R^m_+$ with $\blam^*_{\cB_0}=\bz$ such that
	\[\p^*=\A^{\top}\blam^*.\]
	Combining $\blam^*_{\cB_0}=\bz$ and $(\A\x^{k_l+1}-\b)_{\cA_0}=\bz$, we obtain
	\[ \left<\blam^*,\A\x^{k_l+1}-\b\right>=\left<\blam^*_{\cA_0},(\A\x^{k_l+1}-\b)_{\cA_0}\right>+\left<\blam^*_{\cB_0},(\A\x^{k_l+1}-\b)_{\cB_0}\right> =0.\]
	Note that $\left<\blam^*,\A\x^{k_l}-\b\right>\leq0$ by $\blam^*\geq\bz$ and $\x^{k_l}\in\dom(F)$, i.e., $\A\x^{k_l}-\b\leq\bz$. We see that
	\begin{equation}\label{eq:contra1}
		\left<\p^*,\x^{k_l+1}-\x^{k_l}\right>=\left<\A^{\top}\blam^*,\x^{k_l+1}-\x^{k_l}\right>=\left<\blam^*,\A\x^{k_l+1}-\b\right>-\left<\blam^*,\A\x^{k_l}-\b\right>\geq0.
	\end{equation}
	To yield a contradiction to \eqref{eq:contra1}, we define the index sets
	\[\cI^*_+:=\left\{i:p^*_i>0\right\},\qquad   \cI^*_-:=\left\{i:p^*_i<0\right\}.  \]
	Clearly, $\cI^*_+\cup\cI^*_-\neq\emptyset$ since $\p^*\neq\bz$ with $\|\p^*\|_2=1$. Then, the convergence \eqref{eq:px} implies $p^{k_l+1}_i\to+\infty$ for $i\in\cI^*_+$; and  $p^{k_l+1}_i\to-\infty$ for $i\in\cI^*_-$. Recall the boundedness of $\{\nabla f(\x^k)\}_{k\geq0}$. We know that there exists an index $K>0$ such that for $k_l\geq K$,
	\[	\nabla_i f(\x^{k_l+q})+p^{k_l+1}_i>0,~\quad\forall~i\in\cI^*_+; \qquad	\nabla_i f(\x^{k_l+q})+p^{k_l+1}_i<0,\quad\forall~i\in\cI^*_-; \]
	This, together with the equality \eqref{eq:op}, the separable structure $\nabla h(\z)=(\varphi^{\prime}(z_1),\ldots,\varphi^{\prime}(z_n))$, and strict increasing of $\varphi^{\prime}$, implies that for $k_l\geq K$,
	\[  x^{k_l+1}_i<x^{k_l}_i,\quad\forall~i\in\cI^*_+; \qquad x^{k_l+1}_i>x^{k_l}_i,\quad\forall~i\in\cI^*_-.  \]
	Recall $p^*_i>0$ for $i\in\cI^*_+$; and $p^*_i<0$ for $i\in\cI^*_-$. It follows that 
	\[\left<\p^*,\x^{k_l+1}-\x^{k_l}\right><0,\]
	which contradicts \eqref{eq:contra1}.
	
	\subsubsection{Case 2. The domain is not a polyhedral.}
	In this case, Assumption \ref{assum:general2} assumes that $\dom(g)=\dom(F)=\cF\cap{\cl(\dom(h))}$ with $\cF$ closed convex and
	\begin{equation}\label{eq:contrad2}
		\cN_{\cF}(\x)\cap\spann\{\e_j:x_j\in\bd(\dom(\varphi))\}=\{\bz\},\qquad\forall~\x\in\dom(F).
	\end{equation} 
	By the intersection rule for normal cones \cite[Theorem 6.42]{rockafellar2009variational} and $\x^k\in\dom(F)\cap\inter(\dom(h))$, we have
	\[ \cN_{\dom(g)}(\x^k)=\cN_{\cF}(\x^k)+\partial\delta_{\cl(\dom(h))}(\x^k)=\cN_{\cF}(\x^k)+\{\bz\}=\cN_{\cF}(\x^k).  \]
	Recall $\d^k\in\cN_{\dom(g)}(\x^k)$ and $\|\hat{\p}^k\|_2\leq L_g$ by \eqref{eq:decompose}. We have $\d^k\in\cN_{\cF}(\x^k)$ and $\d^k/\|\p^k\|_2\in \cN_{\cF}(\x^k)$. The convergence \eqref{eq:px} yields
	\[ \p^*=\lim_{k\to\infty}\frac{\hat{\p}^{k_l+1}+\d^{k_l+1}}{\|\p^{k_l+1}\|_2}=\lim_{k\to\infty}\frac{\d^{k_l+1}}{\|\p^{k_l+1}\|_2}\in\cN_{\cF}(\x^*),\]
	where the inclusion is due to the outer semi-continuity of normal cone \cite[Proposition 6.6]{rockafellar2009variational}.
	
	On the other hand, let $k=k_l$ in the optimality condition \eqref{eq:op}, divide $\|\p^{k_l+1}\|_2$ in both sides, and take $l\to\infty$. We have
	\[\p^*=\lim_{l\to\infty} \frac{\nabla h(\x^{k_l})-\nabla h(\x^{k_l+1})}{\alpha_{k_l}\|\p^{k_l+1}\|_2}=\lim_{l\to\infty}\frac{\nabla h(\x^{k_l})-\nabla h(\x^{k_l+1})}{\|\nabla h(\x^{k_l})-\nabla h(\x^{k_l+1})\|_2} . \]
	Let $\cI=\{i:x^*_i\in\inter(\dom(\varphi))\}$. The convergence $\x^{k_l}\to\x^*$ and the condition  $\|\x^{k+1}-\x^k\|_2\to0$ imply $\varphi_i(x^{k_l})-\varphi_i(x^{k_l+1})\to0$ for $i\in\cI$. It follows that 
	\[\p^*_{\cI}=\bz,\quad {\rm i.e.}, \quad\p^*\in\spann\left\{\e_j:x^*_j\in\bd(\dom(\varphi))\right\}.\]
	This, together with $\p^*\in\cN_{\cF}(\x^*)$ and $\|\p^*\|_2=1$, yields
	\[ \p^*\in\cN_{\cF}(\x^*)\cap\spann\left\{\e_j:x^*_j\in\bd(\dom(\varphi))\right\};\qquad \p^*\neq\bz. \]
	This contradicts \eqref{eq:contrad2}. We complete the proof.
	
	\subsection{Proof of Theorem \ref{th:BPPM}}\label{sec:BPPM}
	We verify (H1)---(H3) in order. Under (H1)---(H3), the convergence result directly follows from Theorem \ref{th:general}.
	
	\subsubsection{Verifying (H1) for BPPM sequence.} We begin with the subproblem \eqref{eq:bppm}, which ensures the following sufficient decrease inequality:
	\begin{equation}\label{eq:decrease1}
		\begin{aligned}
			F(\x^k)= F(\x^k)+\frac{1}{\alpha_k}D_h(\x^{k},\x^k)&\geq F(\x^{k+1})+\frac{1}{\alpha_k}D_h(\x^{k+1},\x^k)\\
			&\geq F(\x^{k+1})+\frac{1}{\bar{\alpha}}D_h(\x^{k+1},\x^k),
		\end{aligned}
	\end{equation}
	where the second inequality is due to $\alpha_k\leq \bar{\alpha}$.
	
	Note that the strong convexity of $h$, the boundedness of $\{\x^k\}_{k\geq0}$ and Assumption \ref{assum:general} (i) ensure that $h$ is $\sigma$-strongly convex on a bounded set $\cU\supseteq\{\x^k\}_{k\geq0}$ with $\sigma>0$. It follows that $D_h(\x^{k+1},\x^k)\geq\sigma \|\x^{k+1}-\x^k\|^2_2$. This, together with \eqref{eq:decrease1}, yields
	\[ F(\x^{k+1})-F(\x^k)\leq -\frac{\sigma}{\bar{\alpha}}   \left\|\x^{k+1}-\x^k\right\|^2_2.\]
	Observe that the continuity of $f$ and the boundedness of $\{\x^k\}_{k\geq0}$ imply a lower bound $v>-\infty$ such that $F(\x^{k+1})-F(\x^0)\geq v$ for all $k\geq0$.  It follows that $F(\x^{k+1})-F(\x^k)\to0$, and then $\|\x^{k+1}-\x^k\|\to0$ by the above decrease inequality. Hence, the conditions of Proposition \ref{pro:bound} are satisfied, yielding an upper bound $M>0$ such that
	\begin{equation}\label{eq:boundM}
		\left\|\nabla h(\x^{k+1})-\nabla h(\x^k)\right\|_2\leq M,\qquad\forall~k\geq0.
	\end{equation}
	Combining the bound \eqref{eq:boundM} and Proposition \ref{pro:BPM} (ii), we obtain
	\[ D_h(\x^{k+1},\x^k)\geq \frac{1}{2}\exp(-rM)\left\| \nabla^2h(\x^{k+1})^{\frac12} (\x^{k+1}-\x^{k})\right\|^2_2. \]
	This, together with \eqref{eq:decrease1}, implies (H1).
	\subsubsection{Verifying (H2) for BPPM sequence.} 
	Consider the optimality condition \eqref{eq:op} with $q=1$. Multiply $\nabla^2h(\x^{k+1})^{-\frac12}$ in both sides and rearrange. We have
	\[ 	\nabla^2h(\x^{k+1})^{-\frac12}\left(\nabla f(\x^{k+1})+\p^{k+1}\right)=-\frac1{\alpha_k}\nabla^2h(\x^{k+1})^{-\frac12}\left(\nabla h(\x^{k+1})-\nabla h(\x^k)\right) . \]
	Since $\nabla f(\x^{k+1})+\p^{k+1}\in\partial F(\x^{k+1})$, we have 
	\[\begin{aligned}
		\dist\left(\bz,\partial F(\x^{k+1})\right)&\leq  \left\|\nabla^2h(\x^{k+1})^{-\frac12}\left(\nabla f(\x^{k+1})+\p^{k+1}\right)\right\|_2 \\
		&= \frac1{\alpha_k}\left\|\nabla^2h(\x^{k+1})^{-\frac12}\left(\nabla h(\x^{k+1})-\nabla h(\x^k)\right)\right\|_2\\
		&\leq \frac{1}{\underline{\alpha}}\cdot \exp(rM) \left\|\nabla^2 h(\x^{k+1})^{\frac12}\left(\x^{k+1}-\x^k\right)\right\|_2,
	\end{aligned} \]
	where the second inequality is due to Proposition \ref{pro:BPM} (iii), the bound \eqref{eq:boundM}, and $\alpha_k\geq\underline{\alpha}$. This proves (H2).
	\subsubsection{Verifying (H3) for BPPM sequence.}\label{sec:BPPM_H3}
	We first present a weighted convergence guarantee under the stepsize choice $0<\underline{\alpha}\leq \alpha_k\leq\bar{\alpha}$.
	\begin{fact}\label{fact:conver}
		Suppose that the sequence $\{\z^k\}_{k\geq0}\subseteq\R^n$ converges to $\bar{\z}\in\R^n$, and $\{\alpha_k\}_{k\geq0}$ satisfies $0<\underline{\alpha}\leq \alpha_k\leq\bar{\alpha}$. Then, we have 
		\[\frac{\sum_{l=0}^k\alpha_k\z^k}{\sum_{l=0}^k\alpha_k}\to\bar\z.  \]
	\end{fact}
	\begin{proof}[Proof of Fact \ref{fact:conver}]
		The triangle inequality and  $0<\underline{\alpha}\leq \alpha_k\leq\bar{\alpha}$ imply
		\[
		\begin{aligned}
			\left\|\frac{\sum_{l=0}^k\alpha_k\z^k}{\sum_{l=0}^k\alpha_k}-\z\right\|= \left\|\frac{\sum_{l=0}^k\alpha_k(\z^k-\z)}{\sum_{l=0}^k\alpha_k}\right\|\leq \frac{\sum_{l=0}^k\alpha_k\left\|\z^k-\z\right\|}{\sum_{l=0}^k\alpha_k}\leq \frac{\bar{\alpha}}{\underline{\alpha}}\cdot\frac{\sum_{l=0}^k\left\|\z^k-\z\right\|}{k+1}.\\
		\end{aligned}
		\]
		The convergence $\z^k\to\bar{\z}$ and \cite[Problem 3-1]{mattuck1999introduction} ensure the right-hand side converges to zero, yielding the desired result.\end{proof}
	Then, we are ready to verify (H3) for BPPM.
	
	Multiply the optimality condition \eqref{eq:op} with $\alpha_k$, sum the resulted equation from $0$ to $k$, and divide $\sum_{l=0}^k\alpha_l$. We obtain
	\begin{equation}\label{eq:limitop}
		\frac{\sum_{l=0}^k\alpha_l\nabla f(\x^{l+1})}{\sum_{l=0}^k\alpha_l}+\frac{\sum_{l=0}^k\alpha_l\p^{l+1}}{\sum_{l=0}^k\alpha_l}+\frac{\ \nabla h(\x^{k+1})-\nabla h(\x^0)}{\sum_{l=0}^k\alpha_l} =\bz,\quad\forall~k\geq0.
	\end{equation}
	Our strategy is to show that under the assumption that $\x^k\to\bar{\x}$, the equation \eqref{eq:limitop} converges to the stationarity condition of $f$ at $\bar{\x}$ so that $\bar{\x}$ is a stationary point of $f$. 
	
	Since $f\in\cC^1(\dom(F))$ and $\x^k\to\bar{\x}$, we have $\nabla f(\x^k)\to\nabla f(\bar{\x})$. Then, Fact \ref{fact:conver}  implies 
	\[	\frac{\sum_{l=0}^k\alpha_l\nabla f(\x^{l+1})}{\sum_{l=0}^k\alpha_l}\to \nabla f(\bar{\x}). \]
	Note that the bound \eqref{eq:boundM}, the optimality condition \eqref{eq:op}, and the boundedness of $\{\nabla f(\x^k)\}_{k\geq0}$ ensure the boundedness of $\{\p^{k+1}\}_{k\geq0}$. Recall $0<\underline{\alpha}\leq \alpha_k\leq \bar{\alpha}$. We have the boundedness of $\{{\sum_{l=0}^k\alpha_l\p^{l+1}}/{\sum_{l=0}^k\alpha_l}\}_{k\geq0}$. Then, there is a subsequence $\{k_d\}_{d\in\N}\subseteq\N$ and $\bar{\p}\in\R^n$ such that 
	\[\lim_{d\to\infty}\frac{\sum_{l=0}^{k_d}\alpha_l\p^{l+1}}{\sum_{l=0}^{k_d}\alpha_l}=\bar{\p}. \]
	We then show $\bar{\p}\in\partial g(\bar{\x})$. The above limit implies that for an arbitrary $\x\in\dom(g)$, 
	\begin{equation}\label{eq:barp}
		\begin{aligned}
			\left<\bar{\p},\bar{\x}-\x\right>&= \lim_{d\to\infty}\frac{\sum_{l=0}^{k_d}\alpha_l\left<\p^{l+1},\bar{\x}-\x\right>}{\sum_{l=0}^{k_d}\alpha_l} \\
			&=\lim_{d\to\infty}\frac{\sum_{l=0}^{k_d}\alpha_l\left<\p^{l+1},\bar{\x}-\x^{l+1}\right>}{\sum_{l=0}^{k_d}\alpha_l}+\frac{\sum_{l=0}^{k_d}\alpha_l\left<\p^{l+1},\x^{l+1}-\x\right>}{\sum_{l=0}^{k_d}\alpha_l}. \\
		\end{aligned}
	\end{equation}
	Observe that $\left<\p^{k+1},\bar{\x}-\x^{k+1}\right>\to0$ by $\x^k\to\bar{\x}$ and the boundedness of $\{\p^{k+1}\}_{k\geq0}$. Fact \ref{fact:conver} implies that the limit of first part in the right-hand of \eqref{eq:barp} is zero. Together with the convexity of $g$ and $\p^{k+1}\in\partial g(\x^{k+1})$, this implies
	\[\begin{aligned}
		\left<\bar{\p},\bar{\x}-\x\right>&=\lim_{d\to\infty} \frac{\sum_{l=0}^{k_d}\alpha_l\left<\p^{l+1},\x^{l+1}-\x\right>}{\sum_{l=0}^{k_d}\alpha_l}\\
		&\geq \lim_{d\to\infty} \frac{\sum_{l=0}^{k_d}\alpha_l\left(g(\x^{l+1})-g({\x})\right)}{\sum_{l=0}^{k_d}\alpha_l}\\
		&=g(\bar{\x})-g(\x),\qquad\quad\forall~\x\in\dom(g),
	\end{aligned}   \]
	where the last equality is due to Fact \ref{fact:conver} and $g(\x^{k+1})\to g(\bar{\x})$. It follows that 
	\begin{equation}\label{eq:pp}
		g(\x)-g(\bar{\x})\geq\left<\bar{\p},\x-\bar{\x}\right>,\quad\forall~\x\in\dom(g),  \quad{\rm i.e.},\quad \bar{\p}\in\partial g(\bar{\x}).
	\end{equation} 
	We then consider the limit of the third term in \eqref{eq:limitop}. Let $k=k_d$ in \eqref{eq:limitop} and $d\to\infty$. By the established convergences of the first two terms, we have
	\[\bar{\blam}:=\lim_{d\to\infty} \frac{ \nabla h(\x^{k_d+1})-\nabla h(\x^0)}{\sum_{l=0}^{k_d}\alpha_l}=-\nabla f(\bar{\x})-\bar{\p}. \]
	Using $\x^k\to\bar{\x}$, $\sum_{l=0}^{k_d}\alpha_l\to+\infty$, and $\nabla h(\x^0)/\sum_{l=0}^{k_d}\alpha_l\to\bz$, by Lemma \ref{le:divide}, we have \[\left<\bar\blam,\x-\bar\x\right>\leq0,\qquad\forall~\x\in \dom(h)\supseteq\dom(g).\] Combined with \eqref{eq:pp}, it yields
	\[g(\x)-g(\bar{\x})\geq\left<\bar{\p}+\bar\blam,\x-\bar{\x}\right>,\quad\forall~\x\in\dom(g),  \quad{\rm i.e.},\quad \bar{\p}+\bar\blam\in\partial g(\bar{\x}).  \]
	The above inclusion, together with $\bar\blam=-\nabla f(\bar{\x})-\bar{\p}$, yields
	\[	\bz=	\nabla f(\bar{\x})+\bar{\p}+\bar\blam\in\nabla f(\bar{\x})+\partial g(\bar{\x})= \partial F(\bar{\x}).\]
	This proves (H3). We complete the proof.

	\subsection{Proof of Theorem \ref{th:BPGM}}
	We verify (H1)---(H3) in order, under which the convergence result directly follows from Theorem \ref{th:general}.
	\subsubsection{Verifying (H1) for BPGM sequence.} The subproblem \eqref{eq:bpgm} ensures the following sufficient decrease inequality:
	\begin{equation} \label{eq:decrease2}
		g(\x^k)\geq \left<\nabla f(\x^k),\x^{k+1}-\x^k\right>+g(\x^{k+1})+\frac{1}{\alpha_k}D_h(\x^{k+1},\x^k).
	\end{equation}
	To estimate the right-hand side, we recall the convexity of $Lh-f$, which implies
	\[ Lh(\x^{k+1})-f(\x^{k+1})\geq  Lh(\x^{k})-f(\x^{k})+\left<L\nabla h(\x^{k})-\nabla f(\x^{k}),\x^{k+1}-\x^k\right> .  \]
	Rearrange the above inequality. We have
	\[ f(\x^k)\geq f(\x^{k+1})-\left<\nabla f(\x^{k}),\x^{k+1}-\x^k\right>-LD_h(\x^{k+1},\x^k).  \]
	Sum this inequality with \eqref{eq:decrease2}, and notice $F=f+g$. We obtain the sufficient decrease
	\[  F(\x^k)\geq F(\x^{k+1})+\left(\frac1{\alpha_k}-L\right)D_h(\x^{k+1},\x^k)\geq F(\x^{k+1})+\left(\frac1{\bar{\alpha}}-L\right)D_h(\x^{k+1},\x^k),  \]
	where the second inequality is due to $\alpha_k\leq \bar{\alpha}$.  Clearly, $1/{\bar{\alpha}}-L>0$ since $\bar{\alpha}<1/L$.
	
	Then, using the same arguments for \eqref{eq:boundM}, we know there is a scalar $M>0$ such that
	\begin{equation}\label{eq:boundM2}
		\left\|\nabla h(\x^{k+1})-\nabla h(\x^k)\right\|_2\leq M,\qquad\forall~k\geq0.
	\end{equation}
	This, together with Proposition \ref{pro:BPM} (ii) and the established sufficient decrease, yields (H1).
	
	\subsubsection{Verifying (H2) for BPGM sequence.}
	Rearrange the optimality condition \eqref{eq:op} with $q=0$. We have
	\[   	\nabla f(\x^{k+1})+\p^{k+1}=\nabla f(\x^{k+1})-\nabla f(\x^k)-\frac1{\alpha_k}\left(\nabla h(\x^{k+1})-\nabla h(\x^k)\right)  .\]
	Multiply $\nabla^2h(\x^{k+1})^{-\frac12}$ in both sides. We further have
	\[ \begin{aligned}
		&\nabla^2h(\x^{k+1})^{-\frac12}\left(\nabla f(\x^{k+1})+\p^{k+1}\right) \\
		=& \nabla^2h(\x^{k+1})^{-\frac12} \left(\nabla f(\x^{k+1})-\nabla f(\x^k) \right)-\frac1{\alpha_k}\nabla^2h(\x^{k+1})^{-\frac12}\left(\nabla h(\x^{k+1})-\nabla h(\x^k) \right). 
	\end{aligned} \]
	It follows from  $\nabla f(\x^{k+1})+\p^{k+1}\in\partial F(\x^{k+1})$, $\alpha_k\geq\underline{\alpha}$, and the triangle inequality that
	\[\begin{aligned}
		\dist\left(\bz,\nabla^2h(\x^{k+1})^{-\frac12} \partial F(\x^{k+1})\right) \leq & \left\|\nabla^2h(\x^{k+1})^{-\frac12}\left(\nabla f(\x^{k+1})+\p^{k+1}\right) \right\|_2 \\
		\leq& \left\|\nabla^2h(\x^{k+1})^{-\frac12} \left(\nabla f(\x^{k+1})-\nabla f(\x^k) \right) \right\|_2\\
		&+\frac1{\underline{\alpha}} \left\| \nabla^2h(\x^{k+1})^{-\frac12}\left(\nabla h(\x^{k+1})-\nabla h(\x^k) \right)\right\|_2\\
		\leq& \left(5L +\frac{1}{\underline{\alpha}}\right)\exp(rM)\left\| \nabla^2h(\x^{k+1})^{\frac12}(\x^{k+1}-\x^k)\right\|_2,
	\end{aligned}  
	\]
	where the last inequality is due to the bound \eqref{eq:boundM2}, Proposition \ref{pro:BPM} (iii), and  Proposition \ref{pro:BPMf} (with Remark \ref{re:f1}). This proves (H2).
	\subsubsection{Verifying (H3) for BPGM sequence.} The arguments are nearly the same as the verification of (H3) for the BPPM sequence in Sec. \ref{sec:BPPM_H3} except that $\nabla f(\x^{k+1})$ is replaced by $\nabla f(\x^{k})$, which does not affect the proof. We omit the details for brevity.  
	
	\section{Trajectory Convergence of Mirror Flow}\label{sec:flow}
	With iterate convergence well established for discrete-time BPMs, we now consider trajectory convergence of the continuous-time BPM, i.e., the mirror flow, which is an absolutely continuous curve\footnote{A curve $\x:\R_+\to\R^n$ is absolutely continuous if there exits a map $\y:\R_+\to\R^n$ that is integrable on any compact interval and satisfies $\x(t)=\x(0)+\int_0^t\y(s){\rm d}s$ for all $t\geq0$. In this case, $\y(t)=\dot{\x}(t)$ a.e. on $\R_+$. See \cite[Section 2.1]{davis2020stochastic}.}  (arc)  $\x:\R_+\to\R^n$ that satisfies the following dynamics (see, e.g., \cite[Section 3.1]{ding2025exploration} and \cite{krichene2015accelerated,bolte2003barrier}):
	\begin{equation}\label{eq:mirrorflow}
		\begin{aligned}
			&\bz\in\frac{{\rm d} }{{\rm d}t} \nabla h\left(\x(t)\right)+\partial F\left(\x(t)\right)\quad~&{\rm a.e.~ on }~\R_+,\\  
			&\x(t)\in\inter(\dom(h))\cap\dom(F),\quad&\forall~t\in\R_+.
		\end{aligned}\tag{$\cE$}
	\end{equation} 
	Instead of using the extended SK\L{} property to analyze its convergence, we adopt a more intuitive approach based on transformations induced by the parameterization functions. 
	\begin{proposition}{\rm (cf. \cite[Theorem 5.1]{li2022implicit})}\label{pro:Euclidean}
		Suppose that Assumption \ref{assum:general} holds with $\psi,\phi$ given in Definition \ref{def:psi}. Let $\x:\R_+\to\R^n$ be an arc satisfying \eqref{eq:mirrorflow}.  Define $\Phi:\z\mapsto F(\phi(\z))$. Then, the parameterized curve $\y:t\mapsto\psi(\x(t))$ is an arc that satisfies the Euclidean subgradient dynamics
		\begin{equation}\label{eq:Euclidean_flow}
			\begin{aligned}
				&\bz\in\dot{\y}(t)+ \partial \Phi(\y(t))\quad~&{\rm a.e.~ on }~\R_+;\\
				&y_i(t)\in\inter(\dom(\phi)),&\quad\forall~i\in[n],~t\in\R_+.
			\end{aligned}\tag{$\cG$}
		\end{equation}
	\end{proposition}
	Given the above transformation, a natural approach to establishing the convergence of mirror flow trajectory $\{\x(t)\}_{t\geq0}$ is proving the convergence of parameterized trajectory $\{\y(t)\}_{t\geq0}$ to some point $\bar{\y}$ and then using the continuity of $\phi$ to yield $\lim_{t\to\infty}\x(t)=\lim_{t\to\infty}\phi(\y(t))=\phi(\bar{\y})$. Then, to establish convergence of  $\{\y(t)\}_{t\geq0}$ through the K\L{} framework \cite{absil2005convergence}, two core questions arise:
	\begin{itemize}
		\item When will the parameterized objective function $\Phi$ satisfy the K\L{} inequality?
		\item Is the parameterized trajectory $\{\y(t)\}_{t\geq0}$ bounded when the original trajectory is?
	\end{itemize}
	To address these issues,  we develop the following properties for the parameterization functions.
	\begin{proposition}\label{pro:parameterization2}
		Suppose that Assumption \ref{assum:general} (i) holds with $\dom(h)$ closed. Let $\psi,\phi$ be parameterization functions of $h$ as defined in Definition \ref{def:psi}.  The  following hold:
		\begin{enumerate}[label={{\rm (\roman*)}}]
			\item $\phi$ and $\psi$ are definable in $\cS(\R_{{\rm an,exp}})$.
			\item $\psi([a,b])$ is bounded for any bounded interval $[a,b]\subseteq\dom(\varphi)$. 
		\end{enumerate}
	\end{proposition}
	Proposition \ref{pro:parameterization2} (i), along with the composition rule of definable functions and the fact that definable functions satisfies K\L{} inequality \cite{bolte2007clarke}, ensures the K\L{} inequality of the parameterized objective function $\Phi$ when $F$ is definable. 
	Proposition \ref{pro:parameterization2} (ii) ensures the boundedness of the parameterized trajectory $\{\psi(\x(t))\}_{t\geq0}$ whenever the original trajectory $\{\x(t)\}_{t\geq0}$ is bounded.  The aforementioned questions are thereby addressed. We remark that the closedness of $\dom(h)$ is necessary for item (ii): When $\dom(h)$ is open, $\{\y(t)\}_{t\geq0}$ may diverge even if the original mirror flow trajectory $\{\x(t)\}_{t\geq0}$ converges, making our methodology inapplicable. To see this, it suffices to consider the setting of Example \ref{example:skl_fails}.
	\begin{example}[Unbounded Parameterized Trajectory]\label{example:unbound}
		Consider the Burg entropy kernel $h(x)=-\log(x)$ and the function $F:x\mapsto x+\delta_{\R_+}(x)$, where the inverse parameterization function $\psi$ is $x\mapsto\log(x)$ by Example \ref{example:psi}. Let the mirror flow trajectory $\{x(t)\}_{t\geq0}$ satisfy 
		\[ \frac{{\rm d}}{{\rm d}t}\left(-\frac1{x(t)} \right)+1=0;\qquad x(0)=1\qquad x(t)>0.\]
		It is easy to derive that $x(t)=\frac{1}{1+t}$ for $t\geq0$, which converges to $x^*=0$. 
		However, the parameterized trajectory $\{y(t)\}_{t\geq0}$ is given by $y(t)=\psi(x(t))=-\log(1+t)$, which diverges as $t\to+\infty$.
	\end{example}

	With these preparations, we are ready to establish trajectory convergence of mirror flow.
	\begin{theorem}\label{th:main}
		Suppose that Assumption \ref{assum:general} holds with $\dom(h)$ closed, and $F$ is definable in $\cS(\R_{{\rm an,exp}})$. Let $\x:\R_+\to\R^n$ be a bounded arc satisfying \eqref{eq:mirrorflow}. 
		Then, the trajectory $\{\x(t)\}_{t\geq0}$ converges to a stationary point of $F$ with a finite length.
	\end{theorem}
	Theorem \ref{th:main} establishes, to the best of our knowledge, the first general trajectory convergence result for mirror flow that does not rely on either the convexity of the objective function or isolation property of the stationary points. This represents a significant improvement over existing results on mirror trajectory \cite{ding2025exploration,Attouch2004RegularizedLD,Alvarez2018HessianRG,Bolte2003BarrierOA}, demonstrating the power of our parameterization-based convergence analysis. Consequently, we complete the continuous-time convergence analysis for BPMs.

	\subsection{Proof of Proposition \ref{pro:Euclidean}}
	Note that $\psi$ is locally Lipschitz continuous on $\inter(\dom(\varphi))$ since $\psi^{\prime}(z)=\sqrt{\varphi^{\prime\prime}(z)}$ is continuous for $z\in\inter(\dom(\varphi))$.
	For the arc (i.e., absolutely continuous curve) $\x$ on $\inter(\dom(h))$, it is clear that  $\y=\psi(\x)$ is also an arc. We then prove it satisfies \eqref{eq:Euclidean_flow}.
	
	Due to the a.e. differentiability of $\x$ and the fact that $\x(t)\in\inter(\dom(h))$ for $t\in\R_+$, we have $\frac{{\rm d}}{{\rm d} t} \nabla h(\x(t))=\nabla^2h(\x(t))\dot{\x}(t)$ a.e. on $\R_+$. Hence, we may rewrite the dynamics \eqref{eq:mirrorflow} as 
	\[\bz\in \nabla^2 h(\x(t))\dot{\x}(t)+\partial F(\x(t))\quad~{\rm a.e.~ on }~\R_+.\] 
	Note that the strong convexity of $\varphi$ ensures $\nabla^2h(\x(t))=\Diag(\varphi^{\prime\prime}(x_1(t)),\ldots,\varphi^{\prime\prime}(x_n(t)))\succ\bz$.
	The above inclusion can be further formulated as
	\begin{equation}\label{eq:middle_flow}
		\bz\in \nabla^2 h(\x(t))^{\frac12}\dot{\x}(t)+\nabla^2 h(\x(t))^{-\frac12}\partial F(\x(t))\quad~{\rm a.e.~ on }~\R_+.
	\end{equation}
	Next, we show that \eqref{eq:middle_flow} is equivalent to subdifferential inclusion in \eqref{eq:Euclidean_flow}. To find the derivative of the curve $\y:t\mapsto\psi(\x(t))$, observe that a.e. differentiability of $\x$ yield 
	\begin{equation*}
		\dot{\y}(t)=\frac{{\rm d} }{{\rm d}t} \psi(\x(t))=\Diag\left(\psi^{\prime}(\x(t))\right)\dot{\x}(t)\quad~{\rm a.e.~ on }~\R_+. 
	\end{equation*}
	Using \eqref{eq:psi}, we see that $\psi^{\prime}(x_i(t))=\sqrt{\varphi^{\prime\prime}(x_i(t))}$ for $i\in[n]$ and hence
	\begin{equation}\label{eq:doty}
		\dot{\y}(t)=\Diag\left(\sqrt{\varphi^{\prime\prime}(\x(t))}\right)\dot{\x}(t)= \nabla^2 h(\x(t))^{\frac12}\dot{\x}(t)\quad~{\rm a.e.~ on }~\R_+. 
	\end{equation}
	To characterize $\partial \Phi=\partial (F\circ \phi)$, observe that $\phi=\psi^{-1}$ and  $\psi^{\prime}=\sqrt{\varphi^{\prime\prime}}>0$ yield 
	\begin{equation}\label{eq:fullrank}
		\begin{aligned}
			\Diag\left(\phi^{\prime}(\y(t))\right)&=\Diag\left(\frac1{\psi^{\prime}(x_1(t))},\ldots,\frac1{\psi^{\prime}(x_n(t))} \right)\\
			&=\Diag\left(\frac1{\sqrt{\varphi^{\prime\prime}(x_1(t))}},\ldots,\frac1{\sqrt{\varphi^{\prime\prime}(x_n(t))}} \right)\\
			&=  \nabla^2 h(\x(t))^{-\frac12}\succ\bz.
		\end{aligned}
	\end{equation}
	This, together with \cite[Exercise 10.7]{rockafellar2009variational} (chain rule) and $\x(t)=\phi(\y(t))$, yields
	\begin{equation*}\label{eq:partial_g}
		\partial \Phi(\y(t))=\Diag\left(\phi^{\prime}\left(\y(t)\right)\right)\partial F\left(\phi\left(\y(t)\right)\right)= \nabla^2 h(\x(t))^{-\frac12}\partial F(\x(t)).
	\end{equation*}
	Combining the above subdifferential expression with \eqref{eq:middle_flow} and \eqref{eq:doty}, we obtain
	\[	\bz\in\dot{\y}(t)+ \partial \Phi(\y(t))\quad~{\rm a.e.~ on }~\R_+. \]
	Note that the monotonic increasing of $\psi$ and $x_i(t)\in\inter(\dom(\varphi))$ ensure $y_i(t)=\psi(x_i(t))\in\inter(\range(\psi))=\inter(\dom(\phi))$ for $i\in[n]$.
	We conclude that the curve $\y$ satisfies \eqref{eq:Euclidean_flow}.
	
	\subsection{Proof of Proposition \ref{pro:parameterization2}}
	Thanks to Proposition \ref{pro:parameterization}. Item (ii) directly follows from the  monotonic increasing and well-definedness of $\psi$ on $\dom(\varphi)$ and the closedness of $\dom(\varphi)$. Therefore, we focus on the proof of item (i).
	
	As $\hat{\varphi}:\dom(\varphi)\to\R_+$ is globally subanalytic (i.e., definable in $\cS(\R_{{\rm an}})$), we know that $\hat{\varphi}|_{\inter(\dom(\varphi))}:x\mapsto1/\varphi^{\prime\prime}(x)$ is also globally subanalytic by Proposition \ref{pro:o1} (iii).
	This, together with global subanalyticity of $u:x\mapsto1/\sqrt{x}$ on $(0,+\infty)$ and the composition rule Proposition \ref{pro:o1} (v), implies that $\sqrt{\varphi^{\prime\prime}}:\inter(\dom(\varphi))\to\R$, i.e., $u\circ\hat{\varphi}$, is globally subanalytic.
	
	Now, we are prepared to show that $\psi:x\mapsto \int_{x_0}^x\sqrt{\varphi^{\prime\prime}(s)}{\rm d}s$ is definable in $\cS(\R_{{\rm an,exp}})$. It suffices to consider an arbitrary $x_0\in\inter(\dom(\varphi))$ since different $x_0$ only cause a constant difference for $\psi$. Let us define the function $v:\R\times\R\to\R$ by
	\[v(s,x)=\begin{cases}
		\begin{aligned}
			\sqrt{\varphi^{\prime\prime}(s)},\qquad &\text{ if }(s,x)\in \cD_1=\{(s,x):x_0\leq s\leq x, ~x\in\inter(\dom(\varphi))\};\\
			-\sqrt{\varphi^{\prime\prime}(s)},\qquad &\text{ if }(s,x)\in \cD_2=\{(s,x):x\leq s< x_0,~x\in\inter(\dom(\varphi))\};\\
			0,\qquad&\text{ if }(s,x)\in\cD_3=\R^2\setminus(\cD_1\cup\cD_2).
		\end{aligned}
	\end{cases}\]
	Direct computation shows that its parameterized integral $\hat{v}:\inter(\dom(\varphi))\to\R$ defined by $\hat{v}(x)=\int_{\R}v(s,x){\rm d}s$ satisfies
	\[\hat{v}(x)= 
	\int_{x_0}^x\sqrt{\varphi^{\prime\prime}(s)}{\rm d}s=\psi(x),\qquad \forall~x\in\inter(\dom(\varphi)).
	\]
	As $\graph(v)$ is the union of $\graph(v|_{\cD_1})$, $\graph(v|_{\cD_2})$, and $\{(s,x,0):(s,x)\in\cD_3\}$, which are globally subanalytic, we know that $\graph(v)$ and further $v$ are globally subanalytic. Then, it follows from Lemma \ref{le:integ} that the parameterized integral $\hat{v}$ (and its graph) is definable in $\cS(\R_{{\rm an,exp}})$. 
	
	Finally, observe that $\graph(\psi)=\cl(\graph(\hat{v}))$ due to (i) $\hat{v}(x)=\psi(x)$ for $x\in\dom(\hat{v})=\inter(\dom(\varphi))$; (ii) $\dom(\psi)=\dom(\varphi)=\cl(\dom(\hat{v}))$; and (iii) the continuity of $\psi$.  By Proposition \ref{pro:o1} (i), $\graph(\psi)$ and further $\psi$ are definable in $\cS(\R_{{\rm an,exp}})$. Also, the inverse function $\phi$ is definable in $\cS(\R_{{\rm an,exp}})$ by Proposition \ref{pro:o1} (ii). We complete the proof.

	\subsection{Proof of Theorem \ref{th:main}}
	We first develop a decrease property for the curve $\gamma:t\mapsto g(\y(t))$ to facilitate the classical K\L{} inequality-based analysis. 
	\begin{lemma}\label{le:path}
		Consider the setting of  Theorem \ref{th:main}.  Let $\psi,\phi$ be parameterization functions of $h$ as defined in Definition \ref{def:psi}, and efine the curve $\gamma:\R_+\to\R^n$ by $\gamma(t)=\Phi(\y(t))$. Then, $\gamma$ is absolutely continuous and satisfies
		\begin{equation}\label{eq:suff_descent}
			\gamma^{\prime}(t)=-\left\|\dot{\y}(t) \right\|^2_2 \quad~{\rm a.e.~ on }~\R_+.
		\end{equation}
	\end{lemma}
	\begin{proof} Recall that $\Phi=F\circ \phi$ and note that (i) $F=f+g$ is regular and locally Lipschitz continuous by \cite[Exercise 8.20 (b)]{rockafellar2009variational} and Assumption \ref{assum:general} (ii); and (ii) $\nabla\phi(\z)=\Diag(\phi^{\prime}(z_1),\ldots,\phi^{\prime}(z_n))$ is of full rank for all $\z\in\inter(\dom(\phi))$ since direct computation gives
		\[\phi^{\prime}(z_i)=\psi^{\prime}(\phi(z_i))^{-1}=\varphi^{\prime\prime}(\phi(z_i))^{-\frac12}>0.\]
		Thanks to \cite[Theorem 10.6]{rockafellar2009variational}, we see that  $\Phi$ is regular on $\inter(\dom(\phi))^n\cap\dom(\Phi)$. Moreover, from the expression of $\nabla\phi(\z)$ and continuity of $\phi$, $\varphi^{\prime\prime}$,  we see that $\phi$ is continuously differentiable and hence locally Lipschitz continuous on $\inter(\dom(\phi))$. It follows that $\Phi$ is also locally Lipschitz continuous on $\inter(\dom(\phi))^n\cap\dom(\Phi)$. Then, by the same arguments of \cite[Lemma 5.4]{davis2020stochastic}, we have
		\[	\gamma^{\prime}(t)=\left<\v,\dot{\y}(t)\right>, \forall ~\v\in\partial\Phi(\y(t)) \quad~{\rm a.e.~ on }~\R_+. \]
		Since $-\dot{\y}(t)\in\partial \Phi(\y(t))$ a.e. on $\R_+$ by \eqref{eq:Euclidean_flow}, the above equality yields \eqref{eq:suff_descent} as desired.
	\end{proof}
	Next, we show that the limiting point of  mirror trajectory is stationary once it converges, avoiding spurious stationary points \cite{chen2025spurious} that may trap BPMs. Similar results were given in \cite{ding2025exploration} for linear constraint sets.  In contrast, our setting poses substantially greater challenges due to the general composite structure of $f$. We defer the proof to appendix since it is quite technical and lengthy.
	\begin{lemma}\label{le:limit}
		Consider the setting of  Theorem \ref{th:main}. Suppose that $\dom(\varphi)$ is closed and $\{\x(t)\}_{t\geq0}$ converges to $\bar{\x}\in\R^m$. Then, $\bar{\x}$ is a stationary point of $f$.
	\end{lemma}
	With Lemmas \ref{le:path} and \ref{le:limit} in place, we are ready to prove Theorem \ref{th:main}. Let $\psi$, $\phi$ be the parameterization functions of the kernel $h$ as defined in Definition \ref{def:psi}. Consider the parameterized curve $\y:\R_+\to\R^n$ where $\y(t)\coloneqq \psi(\x(t))$. Our strategy is to apply the classical K\L{} convergence analysis to $\{\y(t)\}_{t\geq0}$ and establish its convergence. This would then ensure the convergence of $\{\x(t)\}_{t\geq0}$, which, together with Lemma \ref{le:limit}, proves Theorem \ref{th:main}.
	\begin{proof}[Proof of Theorem \ref{th:main}]
		The proof is similar to that of \cite[Theorem 4.5]{bolte2007lojasiewicz}. Nevertheless, we still provide a proof for completeness.
		
		Recall that the trajectory $\{\x(t)\}_{t\geq0}$ is assumed bounded. There is an accumulation point $\bar{\x}\in\dom(F)\subseteq\dom(h)$ such that $\x(t_k)\to\bar{\x}$ for some sequence $\{t_k\}_{k\geq0}$ with $t_k\to+\infty$.
		Let $\bar{\y}=\psi(\bar{\x})$. We know that $\bar{\y}\in\R^n$ is a well-defined finite vector by Proposition \ref{pro:parameterization} (i). To prove $\x(t)\to\bar{\x}$, our strategy is to prove $\y(t)\to\bar{\y}$.
		
		As $\y(t_k)=\psi(\x(t_k))$ and $\psi$ is continuous  by Proposition \ref{pro:parameterization} (i), we know that \[\lim\limits_{k\to\infty}\y(t_k)=\lim\limits_{k\to\infty}\psi(\x(t_k))=\psi(\bar{\x})=\bar{\y}.\]
		Further, by the continuity of $\Phi=F\circ\phi$, we have 
		\[\lim_{k\to\infty}\gamma(t_k)=\lim_{k\to\infty}\Phi(\y(t_k))= \Phi(\bar{\y}).\]
		This, together with $\eqref{eq:suff_descent}$, which ensures the monotonic decreasing of $\gamma$, implies 
		\begin{equation}\label{eq:gamma}
			\lim_{t\to+\infty}\gamma(t)=\Phi(\bar{\y});\qquad \gamma(t)\geq \Phi(\bar{\y}),\quad\forall~t\in\R_+.
		\end{equation}
		Consider the trivial case of \eqref{eq:gamma}, where 
		\[\gamma(T)=\Phi(\bar{\y})\qquad \text{ for some }~ T\in\R_+.\] In this case, the limit $\gamma(t)\to \Phi(\bar{\y})$ and the monotonic decrease of $\gamma(t)$ imply $\gamma(t)=\Phi(\bar{\y})$ for $t\geq T$.
		Then, it follows from \eqref{eq:suff_descent} that $\|\dot{\y}(t)\|_2=\gamma^{\prime}(t)=0$ a.e. on $[T,+\infty)$ and further ${\y}(t)\equiv\y(T)$ for $t\geq T$. This, together with $\y(t_k)\to\bar{\y}$, yields $\y(T)=\bar{\y}$ and $\y(t)\to\bar{\y}$. 
		
		Moreover, the trajectory $\y(t)$ has a finite length since $\int_{0}^{+\infty}\|\dot{\y}(t)\|_2{\rm d}t=\int_{0}^T\|\dot{\y}(t)\|_2{\rm d}t$ and
		\begin{equation}\label{eq:finite}
			\begin{aligned}
				\int_{0}^T\|\dot{\y}(t)\|_2~{\rm d}t&\leq \frac12\int_{0}^T\left(1+\|\dot{\y}(t)\|^2_2\right){\rm d}t
				\\
				&=\frac{T}{2}-\frac12	\int_{0}^T\gamma^{\prime}(t){\rm d}t\\
				&=\frac{T}{2}+\frac12\left(\gamma(0)-\gamma(T)\right)<+\infty,
			\end{aligned}	
		\end{equation}
		where the first inequality uses the Cauchy inequality and the second one is due to \eqref{eq:suff_descent}.
		
		We then focus on the nontrivial case where 
		\[\gamma(t)>\Phi(\bar{\y}),\qquad \forall~ t\in\R_+.\]
		Let us first show that the parameterized trajectory $\{\y(t)\}_{t\geq0}$ is bounded.
		By Proposition \ref{pro:parameterization} (ii) and the boundedness of $\{\x(t)\}_{t\geq0}$, i.e., $\x(t)\in[a,b]^n$ for some $a,b\in\dom(\varphi)$, we have
		\[\y(t)\in[c,d]^n,\qquad\quad \forall~t\in\R_+, \]
		with $[c,d]=\psi([a,b])\subseteq\dom(\varphi)$ also bounded.
		
		Then, we show that the K\L{} inequality holds for $\Phi$, for which we study the definability of $\Phi$.
		Write $\Phi|_{[c,d]^n}=F\circ \phi|_{[c,d]^n}$, where $\phi|_{[c,d]^n}=(\phi|_{[c,d]},\ldots,\phi|_{[c,d]})$. By Proposition \ref{pro:parameterization} (iii) and Proposition \ref{pro:o1} (iv), $\phi|_{[c,d]^n}$ is definable in $\cS(\R_{{\rm an,exp}})$. This, together with $F\in\cS(\R_{{\rm an,exp}})$ and the composition rule Proposition \ref{pro:o1} (v), yields 
		\[\Phi|_{[c,d]^n}\in\cS(\R_{{\rm an,exp}}).\]
		Then, by the K\L{} inequality \cite[Theorem 14]{bolte2007clarke} and the fact $\partial \Phi\subseteq \partial_C \Phi$ (see \cite[Equation (7)]{bolte2007clarke}; where $\partial_C$ denotes Clarke subdifferential),  there a scalar $\rho>0$, a strictly increasing continuous definable function $\zeta:[0,\rho)\to\R_+$ which is continuously differentiable on $(0,\rho)$ with $\zeta(0)=0$, and a continuous definable function $w:\R_+\to(0,\rho)$ such that
		\begin{equation}\label{eq:KLw}
			\zeta^{\prime}\left(|\Phi(\z)-\Phi(\bar{\y})|\right)\dist\left(\bz,\partial \Phi(\z)\right)\geq1,
		\end{equation}
		whenever $0<|\Phi(\z)-\Phi(\bar{\x})|\leq w(\|\z\|_2)$. 
		
		Let $\eta=\min_{\z\in[c,d]^n}w(\|z\|_2)$. The continuity of $w:\R_+\to(0,\rho)$ implies $\eta\in(0,\rho)$. Then, the K\L{} inequality \eqref{eq:KLw} holds whenever  $0<|\Phi(\z)-\Phi(\bar{\x})|\leq\eta$ and $\z\in[c,d]^n$.
		
		Recall that $\Phi(\y(t))=\gamma(t)> \Phi(\bar{\y})$ and $\gamma(t)\to \Phi(\bar{\y})$. We know that there exists $T>0$ such that
		\[\gamma(t)-\Phi(\bar{\y})\in(0,\eta],\qquad~\forall~t\geq T.\]
		Let $\z=\y(t)\in[c,d]^n$ in \eqref{eq:KLw} and notice $-\dot{\y}(t)\in\partial \Phi(\y(t))$ by \eqref{eq:Euclidean_flow}. We have
		\begin{equation}\label{eq:KL2}
			\zeta^{\prime}\left(\gamma(t)-\Phi(\bar{\y})\right)\|\dot{\y}(t)\|_2\geq1,\qquad {\rm a.e.}\quad [T,+\infty).
		\end{equation}
		Combined with \eqref{eq:suff_descent}, the K\L{} inequality \eqref{eq:KL2} yields that a.e. on $[T,+\infty)$,
		\[\begin{aligned}
			\frac{{\rm d} }{{\rm d}t }\zeta\left(\gamma(t)-\Phi(\bar{\y})\right)&=\zeta^{\prime}\left(\gamma(t)-\Phi(\bar{\y})\right){\gamma}^{\prime}(t)\\
			&= -\zeta^{\prime}\left(\gamma(t)-\Phi(\bar{\y})\right)\|\dot{\y}(t)\|^2_2\\
			&\leq -\frac{1}{\|\dot{\y}(t)\|}\|\dot{\y}(t)\|^2_2= -\|\dot{\y}(t)\|_2.
		\end{aligned} \]
		Integrate on $[T,+\infty)$ in both sides. It follows that
		\begin{equation}\label{eq:finite2}
			\begin{aligned}
				\int_{T}^{+\infty}\|\dot{\y}(t)\|{\rm d}t
				&\leq\zeta\left(\gamma(T)-\Phi(\bar{\y})\right)-\lim_{t\to\infty}\zeta\left(\gamma(t)-\Phi(\bar{\y})\right)\\
				&=\zeta\left(\gamma(T)-\Phi(\bar{\y})\right)-\zeta(0)\zeta\left(\gamma(T)-\Phi(\bar{\y})\right),
			\end{aligned}
		\end{equation}
		where the first equality is due to the continuity of $\zeta$ and $\gamma(t)\to \Phi(\bar{\y})$, and the second equality uses $\zeta(0)=0$.
		
		The above bound, together with the arguments in \eqref{eq:finite}, ensures the finite length of the trajectory $\{\y(t)\}_{t\geq0}$. We conclude that $\{\y(t)\}_{t\geq0}$ converges to the accumulation point $\bar{\y}$. 
		
		Back to the mirror descent trajectory $\{\x(t)\}_{t\geq0}$. We see from the continuity of $\phi$ that
		\[\lim_{t\to\infty}\x(t)=\lim_{t\to\infty}\phi(\y(t))=\phi(\bar{\y})=\bar{\x}. \]
		This says that the original trajectory $\{\x(t)\}_{t\geq0}$ converges to $\bar{\x}$. Moreover, it has a finite length as the chain rule \eqref{eq:doty}, the boundedness of $\x(t)$, and the continuity of $\hat{\varphi}$ on $\cl(\dom(\varphi))$ by (A3), imply
		\[\int_0^{+\infty}\left\|\dot{\x}(t)\right\|_2{\rm d}t=\int_0^{+\infty}\left\| \Diag\left(\sqrt{\varphi^{\prime\prime}(\x(t))}\right)^{-1}\dot{\y}(t)\right\|_2{\rm d}t\leq \sigma\int_0^{+\infty}\left\|\dot{\y}(t)\right\|_2{\rm d}t<+\infty,\]
		for some $\sigma>0$.
		
		Finally, by Lemma \ref{le:limit}, we know that the limiting point $\bar{\x}$ is a critical point of $f$. We complete the proof.

	\end{proof}

	\section{Conclusion and Discussion}\label{sec:conclusion}
	In this paper, we extended the SK\L{} property in \cite{chen2026skl}, as well as the scaled sufficient decrease and scaled relative error, to general kernels and objective functions, thereby establishing a unified convergence framework for BPMs. To verify its applicability, we showed that the extended SK\L{} property continues to hold for continuous subanalytic functions when the kernel has a closed domain, and that sequences generated by the BPGM and BPPM satisfy the assumptions of the framework under mild regularity conditions. Moreover, through the parameterization technique, we established trajectory convergence of the continuous-time BPM, namely mirror flow, to stationary points for o-minimal definable objective functions. Taken together, these results provide a broad convergence theory for both discrete- and continuous-time BPMs.
	
	Despite these advances, many questions remain open in the convergence analysis of BPMs. First, the closedness of $\dom(h)$ plays a crucial role in our convergence analysis. Examples \ref{example:skl_fails} and \ref{example:unbound} reveal fundamental challenges in analyzing BPMs under kernels with open domains: The (extended) SK\L{} property may fail for simple semi-algebraic functions, and the parameterized trajectory could become unbounded even if the original mirror flow is bounded. These phenomena suggest that fundamentally new techniques may be required to analyze BPMs under open-domain kernels. Second, our unified convergence framework relies on a regularity condition on the domain set (see Assumption \ref{assum:general2}). Since analogous conditions are typically unnecessary for Euclidean distance-based methods and continuous-time BPM, it is natural to ask whether this assumption can be removed so that the framework applies to arbitrary closed convex domains. Third, the present work focuses primarily on asymptotic convergence analysis. It would be of substantial interest to develop non-asymptotic guarantees that characterize the convergence behavior of BPMs in finite time. We hope that future investigations along these directions will further deepen the understanding of Bregman proximal methods.
	
	\section*{Acknowledgements}
	We did not use AI to develop the ideas or proofs in this paper. AI tools (primarily GPT-4o) were used solely for checking grammar and typos.
	
	\appendix
	\section{Proof of Proposition \ref{pro:estimate}}
	Let us first prove the following technical lemma that slightly generalizes the \L{}ojasiewicz inequality for a univariable function, providing both lower and upper bounds on the function value growth.
	\begin{lemma}\label{le:exponent}
		Let $\zeta:\R\to\overline{\R}$ be a continuous globally subanalytic function on a closed interval. Suppose that $\zeta$ is continuously differentiable on $\inter(\dom(\zeta))$ with $\zeta({x})=0$ for  ${x}\in\bd(\dom(\zeta))$ and $\zeta(x)>0$ for $x\in\inter(\dom(\zeta))$. Then, at $\bar{x}\in \bd(\dom(\zeta))$, either $\zeta$ has  K\L\ exponent $0$ or there exist scalars $\Delta>0$, $\eta\in(0,1)$ such that
		\begin{equation}\label{eq:zeta_estimate}
			|\zeta^{\prime}(x)|=\Theta\left(\zeta(x)^{\eta}\right)\qquad{\rm on}\quad [\bar{x}-\Delta,\bar{x}+\Delta]\cap\inter(\dom(\zeta)),
		\end{equation}
		where $\zeta^{\prime}(x)$ is positive (resp. negative) if $\bar{x}$ is the left-end (resp. right-end) point.
	\end{lemma}
	\begin{proof}[Proof of Lemma \ref{le:exponent}]
		We combine the proof of \cite[Theorem 3.1]{bolte2007lojasiewicz} and specialized properties of $\zeta$ to prove the lemma.  As $\dom(\zeta)$ is a closed interval, without loss of generality, we assume that $\bar{x}$ is the left-end point of $\dom(\zeta)$. Since $\zeta(\bar{x})=0$ and $\zeta(x)\geq0$ for $x\in\dom(\zeta)$, we know that $\bar{x}$ is a global optimal minimizer of $\zeta$, and thus a stationary point of $\zeta$. 
		
		To take advantage of the results of \citet{bolte2007lojasiewicz}, we introduce the notion therein. Let $\cS$ be the stationary point set of $\zeta$, $\cS_{\bar{x}}$ be the connected component of $\cS$ containing $\bar{x}$, and $\cU_{\Delta_1}:=\{x\in\dom(\zeta):\dist(x,\cS_{\bar{x}})\leq\Delta_1\}$ with $\Delta_1>0$ be a set that separates $\cS_{\bar{x}}$ from other connected components of $\cS$. Note that the continuity and subanalyticity of $\zeta$ ensure $\zeta(x)\equiv \zeta(\bar{x})=0$ for $x\in\cS_{\bar{x}}$ by \cite[Equation (6)]{bolte2007lojasiewicz}. We have $\cS_{\bar{x}}\subseteq\zeta^{-1}(0)=\bd(\dom(\zeta))$. Since $\bd(\dom(\zeta))$ consists of at most two isolated points, every connected component is a singleton. It follows that $\cS_{\bar{x}}=\{\bar{x}\}$ and $\cU_{\Delta_1}=[\bar{x},\bar{x}+\Delta_1]$ is compact.
		
		Furthermore, we define \[m_{\zeta}(x)\coloneqq\inf\{|v|:v\in\partial \zeta(x)\};\quad \xi(t)=\inf\left\{m_{\zeta}(x):x\in\cU_{\Delta_1}\cap \zeta^{-1}(t)\right\}.\]
		Then, by \cite[Equations (8) and (12)]{bolte2007lojasiewicz}, the following hold:
		\begin{enumerate}[label={{\rm (\roman*)}}]
			\item There exist scalars $r>0$, $\theta\in[0,1)$, and $\Delta_2>0$ with $\bar{x}+\Delta_2\in\inter(\dom(\zeta))$ such that
			\begin{equation}\label{eq:zeta}
				\zeta(x)^{\theta}=|\zeta(x)-\zeta(\bar{x})|^{\theta}\leq r\cdot m_{\zeta}(x),\qquad \forall~x\in [\bar{x},\bar{x}+\Delta_2].
			\end{equation}
			\item $\lim_{t\to0+}\xi(t)=l\in[0,+\infty]$.
			\item  When $l=0$, there exist $c,\eta>0$ such that
			\begin{equation}\label{eq:approxi}
				\xi(t)=ct^{\eta}+o(t^{\eta}).
			\end{equation}
		\end{enumerate}
		To prove \eqref{eq:zeta_estimate}, our strategy is to simplify \eqref{eq:approxi} based on the properties of $\zeta$.
		We begin with $m_{\zeta}(x)$. By the differentiability of $\zeta$ on $\inter(\dom(\zeta))$, we have  $m_{\zeta}(x)=|\zeta^{\prime}(x)|$ for $x\in\inter(\dom(\zeta))$. This, together with \eqref{eq:zeta} and $\zeta(x)>0$ for $x\in\inter(\dom(\zeta))$, implies that
		\[ |\zeta^{\prime}(x)|>0,\qquad \forall~x\in(\bar{x},\bar{x}+\Delta_2]. \]
		As $\zeta^{\prime}$ is continuous on $\inter(\dom(\zeta))$, the sign of $\zeta^{\prime}$ does not change on $(\bar{x},\bar{x}+\Delta_2]$; otherwise there exists $y\in(\bar{x},\bar{x}+\Delta_2)$ such that $\zeta^{\prime}(y)=0$ by the intermediate value theorem, leading to a contradiction. On the other hand, the mean value theorem asserts that for every $x\in(\bar{x},\bar{x}+\Delta_2)$ there exists $y\in (\bar{x},x)$ such that $\zeta(x)-\zeta(\bar{x})=\zeta^{\prime}(y)(x-\bar{x})$. This, along with $\zeta(x)-\zeta(\bar{x})=\zeta(x)>0$ and $x-\bar{x}>0$, implies $\zeta^{\prime}(y)>0$. Taken together, we conclude that 
		\begin{equation}\label{eq:zeta1m}
			\zeta^{\prime}(x)>0,\qquad  m_{\zeta}(x)=|\zeta^{\prime}(x)|=\zeta^{\prime}(x), \qquad \forall~x\in(\bar{x},\bar{x}+\Delta_2]. 
		\end{equation}
		We then simplify $\xi(t)$. Observe that \eqref{eq:zeta1m} ensures that $\bar{x}$ is the unique stationary point in $[\bar{x},\bar{x}+\Delta_2]$. WLOG, we may assume that  $\Delta_1\leq\Delta_2$. It follows that $\zeta$ is strictly monotonically increasing on $\cU_{\Delta_1}=[\bar{x},\bar{x}+\Delta_1]$, and $[\bar{x},\bar{x}+\Delta_1]\cap\zeta^{-1}(t)$ is unique for $t\in[0,\zeta(\bar{x}+\Delta_1)]$. We obtain
		\begin{equation}\label{eq:zeta2}
			\xi(t)=m_{\zeta}(x)=\zeta^{\prime}(x) \text{ with }x=(\bar{x},\bar{x}+\Delta_1]\cap\zeta^{-1}(t), \qquad\forall~ t\in(0,\zeta(\bar{x}+\Delta_1)].
		\end{equation}
		Now, let us prove that either $\zeta$ has a K\L{} exponent of $0$ at $\bar{x}$ or \eqref{eq:zeta_estimate} holds.  
		
		When $\lim_{t\to0+}\xi(t)=l>0$, there is  $\bar{\alpha}\in(0,\zeta(\bar{x}+\Delta_1)]$ such that
		\[\xi(t)\geq\min\{l/2,1\},\qquad\forall~t\in(0,\bar{\alpha}],\]
		where $\xi(t)=\zeta^{\prime}(x)$ with $x=(\bar{x},\bar{x}+\Delta_1]\cap\zeta^{-1}(t)$ by \eqref{eq:zeta2}. This, together with the continuity and strict increasing of $\zeta$ on $(\bar{x},\bar{x}+\Delta_1]$, yields  \[\zeta^{\prime}(x)>\min\left\{l/2,1\right\},\qquad\forall~x\in(\bar{x},\zeta^{-1}(\bar{\alpha})],\]  
		which ensures that $\zeta$ has a K\L{} exponent of $0$ at $\bar{x}$. 
		
		When $\lim_{t\to0+}\xi(t)=0$, we have the estimate \eqref{eq:approxi}. Let $t=\zeta(x)$ for $x\in(\bar{x},\bar{x}+\Delta_1]$ in \eqref{eq:approxi}. By \eqref{eq:zeta2}, we obtain
		\[  \zeta^{\prime}(x)=c\cdot\zeta(x)^{\eta}+o(\zeta(x)^{\eta}). \]
		Clearly, one has $\eta\leq\theta<1$ since \eqref{eq:zeta} ensures $\zeta^{\prime}(x)=\Omega(\zeta(x)^{\theta})$. It follows that $\eta\in(0,1)$. Combined with the continuity of $\zeta$ and $\zeta(\bar{x})=0$, the above estimate implies \eqref{eq:zeta_estimate}. We complete the proof of Lemma \ref{le:exponent}.
	\end{proof}
	With Lemma \ref{le:exponent} in hand, we are ready to prove Proposition \ref{pro:estimate}. Consider $\zeta=\hat{\varphi}^2$ in Lemma \ref{le:exponent} and assume without loss of generality that $\bar{x}\in \bd(\dom(\zeta))=\bd(\dom(\varphi))$ is the left-end point of $\dom(\hat{\varphi})=\cl(\dom(\varphi))$. We see that either $\hat{\varphi}^2$ has the K\L{} exponent $0$ or there exist scalars $\Delta>0$, $\eta\in(0,1)$, such that
	\begin{equation}\label{eq:hatvarphi}
		\frac{\zeta^{\prime}(x)}{\zeta(x)^{\eta}}=\Theta\left(1\right) \qquad{\rm on }\quad x\in (\bar{x},\bar{x}+\Delta]\subseteq\inter(\dom(\varphi)).
	\end{equation}
	We first preclude that $\zeta=\hat{\varphi}^2$ has the K\L{} exponent $0$. If so, then there exists $c,\Delta_1>0$ such that $\zeta^{\prime}(x)\geq c$ for $x\in(\bar{x},\bar{x}+\Delta_1]$. It follows that
	\[ \zeta(x)=\hat{\varphi}^2(x)=\frac{1}{\varphi^{\prime\prime}(x)^2}\geq c(x-\bar{x}),\qquad\forall~x\in (\bar{x},\bar{x}+\Delta_1]. \]
	This implies $\varphi^{\prime\prime}(x)\leq 1/\sqrt{c(x-\bar{x})}$ for $x\in (\bar{x},\bar{x}+\Delta_1]$ and further
	\[\varphi^{\prime}(x)-\varphi^{\prime}(\bar{x}+\Delta_1)=-\int_x^{\bar{x}+\Delta_1}\varphi^{\prime\prime}(s){\rm d}s\geq -\int_x^{\bar{x}+\Delta_1} \frac1{\sqrt{c(s-\bar{x})}}{\rm d}s=\frac{2\left(\sqrt{\Delta_1}-\sqrt{x-\bar{x}} \right)}{\sqrt{c}} .\]
	On the other hand, we have $\varphi^{\prime}(x)\leq\varphi^{\prime}(\bar{x}+\Delta_1)$ by $\varphi^{\prime\prime}>0$ and $x\leq x+\Delta_1$. It follows that $\limsup_{x\to\bar{x}+}|\varphi^{\prime}(x)|<+\infty$, which contradicts the essential smoothness given in Definition \ref{def:kernel} (iii). We conclude that $\zeta$ cannot have the K\L{} exponent $0$ at $\bar{x}$. 
	
	Then, let us focus on the condition \eqref{eq:hatvarphi}, which implies
	\[  \  \frac{\left(\zeta^{1-\eta}\right)^{\prime}(x)}{1-\eta}=\frac{-2\varphi^{\prime\prime\prime}(x)}{\varphi^{\prime\prime}(x)^{3-2\eta}}=\Theta(1),\quad{\rm on }\quad (\bar{x},\bar{x}+\Delta]. \]
	To prove \eqref{eq:estimate}, it suffices to show that $\theta_{\bar{x}}:=3-2\eta\in (1,2]$ and $\varphi^{\prime\prime}(x)=\Theta(|x-\bar{x}|^{-\frac1{\theta_{\bar{x}}-1} })$. Indeed, from the above estimate and $\zeta(\bar{x})=\hat{\varphi}(\bar{x})^2=0$, we see that  
	\[  \frac{1}{\varphi^{\prime\prime}(x)^{2-2\eta} }=\zeta(x)^{1-\eta}=\int^{x}_{\bar{x}}\left(\zeta^{1-\eta}\right)^{\prime}(s){\rm d}s=\Theta\left(|x-\bar{x}|\right)\quad{\rm on }\quad (\bar{x},\bar{x}+\Delta],  \]
	which directly implies 
	\begin{equation}\label{eq:estimate_temp}
		\varphi^{\prime\prime}(x)=\Theta\left(\frac1{(x-\bar{x})^{\frac{1}{2-2\eta}} }\right)=\Theta\left(|x-\bar{x}|^{-\frac1{\theta_{\bar{x}}-1} }\right)\quad{\rm on }\quad (\bar{x},\bar{x}+\Delta].
	\end{equation}
	To estimate the range of $\theta_{\bar{x}}$, we observe that the essential smoothness of $\varphi$, i.e., $\lim_{x\to\bar{x}+}|\varphi^{\prime}(x)|=+\infty$, yields \[\lim_{x\to\bar{x}+}\int^{\bar{x}+\Delta}_x\varphi^{\prime\prime}(s){\rm d}s=\varphi^{\prime}(\bar{x}+\Delta)-\lim_{x\to\bar{x}+}\varphi^{\prime}(x)=+\infty,\] which, together with the estimate \eqref{eq:estimate_temp}, yields $\frac1{2-2\eta}\geq1$. Hence, we have $\eta\in[\frac12,1)$ and $\theta_{\bar{x}}=3-2\eta\in(1,2]$. Finally, we remark that \eqref{eq:estimate_temp} ensures the uniqueness of $\theta_{\bar{x}}$. We complete the proof.
	
	\section{Proof of Corollary \ref{co:closed}}
	To prove the first equivalence, we note that the closedness of $\dom(\varphi)$ is equivalent to $\bar{x}\in\dom(\varphi)$, i.e., $\varphi(\bar{x})<+\infty$, which amounts to  $\liminf_{x\to\bar{x}+}\varphi(x)<+\infty$ by the lower-semicontinuity and convexity of $\varphi$. Recall that Proposition \ref{pro:estimate} asserts that there exists $\Delta>0$ and  $\alpha_{\underline{x}}\in(1,2]$ such that
	\begin{equation}\label{eq:estimate1}
		\varphi^{\prime\prime}(x)=\Theta\left(|x-\bar{x}|^{-\frac1{\theta_{\bar{x}}-1} }\right)\quad {\rm on }\quad (\bar{x},\bar{x}+\Delta]. 
	\end{equation}
	Observe that for $x\in(\bar{x},\bar{x}+\Delta]$,
	\[\varphi(x)-\varphi(\bar{x}+\Delta)=\int_{x}^{\bar{x}+\Delta}-\varphi^{\prime}(s_1){\rm d}s_1=\int_{x}^{\bar{x}+\Delta}\left(\int_{s_1}^{\bar{x}+\Delta}\varphi^{\prime\prime}(s_2){\rm d}s_2-\varphi^{\prime}(\bar{x}+\Delta)\right) {\rm d}s_1. \]
	The equality says that $\liminf_{x\to\bar{x}+}\varphi(x)<+\infty$ is equivalent to
	\[\int_{\bar{x}}^{\bar{x}+\Delta}\int_{s_1}^{\bar{x}+\Delta}\varphi^{\prime\prime}(s_2){\rm d}s_2{\rm d}s_1<+\infty.  \]
	By \eqref{eq:estimate1}, the above integrability holds if and only if $\frac{1}{\theta_{\bar{x}}-1}<2$. As $\theta_{\bar{x}}\in(1,2]$, this is equivalent to $\theta_{\bar{x}}\in(\frac32,2]$.  Therefore, the first equivalence holds.
	
	We the consider the second equivalence. We assume $y\leq x$ without loss of generality. Observe that the continuity of $\varphi^{\prime\prime}$ on $\inter(\dom(\varphi))$ ensures $\sup_{s\in[x,y]}\varphi^{\prime\prime}(s)<+\infty$ and further 
	\begin{equation}\label{eq:boundint}
		0\leq\int^x_y\sqrt{\varphi^{\prime\prime}(s)}{\rm d}s<+\infty\qquad \text{ if}~x,y\in\inter(\dom(\varphi)).
	\end{equation}
	Hence, the equivalence trivially holds if $\dom(\varphi)=\R$. It suffices to consider the case where $\dom(\varphi)\subsetneqq\R$ with $\bd(\dom(\varphi))\neq\emptyset$.
	We only provide proof for the case where $\cl(\dom(\varphi))= [\bar{x},+\infty)$ since the proofs of other cases are nearly the same.

	On the other hand, by $y\leq x$ and $\varphi^{\prime\prime}>0$, we have
	\[ 0\leq\int^{x}_{y}\sqrt{\varphi^{\prime\prime}(s)}~{\rm d}s\leq \int^{x}_{\bar{x}}\sqrt{\varphi^{\prime\prime}(s)}~{\rm d}s=\int^{\bar{x}+\Delta}_{\bar{x}}\sqrt{\varphi^{\prime\prime}(s)}~{\rm d}s+\int_{\bar{x}+\Delta}^{x}\sqrt{\varphi^{\prime\prime}(s)}~{\rm d}s.\]
	Since $x,\bar{x}+\Delta\in\inter(\dom(\varphi))$, the second term in the right-hand side is bounded by \eqref{eq:boundint}. Then, $|\int^{x}_{y}\sqrt{\varphi^{\prime\prime}(s)}~{\rm d}s| <+\infty$ holds if and only if the first term is bounded, which, by \eqref{eq:estimate1}, is equivalent to
	\[\int_{\bar{x}}^{\bar{x}+\Delta}|s-\bar{x}|^{-\frac1{2(\theta_{\bar{x}}-1)} }{\rm d}s <+\infty.  \]
	The above integrability holds if and only if $\frac{1}{2(\theta_{\bar{x}}-1)}<1$. Recall $\theta_{\bar{x}}\in(1,2]$. This is equivalent to $\theta_{\bar{x}}\in(\frac32,2]$.  We conclude that  the second equivalence holds.
	We complete the proof.
	
	\section{Proof of Lemma \ref{le:limit}}
	Recall that $\partial F=\nabla f+\partial g$ by \cite[Exercise 8.8(c)]{rockafellar2009variational}, and the curve $\x$ satisfies the mirror descent dynamics \eqref{eq:mirrorflow}. We know that there is a curve $\p:\R_+\to\R^n$ with $\p(t)\in\partial g(\x(t))$ such that
	\begin{equation}\label{eq:dhdt}
		\frac{{\rm d} }{{\rm d}t} \nabla h\left(\x(t)\right)+\nabla f\left(\x(t)\right)+\p\left(\x(t)\right)=\bz\quad~{\rm a.e.~ on }~\R_+. 
	\end{equation}
	Integrate on $[0,T]$ for $T>0$ and divide $T$. We obtain
	\begin{equation}\label{eq:divide}
		\frac{\nabla h(\x(T))-\nabla h(\x(0))}{T}+\frac{1}{T}\int_{0}^T\nabla f(\x(t)){\rm d}t+\frac{1}{T}\int_{0}^T \p(\x(t)){\rm d}t=\bz,\quad\forall~T>0.
	\end{equation}
	Our goal is proving that the limit of \eqref{eq:divide} is the stationary condition of $f$ at $\bar{\x}$ as  $T\to+\infty$.
	\subsection{Step 1: Basic observations.}
	Lemma \ref{le:divide} ensures the limit of the first term in \eqref{eq:divide} (when exists) belongs to the normal cone $\cN_{\dom(h)}(\bar{\x})$.
	
	Then, we consider the limit of the second term $\frac{1}{T}\int_{0}^T\nabla f(\x(t)){\rm d}t$.
	Note that $\nabla f(\x(t))\to\nabla f(\bar{\x})$ by $\x(t)\to\bar{\x}$ and the continuous differentiability of $f$. We have $\max_{t\in\R_+}\|\nabla f(\x(t))-\nabla f(\bar{\x})\|_2\leq \beta_1$ for some $\beta_1>0$. Moreover, for any $\epsilon>0$, there exists $T_{\epsilon}>0$ such that
	\[\|\nabla f(\x(t))-\nabla f(\bar{\x})\|_2\leq \epsilon,\quad\forall~t\geq T_{\epsilon}.\] 
	It follows that for any $\epsilon>0$,
	\[\begin{aligned}
		\lim_{T\to\infty}\left\|\frac{1}{T}\int_{0}^T\nabla f(\x(t)){\rm d}t-\nabla f(\bar{\x}) \right\|_2
		&=\lim_{T\to\infty}\left\|\frac{1}{T}\int_{0}^T\nabla f(\x(t))-\nabla f(\bar{\x}){\rm d}t \right\|_2\\
		&\leq\lim_{T\to\infty}\frac{1}{T}\int_{0}^T\left\|\nabla f(\x(t))-\nabla f(\bar{\x})\right\|_2{\rm d}t \\
		&=\lim_{T\to\infty}\frac{1}{T}\left( \int_{0}^{T_{\epsilon}}+\int^T_{T_{\epsilon}}\right)\left\|\nabla f(\x(t))-\nabla f(\bar{\x})\right\|_2{\rm d}t \\
		&\leq \lim_{T\to\infty}\frac{T_{\epsilon}\beta_1}{T}+\lim_{T\to\infty}\frac{T-T_{\epsilon}}{T}\epsilon=\epsilon, 
	\end{aligned}\]
	which yields $\frac{1}{T}\int_{0}^T\nabla f(\x(t)){\rm d}t\to \nabla f(\bar{\x})$.
	
	Next, we consider the limit of $\frac{1}{T}\int_{0}^T \p(\x(t)){\rm d}t$.
	Let us define the interior and boundary index sets of $\bar{\x}$ by
	\[\cI\coloneqq\left\{i:\bar{x}_i\in\inter(\dom(\varphi)) \right\};\qquad \cJ\coloneqq\left\{j:\bar{x}_j\in\bd(\dom(\varphi)) \right\}. \]
	For $i\in\cI$, observe that \eqref{eq:divide} and the convergences $\varphi^{\prime}(x_i(t))\to \varphi^{\prime}(\bar{x}_i)$ and $\frac{1}{T}\int_{0}^T\nabla f(\x(t)){\rm d}t\to \nabla f(\bar{\x})$ imply
	\begin{equation}\label{eq:boundI}
		\begin{aligned}
			\lim_{T\to\infty}\frac1T\int_{0}^T p_i(\x(t)){\rm d}t&=	\lim_{T\to\infty}\frac{\varphi^{\prime}(x_i(0))- \varphi^{\prime}(x_i(T)) }{T}-	\lim_{T\to\infty}\frac{1}{T}\int_{0}^T\nabla_i f(\x(t)){\rm d}t\\
			&=-\nabla_if(\bar{\x}),\qquad\forall~i\in \cI.
		\end{aligned}
	\end{equation}
	For $j\in\cJ$, we need to show that $\frac1T\int_{0}^T p_j(\x(t)){\rm d}t$ is bounded on $\R_+$.
	\subsection{Step 2: Establish the boundedness of the mean subdifferential.}
	To show the boundedness of $\frac1T\int_{0}^T \p(\x(t)){\rm d}t$ on $\R_+$,
	our strategy is to estimate the mean inner product $\frac1{T}\int_{0}^{T} \left<\p(\x(t)),\bar{\x}-\z\right>{\rm d}t$ for $\z\in\dom(g)$ and yield a contradiction when the boundedness fails. Observe that for $\z\in\dom(g)$,
	\[\int_{0}^{T} \left<\p(\x(t)),\bar{\x}-\z\right>{\rm d}t=\int_{0}^{T} \left<\p(\x(t)),\x(t)-\z\right>{\rm d}t+\int_{0}^{T} \left<\p(\x(t)),\bar{\x}-\x(t)\right>{\rm d}t. \]
	We estimate the two terms in the right-hand side, respectively.
	
	\textbf{Step 2.1: Estimate of $\int_{0}^{T} \left<\p(\x(t)),\x(t)-\z\right>{\rm d}t$.} 
	
	Recall that $\x(t)\to\bar{\x}$ and $g$ is continuous on $\dom(g)$. We have $\min_{t\in\R_+}g(\x(t))\geq \beta_2$ for some $\beta_2>-\infty$. Moreover, for any $\epsilon>0$, there exits $t_{\epsilon}>0$ such that 
	\[  \|\bar{\x}-\x(t)\|_2\leq\epsilon;\qquad |g(\x(t))-g(\bar{\x})|\leq \epsilon,\quad \forall~ t\geq t_{\epsilon}.\]
	Combined with the convexity of $g$, these imply that for any $\epsilon>0$, $T\geq t_{\epsilon}$, and $\z\in\dom(g)$,
	\begin{equation}\label{eq:p}
		\begin{aligned}
			& \int_{0}^{T} \left<\p(\x(t)),\x(t)-\z\right>{\rm d}t\\
			\geq&  \int_{0}^{T} g(\x(t))-g(\z){\rm d}t\\
			=& \left(\int_{0}^{t_{\epsilon}}+\int_{t_{\epsilon}}^{T}\right) g(\x(t))-g(\z){\rm d}t 	\\
			\geq& {t_{\epsilon}}\left(\beta_2-g(\z)\right)+\left(T-t_{\epsilon}\right)\left( g(\bar{\x})-g(\z)-\epsilon \right).\\
		\end{aligned}
	\end{equation}
	
	\textbf{Step 2.2: Estimate of $\int_{0}^{T} \left<\p(\x(t)),\bar{\x}-\x(t)\right>{\rm d}t$.}
	
	Using \eqref{eq:dhdt} and integration by parts, we have
	\begin{equation*}
		\begin{aligned}
			&\int_{0}^{T} \left<\nabla f(\x(t))+\p(\x(t)),\bar{\x}-\x(t)\right>{\rm d}t\\
			=&  \int_{0}^{T}\left<\frac{{\rm d} }{{\rm d}t} \nabla h\left(\x(t)\right),\x(t)-\bar{\x}\right>{\rm d}t\\
			=&\left<\nabla h(\x(T)),\x(T)-\bar{\x}\right>-\left<\nabla h(\x(0)),\x(0)-\bar{\x}\right> -\int_{0}^{T}\left<\nabla h(\x(t)),\dot{\x}(t)\right>{\rm d}t.\\
		\end{aligned}
	\end{equation*}
	Note that the convexity of $h$ and direct integration imply
	\[ \langle\nabla h(\x(T)),\x(T)-\bar{\x}\rangle\geq h(\x(T))- h(\bar{\x});\]
	\[\int_{0}^{T}\left<\nabla h(\x(t)),\dot{\x}(t)\right>{\rm d}t=\int_{0}^{T}\frac{{\rm d}}{{\rm d}t}h(\x(t)){\rm d}t=h(\x(T))-h(\x(0)).\]
	We further have
	\begin{equation}\label{eq:inth}
		\begin{aligned}
			&\int_{0}^{T} \left<\nabla f(\x(t))+\p(\x(t)),\bar{\x}-\x(t)\right>{\rm d}t\\
			\geq&h(\x(T))-h(\bar{\x})-\left<\nabla h(\x(0)),\x(0)-\bar{\x}\right> +h(\x(0))-h(\x(T))\\ 
			= & -h(\bar{\x})-\left<\nabla h(\x(0)),\x(0)-\bar{\x}\right>+h(\x(0))\\
			=& -D_h(\bar{\x},\x(0)),
		\end{aligned}
	\end{equation}
	where  $h(\bar{\x}),D_h(\bar{\x},\x(0))<+\infty$ since $\dom(h)$ is closed and $\bar{\x}\in\dom(F)\subseteq\dom(h)$. 
	
	We then estimate $ \int_{0}^{T} \left<\nabla f(\x(t)),\bar{\x}-\x(t)\right>{\rm d}t$.
	By the convergence of $\x(t)$ and continuity of $\nabla f$, we know that $ \max_{t\in\R_+}\|\nabla f(\x(t))\|_2\cdot \max_{t\in\R_+}\|\bar{\x}-\x(t)\|_2\leq\beta_3$ for some $\beta_3>0$.
	Similar to \eqref{eq:p}, we have 
	\begin{equation}\label{eq:f1}
		\begin{aligned}
			& \int_{0}^{T} \left<\nabla f(\x(t)),\bar{\x}-\x(t)\right>{\rm d}t\\
			\leq& {\max_{t\in\R_+}\left\|\nabla f(\x(t))\right\|_2}\int_{0}^T\left\|\bar{\x}-\x(t)\right\|_2{\rm d}t\\
			=& {\max_{t\in\R_+}\left\|\nabla f(\x(t))\right\|_2}\left(\int_{0}^{t_{\epsilon}}+\int_{t_{\epsilon}}^T\right)\left\|\bar{\x}-\x(t)\right\|_2{\rm d}t\\
			\leq&{t_{\epsilon}\cdot\max_{t\in\R_+}\|\nabla f(\x(t))\|_2 \max_{t\in\R_+}\|\bar{\x}-\x(t)\|_2}+\left(T-t_{\epsilon}\right)\epsilon\\
			=& t_{\epsilon}\beta_3+\left(T-t_{\epsilon}\right)\epsilon.
		\end{aligned}
	\end{equation}
	Combining \eqref{eq:inth} and \eqref{eq:f1}, we obtain 
	\begin{equation}\label{eq:p1}
		\int_0^{T}\left<\p(\x(t)),\bar{\x}-\x(t)\right>{\rm d}t\geq -D_h(\bar{\x},\x(0))- {t_{\epsilon}\beta_3}- {\left(T-t_{\epsilon}\right)\epsilon}.
	\end{equation}
	Sum up \eqref{eq:p} and \eqref{eq:p1} and divide ${T}$ in both sides. 
	We see that for all $\epsilon>0$, $T\geq t_{\epsilon}$ and $\z\in\dom(g)$,
	\begin{equation}\label{eq:p3}
		\begin{aligned}
			&\frac1{T}\int_{0}^{T} \left<\p(\x(t)),\bar{\x}-\z\right>{\rm d}t\\
			\geq& \frac{t_{\epsilon}}{T}\left(\beta_2-g(\z)\right)+\frac{T-t_{\epsilon}}{T}\left( g(\bar{\x})-g(\z)-\epsilon \right)\\
			&- \frac{D_h(\bar{\x},\x(0))}T-\frac{t_{\epsilon}\beta_3}T-\frac{\left(T-t_{\epsilon}\right)\epsilon}T.
		\end{aligned}
	\end{equation}
	
	\textbf{Step 2.3: Find a contradiction.}
	
	Now, suppose that there is a sequence $\{t_k\}_{k\geq0}$ with $t_k\to+\infty$ such that 
	\[\lim_{k\to\infty}\left\|\frac1{t_k}\int_{0}^{t_k} \p(\x(t)){\rm d}t\right\|_2=+\infty.\]
	Moreover, by passing to a subsequence if necessary, suppose  \[\frac{\frac1{t_k}\int_{0}^{t_k} \p(\x(t)){\rm d}t}{\|\frac1{t_k}\int_{0}^{t_k} \p(\x(t)){\rm d}t\|_2}\to\p^*\in\R^n.\]
	In \eqref{eq:p3}, let $T=t_k$, divide $\|\frac1{t_k}\int_{0}^{t_k} \p(\x(t)){\rm d}t\|_2$ in both sides, and take $k\to+\infty$. As $\|\frac1{t_k}\int_{0}^{t_k} \p(\x(t)){\rm d}t\|_2\to+\infty$ ensure the right-side goes to zero, we obtain
	\[	\left<\p^*,\bar{\x}-\z\right>\geq 0,\qquad\forall~\z\in\dom(g).\]
	In particular, 
	\begin{equation}\label{eq:contrad}
		\left<\p^*,\bar{\x}-\x(0)\right>\geq 0.
	\end{equation}
	We then seek a contradiction to \eqref{eq:contrad}. Let $T=t_k$ in \eqref{eq:divide}, divide $\|\frac1{t_k}\int_{0}^{t_k} \p(\x(t)){\rm d}t\|_2$, and take $k\to\infty$. One has
	\begin{equation}\label{eq:p*}
		\begin{aligned}
			\p^*=\lim_{k\to\infty}\frac{\nabla h(\x(0))-\nabla h(\x(t_k))}{\left\|\int_{0}^{t_k} \p(\x(t)){\rm d}t\right\|_2 }.
		\end{aligned}
	\end{equation}
	For $\p^*_{\cI}$, the convergence $x_i(t_k)\to\bar{x}_i\in\inter(\dom(\varphi))$ yields the boundedness of $\{\varphi^{\prime}(x_i(t_k))\}_{k\geq0}$, $i\in\cI$. It follows from \eqref{eq:p*} and $\|\int_{0}^{t_k} \p(\x(t)){\rm d}t\|_2\to+\infty$ that $\p^*_{\cI}=\bz$. 
	
	For $\p^*_{\cJ}$,  the limit \eqref{eq:p*}, the convergence $\x(t_k)\to\bar{\x}$, and the convexity of $\varphi$ imply
	\[p^*_j\left(\bar{x}_j-x_j(0)\right)=\lim_{k\to\infty}\frac{\left(\varphi^{\prime}(x_j(0))-\varphi^{\prime}(x_j(t_k))\right)(x_j(t_k)-x_j(0))}{\left\|\int_{0}^{t_k} \p(\x(t)){\rm d}t\right\|_2 } \leq0,\quad\forall~j\in\cJ, \]  
	where the equality holds if and only if $p^*_j=0$ since $\bar{x}_j-x_j(0)\neq0$ for $j\in\cJ$.
	
	The above inequality, together with $\p^*_{\cJ}\neq\bz$ ($\|\p^*_{\cJ}\|_2=\|\p^*\|_2=1$), yields the strict inequality  $\langle\p^*_{\cJ},\bar{\x}_{\cJ}-\x_{\cJ}(0)\rangle<0$. This, together with $\p^*_{\cI}=\bz$, implies
	\[\left<\p^*,\bar{\x}-\x(0)\right>=\left<\p^*_{\cJ},\bar{\x}_{\cJ}-\x_{\cJ}(0)\right><0, \]
	which contradicts \eqref{eq:contrad}. We conclude that $\frac1T\int_{0}^T \p(\x(t)){\rm d}t$ is bounded on $\R_+$.
	
	\subsection{Step 3: Deduce the stationarity condition.}
	Due to the boundedness of $\frac1T\int_{0}^T \p(\x(t)){\rm d}t$ on $\R_+$, we may assume that \[\frac1{T_k}\int_{0}^{T_k} \p(\x(t)){\rm d}t\to\bar{\p}\in\R^n\]
	for some positive sequence $\{T_k\}_{k\geq0}$ with $T_k\to+\infty$.  Let $T=T_k$ in \eqref{eq:p3} and take  $k\to+\infty$. We see that for  any $\epsilon\in(0,r)$,
	\[\left<\bar{\p},\bar{\x}-\z\right>\geq g(\bar{\x})-g(\z)-2\epsilon,\qquad\forall~\z\in\dom(g). \]
	Let $\epsilon\to0$ and rearrange the inequality. We obtain
	\begin{equation}\label{eq:pf2}
		g(\z)- g(\bar{\x})\geq \left<\bar{\p},\z-\bar{\x}\right>,\qquad\forall~\z\in\dom(g).
	\end{equation}
	This says that $\bar{\p}\in\partial g(\bar{\x})$ for the convex function $g$.
	
	Now, let us deduce the limit of \eqref{eq:divide}. Let $T=T_k$ in \eqref{eq:divide} with $k\to\infty$, we have
	\[\bar\blam\coloneqq\lim_{k\to\infty}\frac{\nabla h(\x(T_k))-\nabla h(\x(0))}{T_k}=-\nabla f(\bar{\x})-\bar{\p},\]
	where the second equality is due to $\frac{1}{T}\int_{0}^T\nabla f(\x(t)){\rm d}t\to \nabla f(\bar{\x})$ and $\frac1{T_k}\int_{0}^{T_k} \p(\x(t)){\rm d}t\to\bar{\p}$. 
	
	Moreover, by Lemma \ref{le:divide}, $\x(T_k)\to\bar{\x}$, and $\nabla h(\x^0)/T_k\to\bz$, we have
	\[ \left<\bar{\blam},\z-\bar{\x} \right>\leq0,\quad\forall~\z\in\dom(h).\]
	This, together with \eqref{eq:pf2}, implies
	\[ g(\z)- g(\bar{\x})\geq \left<\bar{\p}+\bar{\blam},\z-\bar{\x}\right>,\qquad\forall~\z\in\dom(g). \]
	This says that $\bar{\p}+\bar{\blam}\in\partial g(\bar{\x})$. Recall $\bar\blam=-\nabla f(\bar{\x})-\bar{\p}$. We have
	\begin{equation}\label{eq:critical0}
		\bz=	\nabla f(\bar{\x})+\bar{\p}+\bar\blam\in\nabla f(\bar{\x})+\partial g(\bar{\x})= \partial F(\bar{\x}).
	\end{equation}
	That is, $\bar{\x}$ is a stationary point of $F$. We complete the proof.
	
	\bibliography{ref}
	\bibliographystyle{plainnat}
\end{document}